\documentclass[onefignum,onetabnum]{article}

\usepackage{arxiv}
\usepackage{amsmath, amsthm, amssymb,mleftright}
\usepackage{amsfonts}
 
\usepackage{parskip}
\usepackage{tikz,tikz-3dplot}
\usepackage{framed,multirow}
\usepackage[utf8]{inputenc} 
\usepackage[T1]{fontenc}    
\usepackage{hyperref}       
\usepackage{url}            
\usepackage{booktabs}       
\usepackage{amsfonts}       
\usepackage{nicefrac}       
\usepackage{microtype}      
\usepackage{graphicx}
\usepackage[numbers,sort]{natbib}
\usepackage{doi}
\usepackage{mathrsfs}
\usepackage{subcaption}
\usepackage{listings}
\usepackage{algorithm,algpseudocode}

\numberwithin{equation}{section}

\usepackage{url}
\usepackage{xcolor}
\definecolor{newcolor}{rgb}{.8,.349,.1}

\newtheorem{prop}{Proposition}
\newtheorem{theorem}{Theorem}
\newtheorem{definition}{Definition}[section] 
\newtheorem{remarknn}{Remark}
\title{Factorized structure of  the long-range two-electron integrals tensor  and its application in quantum chemistry}

\author{ {\hspace{1mm}Siwar Badreddine}\thanks{Laboratoire Jacques-Louis Lions, Sorbonne Université, INRIA Alpines, (siwar.badreddine@inria.fr).}  
	\And
	 { \hspace{1mm} Igor Chollet}\thanks{Laboratoire Analyse, Géométrie et Applications,
(igor.chollet@inria.fr).}    
\And
	 { \hspace{1mm} Laura Grigori}\thanks{Laboratoire Jacques-Louis Lions, Sorbonne Université, INRIA Alpines,
(laura.grigori@inria.fr).}    
}

\hypersetup{
pdftitle={A template for the arxiv style},
pdfsubject={q-bio.NC, q-bio.QM},
pdfauthor={David S.~Hippocampus, Elias D.~Striatum},
pdfkeywords={First keyword, Second keyword, More},
}

\begin{document}
\def\lrttei{TA}
\def\fmm{LTEI-FMM}
\renewcommand{\algorithmicrequire}{\textbf{Input:}}
\renewcommand{\algorithmicensure}{\textbf{Output:}}
\maketitle

\begin{abstract}
We introduce two new approximation methods for the  numerical evaluation of the long-range Coulomb potential and  the  approximation of the resulting high dimensional Two-Electron Integrals tensor (TEI)  with  long-range interactions arising in molecular simulations. The first method exploits the tensorized structure of the compressed two-electron integrals obtained through two-dimensional Chebyshev interpolation combined with Gaussian quadrature. The  second method is based on  the Fast Multipole Method (FMM). Numerical experiments for different medium size  molecules on high quality basis sets outline the efficiency of the two methods. Detailed algorithmic is provided in this paper as well as numerical comparison of the introduced approaches.
\end{abstract}


\section{Introduction}
 In this paper we  are interested in  the numerical evaluation of the  long-range Coulomb interaction  and the   approximation of the resulting Two-Electron Integrals (TEI) tensor. The evaluation of the two-electron integrals is considered as a challenging problem in quantum chemistry. These integrals are  essential  to approximate the solution of the so-known Schrodinger equation for a general N-body system \cite{box} arising   in electronic and molecular structure calculations. This equation  describes the state function of a quantum-mechanical system which is given  in the time-independant form  as follows \cite{CANCES20033},
\begin{equation}
\label{eq::Schro}
     \mathbf{H} \mathbf{\psi}=E \mathbf{\psi},
\end{equation}
where $\mathbf{H}$ is the Hamiltonian operator that can be described by the sum of three terms: the kinetic energy, the Coulomb interaction between electrons and nuclei, and the electron-electron Coulomb repulsion  \cite{hamiltonian,quatummech}, $\mathbf{\psi}$ is the wave-function or the state-function, and $E$ is the full energy of the system. Under the Born-Oppenheimer approximation, i.e the motion of atomic nuclei and electrons can be treated separately given that the nuclei are much heavier than the electrons \cite{hamiltonian,quatummech}, finding an exact, analytic solution of the Schrodinger
equation becomes intractable for systems with more than one electron \cite{hamiltonian}. Therefore,  additional assumptions are considered such as the Hartree-fock strategy  and  the Galerkin approximation procedure \cite{quatummech}. These assumptions yield to include the evaluation of the two-electron integrals such that,  given the finite basis set $\left\{g_{\mu}\right\}_{1 \leq \mu \leq N_{b}}$, $g_{\mu} \in H^{1}(\mathbb{R}^{3})$, these integrals are defined by  \cite{khorom}
\begin{equation}
       \mathcal{B}(\mu,\nu,\kappa,\lambda)=\int_{ \mathbb{R}^{3}}\int_{ \mathbb{R}^{3}}\frac{g_{\mu}(\boldsymbol{x}) g_{\nu}(\boldsymbol{x})g_{\kappa}(\boldsymbol{y})g_{ \lambda }(\boldsymbol{y})}{ \left \| \boldsymbol{x}-\boldsymbol{y}    \right \|}d\boldsymbol{x}d\boldsymbol{y},~with ~\mu,\nu,\kappa,\lambda \in \left \{ 1,..,N_{b} \right \}.
\end{equation}
These six-dimensional integrals are the entries of a fourth-order tensor, referred to as $\mathcal{B}$, with $O(N_b^4)$ entries with $N_b$  being the number of basis functions $\left\{g_{\mu}\right\}_{1 \leq \mu \leq N_{b}}$. Many works exist in the literature for the analytic evaluation of these integrals using certain types of basis functions, mainly Slater type functions and Gaussian-type functions \cite{theo1,theo2}.
Considerable efforts have been devoted to minimize the cost of the integrals evaluation which is a challenging computational problem since it requires the evaluation $N_b^4$ six-dimensional  integrals that are singular due to the presence of the Coulomb potential $\frac{1}{ \left \| \boldsymbol{x}-\boldsymbol{y}\right \|}$  and where $N_b$ increases drastically with the molecular system size. An alternative approach to tackle this problem  is to develop methods dealing with smooth potential. We consider in our work an approach that relies on  the range-seperation of the Coulomb potential \cite{separation,Sepemma,Sepsavin,theo1,short,FMM,portee,Lecours,Resolutions,unknown}  where  the last is split into a smooth range-part and a complementary diverging part. The splitting is done through the function $erf(\omega \left \| \boldsymbol{x}-\boldsymbol{y}    \right \|) $ with $\omega$ being the range separator parameter. This separation writes 
\begin{equation}
\label{eq_erf_erfc}
    \frac{1}{ \left \| \boldsymbol{x}-\boldsymbol{y}    \right \|} = \frac{erfc\left(\omega  \left \| \boldsymbol{x}-\boldsymbol{y}    \right \|\right)}{ \left \| \boldsymbol{x}-\boldsymbol{y}    \right \|}+\frac{erf\left(\omega  \left \| \boldsymbol{x}-\boldsymbol{y}    \right \|\right)}{ \left \| \boldsymbol{x}-\boldsymbol{y}    \right \|}, 0 \leq \omega <  \infty,  
\end{equation}
with 
\begin{equation}
    \label{eq::long-r}
    \frac{erf\left(\omega  \left \| \boldsymbol{x}-\boldsymbol{y}    \right \|\right)}{ \left \| \boldsymbol{x}-\boldsymbol{y}    \right \|}
    =\frac{2}{ \left \| \boldsymbol{x}-\boldsymbol{y}    \right \|\sqrt\pi}\int_{   0}^{ \omega \left \| \boldsymbol{x}-\boldsymbol{y}    \right \|}exp(-t^{2})dt, \text{ and } erfc\left(\omega  \left \| \boldsymbol{x}-\boldsymbol{y}    \right \|\right)=1-erf\left(\omega  \left \| \boldsymbol{x}-\boldsymbol{y}    \right \|\right),
\end{equation}
where $\boldsymbol{x}=(\boldsymbol{x}_1,\boldsymbol{x}_2,\boldsymbol{x}_3), \boldsymbol{y}=(\boldsymbol{y}_1,\boldsymbol{y}_2,\boldsymbol{y}_3) \in \mathbb{R}^{3}$,  $\omega$ is a positive parameter that controls the separation range. The long-range contribution in equation  \eqref{eq::long-r} is  a smooth function  such that, for small $\omega$, the singularity is eliminated at  $ \left \| \boldsymbol{x}-\boldsymbol{y}    \right \|=0$. When $\omega=0$, the long-range part vanishes and when $\omega \rightarrow \infty$, it approaches the Coulomb potential $\frac{1}{\left \| \boldsymbol{x}-\boldsymbol{y}    \right \| }$.  The  short range contribution (the complementary function in equation \eqref{eq_erf_erfc})  has singularity at $ \left \| \boldsymbol{x}-\boldsymbol{y}    \right \|=0$.  The long-range part, that we denote $K(\boldsymbol{x},\boldsymbol{y})= \frac{erf(\omega \left \| \boldsymbol{x}-\boldsymbol{y}    \right \| ))}{ \left \| \boldsymbol{x}-\boldsymbol{y}    \right \| }, \boldsymbol{x},\boldsymbol{y} \in \mathbb{R}^3$,  is treated usually through employing numerical integration in Fourier space \cite{PMID:34686052}. Following equation \eqref{eq_erf_erfc}, the two-electron integrals tensor can be expressed as the sum of two terms
\begin{small}
\begin{eqnarray}
    \mathcal{B}(\mu,\nu,\kappa,\lambda) &=& \underbrace{\int_{\mathbb{R}^{3}}\int_{\mathbb{R}^{3}}\frac{erf\left(\omega  \left \| \boldsymbol{x}-\boldsymbol{y}    \right \| \right) g_{\mu}(\boldsymbol{x}) g_{\nu}(\boldsymbol{x})g_{\kappa}(\boldsymbol{y})g_{ \lambda }(\boldsymbol{y})}{ \left \| \boldsymbol{x}-\boldsymbol{y}    \right \|}d\boldsymbol{x}d\boldsymbol{y}}_{\mathcal{B}^{lr}(\mu,\nu,\kappa,\lambda)} \\
&+&\underbrace{\int_{\mathbb{R}^{3}}\int_{\mathbb{R}^{3}}\frac{erfc\left(\omega  \left \| \boldsymbol{x}-\boldsymbol{y}    \right \| \right) g_{\mu}(\boldsymbol{x}) g_{\nu}(\boldsymbol{x})g_{\kappa}(\boldsymbol{y})g_{ \lambda }(\boldsymbol{y})}{ \left \| \boldsymbol{x}-\boldsymbol{y}    \right \|}d\boldsymbol{x}d\boldsymbol{y}}_{\mathcal{B}^{sr}(\mu,\nu,\kappa,\lambda)},
\end{eqnarray}
\end{small}
with ~$\mu,\nu,\kappa,\lambda \in \left \{ 1,..,N_{b} \right \},$ $\mathcal{B}^{lr}$ is referring to the long-range  two-electron integrals tensor and $\mathcal{B}^{sr}$ is referring to the short-range two-electron integrals tensor. In this paper, we focus on the  numerical evaluation of the long-range kernel $K(\boldsymbol{x},\boldsymbol{y})$ and on the approximation of the long-range two-electron integrals given by 
\begin{equation}
\label{long-range}
    \mathcal{B}^{lr}(\mu,\nu,\kappa,\lambda)=\int_{\mathbb{R}^{3}}\int_{\mathbb{R}^{3}}g_{\mu}(\boldsymbol{x}) g_{\nu}(\boldsymbol{x})K(\boldsymbol{x},\boldsymbol{y}) g_{\kappa}(\boldsymbol{y})g_{ \lambda }(\boldsymbol{y}) d\boldsymbol{x}d\boldsymbol{y}, \mu,\nu,\kappa,\lambda \in  \left \{ 1,\ldots,N_{b} \right \}.
\end{equation}
We use finite linear combinations of primitive Gaussians as basis functions $\left\{g_{\mu}\right\}_{1 \leq \mu \leq N_{b}}$. Such basis functions are expressed as   linear combinations of $I_\mu$ primitive Gaussians functions \cite{set} 
\begin{equation}\label{basis}
g_{\mu}(\boldsymbol{x})
= \sum_{j=1}^{I_\mu} c_{j} \prod_{l=1}^{3}g_{\mu}^{(j)}(\boldsymbol{x}_l),\boldsymbol{x}_l \in \mathbb{R}, I_\mu \in \mathbb{N},
\end{equation}
where  primitive Gaussians are defined by
\begin{equation}
    \label{eq::primi}
    g_{\mu}^{(j)}(\boldsymbol{x}_l)=(\boldsymbol{x}_l-\boldsymbol{r}_{l})^{p_{\mu_l}}exp\left (-\mu_{j}   (\boldsymbol{x}_l-\boldsymbol{r}_l)^{2} \right ),\boldsymbol{x}_l \in \mathbb{R},\mu \in \left \{ 1..N_b \right \},
\end{equation}
where  $N_b$ is the number of basis functions defined in \eqref{basis}, $c_{j}$ refers to a normalization constant, $\mu_j$ is a parameter whose reference value is, for instance, given in \cite{basis}, $\boldsymbol{r}_l$ refers to the coordinates of atom nucleus that is known in practice, and the $p_{\mu_l}$ are exponents depending  on the chosen basis function. These basis functions $\left\{g_{\mu}\right\}_{1 \leq \mu \leq N_{b}}$ correspond in chemistry to approximations of the atomic orbitals. In addition, we consider restrictions of these basis functions to sufficiently large  compact support $[-b,b]^3 \subset \mathbb{R}^{3}$  such that we have 
\begin{equation}
\label{eq::compact}
    \mathcal{B}^{lr}(\mu,\nu,\kappa,\lambda)=\int_{[-b,b]^{3}}\int_{[-b,b]^{3}}g_{\mu}(\boldsymbol{x}) g_{\nu}(\boldsymbol{x})K(\boldsymbol{x},\boldsymbol{y}) g_{\kappa}(\boldsymbol{y})g_{ \lambda }(\boldsymbol{y}) d\boldsymbol{x}d\boldsymbol{y}, \mu,\nu,\kappa,\lambda \in  \left \{ 1,\ldots,N_{b} \right \}.
\end{equation}
We note that the range-separation representation of the Coulomb potential is important in molecular simulations to describe non-local correlation effects and to allow an accurate evaluation of the  long-range two-electron integrals while keeping the computational cost reasonably low \cite{separation}.

In this work, we introduce two numerical approaches  for the numerical evaluation of  the smooth long-range interaction and the approximation of  the long-range two electron integrals tensor. First, instead of performing  a naive numerical computation of $K(\boldsymbol{x},\boldsymbol{y})$ over $N\times N \times N$ 3D Cartesian grids, we consider two-dimensional Chebyshev interpolation method using only $N^{\frac{1}{3}} \times N^{\frac{1}{3}}$ isotropic Chebyshev  grids combined with Gaussian-quadrature rule in order to approximate $K(\boldsymbol{x},\boldsymbol{y})$. We refer to this approach as {\lrttei} for Tensorized Approximation and we denote the approximation method for the evaluation of the long-range two-electron integrals by LTEI-{\lrttei}. This numerical approximation yields to a tensorized expression of the six-dimensional integrals with $erf$-interaction leading to substantial time complexity reduction to evaluate one integral of the form $\mathcal{B}^{lr(\mu,\nu,\kappa,\lambda)}$. However, these six-dimensional integrals are only one element from the fourth-order tensor $\mathcal{B}^{lr}$. This means that for a basis set consisting of $N_b$ basis functions, there are $\mathcal{O}(N_b^4)$ integrals to evaluate. Therefore, we introduce, using {LTEI-\lrttei} approach, a new  alternative way to approximate these integrals by means of a factorized representation of the fourth-order tensor $\mathcal{B}^{lr}\in$ $\mathbb{R}^{N_{b}\times N_{b}\times N_{b}\times N_{b}}$, leading to an efficient application of the matricization of $\mathcal{B}^{lr}$  to a vector with a significant reduction in time complexity to $\mathcal{O}(\epsilon N^{4/3}), \epsilon \ll  N$ instead of $\mathcal{O}\left(N^{2}\right)$ given a naive computation. These complexities may be further reduced due to properties of Gaussian-type functions. Hence, we propose to express the high dimensional fourth-order tensor $\mathcal{B}^{lr}$ in a more compressed format by using screening  techniques and low-rank approximation methods. Second, we consider Chebyshev interpolation combined with Fast Multipole Method (FMM) \cite{kernel,GREENGARD1987325} leading to linear time complexity  when computing the FMM-accelerated matrix vector product involving the two-electron integrals tensor. This method is referred to as {\fmm}. We provide detailed comparison between the two approaches and discuss to what extent the relative performances  of these methods make them attractive for different application cases. In order to test  the performance of our algorithm, we use the data sets of molecular properties calculated from quantum chemistry for some moderate size molecules. These data sets are extracted from \textit{quantum package} \cite{quan_pack}.

The paper is organized as follows. In Section~\ref{sec::notations} we introduce notations, problem definitions, and properties. In Section~\ref{sec::LRTTEI}, we describe  our new tensorized method to approximate $K(\boldsymbol{x},\boldsymbol{y})$ and we present our LTEI-{\lrttei} scheme for the element-wise evaluation of the two-electron integrals based on the underlying tensorized structure. We describe also using  LTEI-{\lrttei} a factorized expression of the two-electron integrals tensor and we derive  error bounds and theoretical complexities for the approximation process we use. In Section~\ref{subsub::fmm} we demonstrate  that our kernel $K(\boldsymbol{x},\boldsymbol{y})$ is asymptotically smooth, so that we can benefit from fast hierarchical methods (especially Fast Multipole Methods) in order to efficiently evaluate the two-electron integrals decompositions. Hence, we reformulate  these decompositions as $N$-body problems on non-uniform particle distributions. In Section~\ref{sec::application}, we propose an application case in electronic calculations by using the decompositions of the two-electron integrals tensor obtained through the new introduced approaches. In Section~\ref{sub::LR}, further compression techniques are also presented, extending screening approaches and low rank approximation methods to our new decompositions. Finally, results of numerical tests of both methods are  presented as well as a summary of our findings. We use Julia open-source language to test   the new approximation method {\lrttei} and the evaluation scheme LTEI-{\lrttei}\footnote{https://github.com/sbadred/LTEI\_TA.jl.git} and the C++ library $defmm$\footnote{https://github.com/IChollet/defmm}for {\fmm}.

\section{Preliminaries}
\label{sec::notations}
This section introduces our notations  as well as several definitions and properties that will be used in the paper. The matrix operations notations are defined in Section \ref{sec::def}.
\subsection{Notations} 
We use the following notations:
\begin{itemize}
    \item $\mathcal{B} \in \mathbb{R}^{I_{1} \times I_{2} \times I_{3 }  \times I_{4}}$ is a  fourth-order tensor with modes $I_{1}, I_{2}, I_{3}, I_{4}$.
    \item $\mathbf{B}_{(j)} \in \mathbb{R}^{I_{j} \times I_{1} I_{2} \cdots I_{j-1} I_{j+1} \cdots I_{4}}$ is the mode-j matricization of  $\mathcal{B}$ and $\mathbf{B}_{(1,\ldots,j)} \in \mathbb{R}^{I_{1} I_{2} \cdots I_{j} \times I_{j+1} \cdots I_{4}}$ is the mode-$(1,\ldots,j)$ matricization of  $\mathcal{B},$ with $j \in  \left \{ 1,2,3,4 \right \}$ (see \eqref{eq::mode} for more details).
    \item $\text { Scalars are either lowercase letters } x, y, z, \alpha, \beta, \gamma \text { or uppercase Latin letters } N, M, T$. Vectors $\boldsymbol{b}$ are denoted by lowercase boldface letters, matrices $\mathbf{B}$ are denoted by uppercase boldface letters.
    \item  $\boldsymbol{b}_i$ is the  $i-th$ element of the vector $\boldsymbol{b}$, $\mathbf{B}(i_1,i_2)$ is the $(i_1,i_2)$th entry of the matrix $\mathbf{B}$, $\mathcal{B}(i_1,i_2,i_3,i_4)$ is the  $(i_1,i_2,i_3,i_4)$th entry of the tensor $\mathcal{B} \in \mathbb{R}^{I_{1} \times I_{2} \times I_{3 }  \times I_{4}} $.
    \item $\mathbf{B}[:,j]$ (Julia/Matlab notations) denotes the subvector containing the column of $\mathbf{B}$ indexed by $j$, $\mathbf{B}[j,:]$ (Julia notation) denotes the subvector containing the row of $\mathbf{B}$ indexed by $j$, $\mathcal{B}[:,:,j]$  denotes the submatrix extracted from $\mathcal{B}$ at index $j$, $\mathcal{B}[:,:,:,j]$  denotes the subtensor extracted from $\mathcal{B}$ at index $j$.
    \item $\left \|   \boldsymbol{x}-\boldsymbol{y}\right \|=\sqrt{(\boldsymbol{x}_1-\boldsymbol{x}_1)^2+(\boldsymbol{x}_2-\boldsymbol{y}_2)^2+(\boldsymbol{x}_3-\boldsymbol{y}_3)^2}$ is the euclidean distance between two points $\boldsymbol{x}, \boldsymbol{y} \in \mathbb{R}^3$ with coordinates $(\boldsymbol{x}_1,\boldsymbol{x}_2,\boldsymbol{x}_3), (\boldsymbol{y}_1,\boldsymbol{y}_2,\boldsymbol{y}_3)$ respectively.
    \item $|x|$ is the absolute value of $x$.
    \item $ \otimes$ is the kronecker product, $\odot$ is the Hadamard product, $\diamond$ is the row-wise Khatri-rao product, and $*$ is the column-wise Khatri-rao product.
    \item $\left \|f  \right \|_{\infty,\mathbf{S}} :=\sup \{|f(s)|: s \in S\}.$
\end{itemize}
\subsection{Definitions and properties} 
\label{sec::def}
We give in the following several definitions and properties that we use in the subsequent sections.
In the different approximations derived in this paper, the product of two Gaussian type functions is often used. Therefore,  we recall the general product rule between two Gaussian functions.
\begin{prop}[\cite{set}]
\label{def::Gaussrule}
Let $g_1(\boldsymbol{x})=exp\left(-c_1\left |\boldsymbol{x}-\boldsymbol{r}  \right |^2 \right), g_2(\boldsymbol{x})=exp\left(-c_2\left |\boldsymbol{x}-\tilde{\boldsymbol{r}}  \right |^2  \right)$ be Gaussian functions with $\boldsymbol{x},\boldsymbol{r},\tilde{\boldsymbol{r}} \in  \mathbb{R}^3, c_1,c_2 \in  \mathbb{R}$. The product of these functions is
\begin{equation}
\label{eq::productrule}
    g_{12}(\boldsymbol{x})= g_{1}(\boldsymbol{x}) g_{2}(\boldsymbol{x})=exp\left(\frac{-c_1c_2}{c_1+c_2} \left | \boldsymbol{x}-\boldsymbol{r_{12}}\right |^2\right)exp\left(-(c_1+c_2)\left | \boldsymbol{x}-\boldsymbol{r_{12}} \right |^2\right),
\end{equation}
where $\boldsymbol{r_{12}}=\frac{c_1}{c_1+c_2}\boldsymbol{r}+ \frac{c_2}{c_1+c_2}\tilde{\boldsymbol{r}}$.
\end{prop}
We also have recourse to two-dimensional Chebyshev interpolation. Therefore, we give the expressions of the  Chebyechev polynomials as well as the Chebyshev coeffcients.
\begin{definition}[Two dimensional Chebyshev interpolation \cite{scheiber2015chebyshev,chebfun2}]
\label{def::Chebyspoints}
 For a given continuous function f(x,y) on $[a,b]^2, a,b \in  \mathbb{R}$, the two-dimensional Chebyshev interpolation of this function is given by  its interpolating polynomial that we denote
\begin{equation}
\label{eq::chebpoly}
    \Tilde{f}(x,y)=\sum_{n,m=0}^{N}\alpha_{nm} T_{n}(x) T_{m}(y),
\end{equation}
where $N$ is the number of interpolation nodes, $T_{n}(x)=cos(n~acos(x)), x \in [a,b], n \in  \left \{ 1,\dots,N\right \}$ are the Chebyshev polynomials,
\begin{equation}
\label{eq::chebcoeff}
    \alpha _{nm}=\frac{c_{nm}}{N^2}\sum_{k,k'=1}^{N} f(x_k,y_{k}')T_{n}(x_{k})T_{m}(y_{k}'),\hspace{0.4cm} c_{n,m} = \begin{cases}
    1 & \text{ if }m=n=0\\
    2 & \text{ if }m\neq n=0 \text{ or }n\neq m=0\\
    4 & \text{ if }m\neq 0, n\neq 0\\
    \end{cases}
\end{equation}
are Chebyshev interpolation coefficients. The nodes $x_{k},y_{k}'$ form the Chebyshev two-dimensional grids such as  Chebyshev-Gauss points (first kind)
 \begin{equation}
     x_{k}=\cos \theta_{k}, \hspace{0.4cm}\theta_{k}=\frac{(2 k-1) \pi}{2 N}, \quad k=1, \ldots, N,
 \end{equation}
 or Chebyshev-Lobatto points (second kind)
 \begin{equation}
     x_{k}=\cos \phi_{k}, \hspace{0.4cm}\phi_{k}=\frac{(k-1) \pi}{N-1}, \quad k=1,\ldots, N.
 \end{equation}
\end{definition}
The following proposition gives the interpolation error of the two-dimensional Chebyshev approximation.
\begin{prop}[Interpolation error \cite{err2}]
\label{def::Chebyupperbound}
Let $\Tilde{f}(x,y)$ be an interpolating polynomial of $f(x,y)$ on  $[a,b]^2$ at Chebyshev N interpolation nodes and suppose that  the  partial derivatives  $\partial^{N+1} f(x, y) / \partial x^{N +1}$  and  $\partial^{N +1} f(x, y) / \partial y^{N +1}$  exist and are continuous for all $(\mathrm{x}, \mathrm{y}) \in[a,b]^2$. We have  
\begin{equation}
|f(x, y)-\Tilde{f}(x, y)| \leq \frac{\left (\frac{b-a}{2}  \right )^{N+1}}{2^{N}\left(N+1\right) !}c_1+ \frac{\delta \left (\frac{b-a}{2}  \right )^{N+1}}{2^{N}\left(\mathrm{N}+1\right) !}c_2,
\end{equation}
where 
\begin{eqnarray}
     c_1 =\max _{\xi \in[a,b]}\left|\frac{\partial^{N+1} f(\xi,y)}{\partial \xi^{N+1}}\right|, c_2=\max _{(\xi, \eta) \in[a,b]^2 } \left|\frac{\partial^{\mathrm{N}+1} f(\xi, \eta)}{\partial \eta^{\mathrm{N}+1}}\right|, \text{ and } 
     \delta =\max _{s \in[a,b]}\sum_{i=0}^{N}\left|L_{i,N}(s)\right|.
\end{eqnarray}
The so-called Lebesgue constant $\delta$ grows only logarithmically if Chebyshev interpolation nodes are used, $L_{i,N}(s)$ are Lagrange polynomials of degree $N$.
\end{prop}
 The following proposition recalls the upper bound of Gaussian-quadrature rule error.
\begin{prop}[Quadrature error, Section~5.2 \cite{err1}]
\label{def::gauss}
 Let $\left [ a,b \right ]$ be a real closed interval of length $ \left | b-a \right |>0$ and let $f$ $\in$ $\mathbb{C}^{2N_q}(\left [ a,b \right ]), N_q \geq 1$, the integration of $f$ over $\left [ a,b \right ]$ can be given as follows, using  Gaussian quadrature rule
\begin{equation}
    \int_{\left [ a,b \right ]} f(x) d x=\int_{-1}^{1} f\left(\frac{b-a}{2} z+\frac{a+b}{2}\right) \frac{d x}{d z} d z=\frac{b-a}{2}\sum_{i=1}^{N_q} w_{i} f\left(\frac{b-a}{2} z_{i}+\frac{a+b}{2}\right)+R_{N_q},
\end{equation}
where $w_i$ and $x_i$ are the weights and nodes of the quadrature rule, $N_q$ is the number of quadrature points and $R_{N_q}$ refers to the Gaussian quadrature error. This last quantity verifies
\begin{equation}
    |R_{N_q}|  \leq \frac{\left | b-a \right |^{2 N_q+1}(N_q !)^{4}}{(2 N_q+1)[(2 N_q) !]^{3}}\left \|  \frac{d^{2N_q}}{d s^{2N_q}}f(s) \right \|_{\infty,\left [ a,b \right ]}.
\end{equation}
\end{prop}
We recall now several matrix products that are used in this paper. The \textit{Hadamard product} between matrices $\mathbf{A}$ and $\mathbf{B} \in \mathbb{R}^{I \times J}$ is $\mathbf{A} \odot \mathbf{B} \in \mathbb{R}^{I \times J}$ defined as
\begin{equation}
     \mathbf{A} \odot \mathbf{B}=\left[\begin{array}{cccc}
a_{11} b_{11} & a_{12} b_{12} & \cdots & a_{1 J} b_{1 J} \\
a_{21} b_{21} & a_{22} b_{22} & \cdots & a_{2 J} b_{2 J} \\
\vdots & \vdots & \ddots & \vdots \\
a_{I 1} b_{I 1} & a_{I  2} b_{I  2} & \cdots & a_{I  J } b_{I  J }
\end{array}\right].
\end{equation}
The \textit{Kronecker product} of matrices $\mathbf{A} \in \mathbb{R}^{I_1 \times J_1}$ and $\mathbf{B}  \in \mathbb{R}^{I_2 \times J_2}$ is $\mathbf{A} \otimes \mathbf{B} \in \mathbb{R}^{I_1I_2\times J_1J_2}$ defined as 
\begin{equation}
   \mathbf{A} \otimes \mathbf{B}=\left[\begin{array}{cccc}
a_{11} \mathbf{B} & a_{12} \mathbf{B} & \cdots & a_{1 J_1} \mathbf{B} \\
a_{21} \mathbf{B} & a_{22} \mathbf{B} & \cdots & a_{2 J_1} \mathbf{B} \\
\vdots & \vdots & \ddots & \vdots \\
a_{I_1 1} \mathbf{B} & a_{I_1 2} \mathbf{B} & \cdots & a_{I_1 J_1} \mathbf{B}
\end{array}\right]
\end{equation}
We also use the compact product notation $\otimes_{k=1}^{d} $. Given $d$ matrices $\mathbf{A}_k  \in \mathbb{R}^{I_k\times J_k}, k \in  \left \{ 1,\ldots,d \right \}$, we have
\begin{equation}
\otimes_{k=1}^{d}\mathbf{A}_k=\mathbf{A}_1 \otimes \mathbf{A}_2 \otimes \ldots \otimes \mathbf{A}_d \in \mathbb{R}^{\prod_{k=1}^d I_k\times \prod_{k=1}^d J_k}.
\end{equation}
Consider two matrices $\mathbf{A}=\begin{bmatrix}
\mathbf{A}[1,:]^\top&
\mathbf{A}[2,:]^\top &
\hdots &
\mathbf{A}[I_1,:]^\top
\end{bmatrix}^\top \in \mathbb{R}^{I_1 \times J_1}$
and $\mathbf{B}=\begin{bmatrix}
\mathbf{B}[1,:]^\top&
\mathbf{B}[2,:]^\top& 
\hdots &
\mathbf{B}[I_1,:]^\top
\end{bmatrix}^\top  \in \mathbb{R}^{I_1\times J_2}$, where $ \mathbf{A}[k,:]  \in \mathbb{R}^{1 \times J_1}$ and $ \mathbf{B}[k,:]  \in \mathbb{R}^{1 \times J_2}$ for $k \in \left \{ 1,...,I_1 \right \}$ . The row-wise Khatri-Rao product $\mathbf{A} \diamond \mathbf{B}$ is a matrix of dimension $I_1 \times (J_1J_2)$ defined as
\begin{equation}
\mathbf{A} \diamond \mathbf{B}=\begin{bmatrix}
\mathbf{A}[1,:]^\top \otimes \mathbf{B}[1,:]^\top &
\mathbf{A}[2,:]^\top  \otimes \mathbf{B}[2,:]^\top&
\hdots &
\mathbf{A}[I_1,:]^\top  \otimes \mathbf{B}[I_1,:]^\top
\end{bmatrix}^\top.
\end{equation}
Given $d$ matrices $\mathbf{A}_k   \in \mathbb{R}^{I_1\times J_k}, k \in  \left \{ 1,\ldots,d \right \}$, we use the notation
\begin{equation}
    \diamond_{k=1}^{d}\mathbf{A}_k=\mathbf{A}_1 \diamond \mathbf{A}_2 \diamond\ldots \diamond \mathbf{A}_d \in \mathbb{R}^{I_1 \times \prod_{k=1}^d J_k}.
\end{equation}
Consider two matrices $\mathbf{A}=\begin{bmatrix}
\mathbf{A}[:,1]& 
\mathbf{A}[:,2]& 
\ldots &
\mathbf{A}[:,J_1]
\end{bmatrix} \in \mathbb{R}^{I_1\times J_1}$
and $\mathbf{B}=\begin{bmatrix}
\mathbf{B}[:,1]&
\mathbf{B}[:,2]&
\ldots &
\mathbf{B}[:,J_1]
\end{bmatrix} \in \mathbb{R}^{I_2\times J_1}$, where $ \mathbf{A}[:,k] \in \mathbb{R}^{I_1 \times 1}$ and $ \mathbf{B}[:,k]  \in \mathbb{R}^{I_2 \times 1}$ for $k \in \left \{ 1,...,J_1 \right \}$ . The column-wise Khatri-Rao product $\mathbf{A} * \mathbf{B}$ is a matrix of dimension $(I_1I_2) \times J_1$ defined as
\begin{equation}
    \mathbf{A} * \mathbf{B}=\left[\mathbf{A}[:,1] \otimes \mathbf{B}[:,1] \hspace{0.2cm} \mathbf{A}[:,2] \otimes \mathbf{B}[:,2] \quad \ldots \quad \mathbf{A}[:,J_1] \otimes \mathbf{B}[:,J_1]\right],
\end{equation}
where $\mathbf{A}[:,k]\otimes\mathbf{B}[:,k] $ for $k \in \left \{ 1,...,J_1 \right \}$ defines the Kronecker product between vectors $\mathbf{A}[:,k]$ and $\mathbf{B}[:,k]$. That is, each column of $\mathbf{A} * \mathbf{B}$ is the Kronecker product between the respective columns of the two input matrices $\mathbf{A}$ and $\mathbf{B}$. The relation between column-wise and row-wise Khatri-Rao product is the following
\begin{equation}
\label{eq::relkhatr}
    (\mathbf{A} * \mathbf{B})^\top=\mathbf{A}^\top \diamond \mathbf{B}^\top.
\end{equation}
We give several useful relations among these matrix products that we use in our derivations.
\begin{prop}[\cite{kroni}]
\label{prop::relations}
Consider matrices $\mathbf{A}  \in \mathbb{R}^{I_1\times J_1}$  , $\mathbf{B} \in \mathbb{R}^{I_1 \times J_2}$ , $\mathbf{C}  \in \mathbb{R}^{J_1 \times J_3}$, and $\mathbf{D}  \in \mathbb{R}^{J_2 \times J_4}$, then 
\begin{equation}
    (\mathbf{A} \diamond \mathbf{B})(\mathbf{C} \otimes \mathbf{D})=(\mathbf{A} \mathbf{C}) \diamond(\mathbf{B} \mathbf{D}).
\end{equation}
Consider matrices $\mathbf{A}   \in \mathbb{R}^{I_1 \times J_1}$ and $\mathbf{B}   \in \mathbb{R}^{I_1 \times J_2}$ and $\mathbf{C}  \in \mathbb{R}^{J_1 \times J_3}$, and $\mathbf{D}  \in \mathbb{R}^{J_2 \times J_3}$, then
\begin{equation}
    (\mathbf{A} \diamond \mathbf{B})(\mathbf{C} * \mathbf{D})=(\mathbf{A} \mathbf{C}) \odot(\mathbf{B D}).
\end{equation}
\end{prop}
In this paper, we use the concept of matricization, also called tensor unfolding \cite{doi:10.1137/07070111X}. The mode-$j$ matricization of a tensor  $\mathcal{B} \in \mathbb{R}^{I_{1} \times I_{2} \times \cdots \times I_{d}}, d \in  \mathbb{N}$, referred to as $\mathbf{B}_{(j)} \in \mathbb{R}^{I_{j} \times I_{1} I_{2} \cdots I_{j-1} I_{j+1} \cdots I_{d}}, j \in \left \{ 1,\ldots,d \right \}$, can be defined by the following mapping
\begin{equation}
\label{eq::mode}
  \mathcal{B}(i_1,i_2,\cdots,i_d) = \mathbf{B}(i_{j}, i_1i_2 \ldots i_{j-1}i_{j+1} \ldots i_d)=\mathbf{B}(i_{j},\bar{i}),
\end{equation}
$\text{ with } \bar{i}=1+\sum_{\substack{k=1 \\  k \neq j}}^{d}\left(\left(i_{k}-1\right) \prod_{\substack{m=1 \\ m \neq j}}^{k-1} I_{m}\right).$ For example, if $d=4$, the mode-$1$ matricization of $\mathcal{B} \in \mathbb{R}^{I_{1} \times I_{2} \times I_{3} \times I_{4}}$ which is denoted by $\mathbf{B}_{(1)} \in \mathbb{R}^{I_{1} \times I_{2} I_{3} I_{4}}$ can be defined by the following mapping
\begin{equation}
    \mathcal{B}(i_1,i_2,i_3,i_4) = \mathbf{B}(i_1, i_2i_3i_4)=\mathbf{B}(i_{1},\bar{i}),
\end{equation}
with $\bar{i}=1+(i_2-1)I_1+(i_3-1)I_2+(i_4-1)I_2I_3.$ The mode-$(1,2)$ matricization of $\mathcal{B} \in \mathbb{R}^{I_{1} \times I_{2} \times I_{3} \times I_{4}}$ which is denoted by $\mathbf{B}_{(1,2)} \in \mathbb{R}^{I_{1} I_{2} \times  I_{3} I_{4}}$ can be denoted entry-wise as follows
\begin{equation}
    \mathcal{B}(i_1,i_2,i_3,i_4) = \mathbf{B}(i_1 i_2,i_3i_4).
\end{equation}

\section{Long-range TEI tensor factorization through  Tensorized Approximation (LTEI-\lrttei)}
\label{sec::LRTTEI}
In this section we introduce a new numerical method that allows to evaluate efficiently the two-electron integrals through the factorization of the long-range Coulomb potential.
This method, that we refer to as {\lrttei}, factorizes the fourth order long-range two-electron integrals tensor $\mathcal{B}^{lr}$ through the approximation of the long-range kernel  $K(\boldsymbol{x},\boldsymbol{y})$ with two-dimensional Chebyshev interpolation and Gaussian quadrature.  Error bounds for the numerical approximation of the long-range two-electron integrals are also provided.
approximated six-dimensional integral.
\subsection{The element-wise evaluation of the TEI tensor}
\label{subsubsection:ewf}
We first describe the efficient evaluation of the six-dimensional integrals $\mathcal{B}^{lr}(\mu,\nu,\kappa,\lambda)$ defined in \eqref{long-range}. We start by presenting our approach for computing the long-range $K(\boldsymbol{x},\boldsymbol{y})$ defined as
\begin{equation}
\label{eq::kernel1}
    K(\boldsymbol{x},\boldsymbol{y})=\frac{erf(\omega  \left \| \boldsymbol{x}-\boldsymbol{y}    \right \| )}{  \left \| \boldsymbol{x}-\boldsymbol{y}    \right \|}=\frac{2}{\sqrt \pi}\frac{\int_{\left [0, \omega \left \| \boldsymbol{x}-\boldsymbol{y}    \right \| \right ]} exp\left (-t^{2}  \right )dt}{  \left \| \boldsymbol{x}-\boldsymbol{y}    \right \|},\boldsymbol{x},\boldsymbol{y} \in \mathbb{R}^{3}.
\end{equation}
Let $t=s \left \| \boldsymbol{x}-\boldsymbol{y}    \right \|$. With this change of variable, we obtain
\begin{equation}
\label{eq::kernel}
    K(\boldsymbol{x},\boldsymbol{y})=\frac{2}{\sqrt \pi}\int_{ \left [0,\omega  \right ]}exp\left (-s^{2} \left \| \boldsymbol{x}-\boldsymbol{y}    \right \|^{2}\right )ds,\boldsymbol{x},\boldsymbol{y} \in \mathbb{R}^{3}.
\end{equation}
Using the Gaussian quadrature rule (see Proposition~\ref{def::gauss}), we  can evaluate numerically the integral in \eqref{eq::kernel} as
\begin{small}
\begin{equation}\label{TEI2}
 \int_{[0,\omega]} exp\left (-s^{2} \left \| \boldsymbol{x}-\boldsymbol{y}    \right \|^{2}\right )ds=\frac{\omega}{2}\int_{\left [-1,1  \right ]}exp\left (-(\frac{\omega}{2}+\frac{\omega}{2}z)^{2} \left \| \boldsymbol{x}-\boldsymbol{y}    \right \|^{2}  \right )dz \approx  \frac{\omega}{2} \sum_{i=1}^{N_{q_1}} w_{i} exp\left (-(\frac{\omega}{2}+\frac{\omega}{2}z_{i})^{2} \left \| \boldsymbol{x}-\boldsymbol{y}    \right \|^{2}  \right ),
\end{equation}
\end{small}
where $w_{i}$ are the Gaussian quadrature weights, $z_{i}$ are the Gaussian quadrature nodes, and $N_{q_1}$ is the number of quadrature points. The coordinates of $\boldsymbol{x}$ and $\boldsymbol{y}$ are denoted by  $(\boldsymbol{x}_1,\boldsymbol{x}_2,\boldsymbol{x}_3), (\boldsymbol{y}_1,\boldsymbol{y}_2,\boldsymbol{y}_3)$ respectively. The exponential term in \eqref{TEI2} can be written as 
\begin{equation}
    exp\left (-(\frac{\omega}{2}+\frac{\omega}{2}z_{i})^{2} \left \| \boldsymbol{x}-\boldsymbol{y}    \right \|^{2}  \right )=\prod_{l=1}^{3}exp\left (-(\frac{\omega}{2}+\frac{\omega}{2}z_{i})^{2}(\boldsymbol{x}_l-\boldsymbol{y}_l)^{2} \right ), l \in \left \{ 1,2,3 \right \}.
\end{equation}
Given the truncated computational box $[-b,b]^3, b \in \mathbb{R}$, each function of the form $exp\left (-(\frac{\omega}{2}+\frac{\omega}{2}z_{i})^{2}(\boldsymbol{x}_l-\boldsymbol{y}_l)^{2}  \right ), i \in \left \{ 1,\ldots,N_{q_1} \right \},l \in \left \{ 1,2,3 \right \}$ is  smooth, differentiable (hence continuous) on $[-b,b]^2$, so that it is  an excellent candidate for two-dimensional Chebyshev interpolation. According to Definition~\ref{def::Chebyspoints}, the interpolated function can be written as 
\begin{equation}\label{chebcheb}
     exp\left (-(\frac{\omega}{2}+\frac{\omega}{2}z_{i})^{2}(\boldsymbol{x}_l-\boldsymbol{y}_l)^{2}  \right ) \approx  \sum_{n_l,m_l=1}^{N_i} \alpha^{(i)}_{n_lm_l} T^{(i)}_{n_l}(\boldsymbol{x}_l)T^{(i)}_{m_l}(\boldsymbol{y}_l),
\end{equation}
where $N_i$ is the number of interpolation nodes for $i \in \left \{ 1,\cdots,N_{q_1} \right \}$, $\boldsymbol{x}_l,~\boldsymbol{y}_l \in [-b,b]$, and $l \in \left \{ 1,2,3 \right \}.$
We recall that among the advantages of using two-dimensional Chebyshev interpolation method is that forming two-dimensional Chebyshev grids $N_i \times N_i$ for each function \eqref{chebcheb} takes $\mathcal{O}(N_i^2)$ storage complexity, where $N_i$ is the number of interpolation points needed. 
Furthermore, Chebyshev-Lobatto nodes can be obtained in linearithmic time using Fast Fourier Transform (FFT) \cite{Platte2010}. This is one of the reasons for which we use Chebyshev basis.
Our implementation that we discuss in more details in Section~\ref{sec::numerical} uses FFTW \cite{FFTW} routine in Julia  and the \textit{chebfun2} library \cite{chebfun2} to find the number of interpolation points $N_{i}$ of the functions in \eqref{chebcheb}. By replacing \eqref{chebcheb} and \eqref{TEI2} in \eqref{eq::kernel}, the numerical approximation of the kernel $K(\boldsymbol{x},\boldsymbol{y})$ becomes
\begin{equation}
\label{eq::Approx_kernel}
    K(\boldsymbol{x},\boldsymbol{y}) \approx \frac{\omega}{\sqrt \pi}  \sum_{i=1}^{N_{q_1}} w_{i} \left (\sum_{  n_1,m_1,\cdots,n_3,m_3=1 }^{N_i} \prod_{l=1}^{3} \alpha^{(i)}_{n_lm_l} T^{(i)}_{n_l}(\boldsymbol{x}_l)T^{(i)}_{m_l}(\boldsymbol{y}_l)    \right ),
\end{equation}
where  $\omega \geqslant 0$ is the parameter that regulates the separation range of the long-range/short-range interactions, $\alpha^{(i)}_{n_lm_l}$ are the $N_i$ Chebyshev  nodes, $T^{(i)}_{n_l}(\boldsymbol{x}_l), T^{(i)}_{m_l}(\boldsymbol{y}_l) $ are the Chebyshev polynomials (see Definition~\ref{def::Chebyspoints}) and $w_{i}$ are the Gaussian quadrature weights with $i \in  \left \{ 1, \cdots, N_{q_1} \right \}$. All along this paper,  we denote  $N$   the maximum number of interpolation points in the tensorized Chebyshev grid in all directions such that $N=(max(N_i)_{i \in \left \{ 1,\dots,N_{q_1} \right \}})^3$. The precomputation cost here to approximate the kernel \eqref{eq::Approx_kernel} is $\mathcal{O}(N_{q_1}N^{\frac{1}{3}}(log(N^{\frac{1}{3}})+N^{\frac{1}{3}}))$ : $\mathcal{O}(N_{q_1}N^{\frac{1}{3}}log(N^{\frac{1}{3}}))$ FLOPS for the evaluation of the Chebyshev  coefficient matrices using FFT algorithm, linearithmic in the number of interpolation points  \textit{in a single direction} $N^{\frac{1}{3}}$ and linear in the number of quadrature points,  and $\mathcal{O}(N_{q_1}N^{\frac{2}{3}})$ FLOPS for forming the Chebyshev two-dimensional grids.

We consider now the  finite six-dimensional integral $\mathcal{B}^{lr}(\mu,\nu,\kappa,\lambda)$ defined in \eqref{eq::compact} on the same truncated computational box $[-b,b]^3 \times [-b,b]^3, b \in \mathbb{R}$  with ~$\mu,\nu,\kappa,\lambda \in \left \{ 1,..,N_{b} \right \}$,  where $N_b$ is  the number of basis functions that we defined in \eqref{basis} and $b$ is the size of the computational box  that is chosen according to the most slowly decaying   basis functions.  We discuss this aspect in more details in Section~\ref{subsub::adap}. 
By replacing   $K(\boldsymbol{x},\boldsymbol{y})$ with its approximation from \eqref{eq::Approx_kernel},  the numerical  approximation of $\mathcal{B}^{lr}(\mu,\nu,\kappa,\lambda)$, denoted by $\mathcal{B}_{LTEI-{\lrttei}}^{lr}(\mu,\nu,\kappa,\lambda)$, writes

\begin{small}
    \begin{align}
    \label{eq::approx3}
    \mathcal{B}_{LTEI-\lrttei}^{lr}(\mu,\nu,\kappa,\lambda)=\frac{\omega}{\sqrt \pi}  \sum_{i=1}^{N_{q_1}} w_{i}\left (\int_{[-b,b]^{3}}\int_{[-b,b]^{3}} g_{\mu}(\boldsymbol{x})g_{\nu}(\boldsymbol{x})g_{\kappa}(\boldsymbol{y})g_{\lambda}(\boldsymbol{y})\left (\sum_{  n_1,m_1,\cdots,n_3,m_3=1 }^{N_i} \prod_{l=1}^{3}\alpha^{(i)}_{n_lm_l}T^{(i)}_{n_l}(\boldsymbol{x}_l)T^{(i)}_{m_l}(\boldsymbol{y}_l)  \right )d\boldsymbol{x}d\boldsymbol{y}  \right ).
\end{align}
\end{small}

To obtain an efficient factorized representation of $\mathcal{B}_{LTEI-{\lrttei}}^{lr}(\mu,\nu,\kappa,\lambda)$, we further consider the separability of the Gaussian primitives. Let  $g_{\mu\nu}(\boldsymbol{x})=g_{\mu}(\boldsymbol{x})g_{\nu}(\boldsymbol{x})$, $g_{\kappa\lambda}(\boldsymbol{y})=g_{\kappa}(\boldsymbol{y})g_{\lambda}(\boldsymbol{y})$ such that according to \eqref{basis} we have (showing only $ g_{\mu\nu}$ expression)
\begin{eqnarray}\label{eq::gaussians1}
    g_{\mu\nu}(\boldsymbol{x})=g_{\mu}(\boldsymbol{x})g_{\nu}(\boldsymbol{x})&
  =&\sum_{j_1=1}^{I_\mu}\sum_{j_2=1}^{I_\nu}c_{j_1}c_{j_2}\prod_{l=1}^{3}g_{\mu}^{(j_1)}(\boldsymbol{x}_l)g_{\nu}^{(j_2)}(\boldsymbol{x}_l)
  \label{eq::gaussians2}=
\sum_{j=1}^{I_{\mu\nu}}c_j\prod_{l=1}^{3}g_{\mu\nu}^{(j)}(\boldsymbol{x}_l),
\end{eqnarray}
where $I_{\mu\nu}=I_{\mu}I_{\nu},c_j=c_{j_1}c_{j_2}, g_{\mu\nu}^{(j)}(\boldsymbol{x}_l)=g_{\mu}^{(j_1)}(\boldsymbol{x}_l)g_{\nu}^{(j_2)}(\boldsymbol{x}_l)$. Expressing  the three dimensional function $g_{\mu\nu}(\boldsymbol{x})$ as a sum of separable functions is important to reduce the evaluation cost of  $\mathcal{B}_{LTEI-\lrttei}^{lr}(\mu,\nu,\kappa,\lambda)$ such that  after replacing the Gaussian basis functions  in \eqref{eq::approx3} by their separable expression \eqref{eq::gaussians2} we obtain 
\begin{small}
\begin{align}
    &\mathcal{B}_{LTEI-\lrttei}^{lr}(\mu,\nu,\kappa,\lambda)=\frac{\omega}{\sqrt \pi}  \sum_{i=1}^{N_{q_1}} w_{i}\left (\int_{[-b,b]^{3}}\int_{[-b,b]^{3}} g_{\mu\nu}(\boldsymbol{x})g_{\kappa\lambda}(\boldsymbol{y})\left (\sum_{  n_1,m_1,\cdots,n_3,m_3=1 }^{N_i} \prod_{l=1}^{3}\alpha^{(i)}_{n_lm_l}T^{(i)}_{n_l}(\boldsymbol{x}_l)T^{(i)}_{m_l}(\boldsymbol{y}_l)  \right )d\boldsymbol{x}d\boldsymbol{y}  \right )\\&=
    \frac{\omega}{\sqrt \pi}  \sum_{i=1}^{N_{q_1}} w_{i}\left (\int_{[-b,b]^{3}}\int_{[-b,b]^{3}} \left (\sum_{  j=1 }^{I_{\mu\nu}}\sum_{j'=1 }^{I_{\kappa\lambda}}c_jc_{j'}\prod_{l=1}^{3}g_{\mu\nu}^{(j)}(\boldsymbol{x}_l) g_{\kappa\lambda}^{(j')}(\boldsymbol{y}_l) \right )
\left (\sum_{  n_1,m_1,\cdots,n_3,m_3=1 }^{N_i}\prod_{l=1}^{3}\alpha^{(i)}_{n_lm_l}T^{(i)}_{n_l}(\boldsymbol{x}_l)T^{(i)}_{m_l}(\boldsymbol{y}_l)  \right )d\boldsymbol{x}d\boldsymbol{y}  \right )
    \\ \label{eq::approx4} &=\frac{\omega}{\sqrt \pi} \sum_{i=1}^{N_{q_1}} w_{i} \sum_{  j=1 }^{I_{\mu\nu}}\sum_{j'=1 }^{I_{\kappa\lambda}}c_jc_{j'}\underbrace{ \sum_{  \substack{n_1,n_2,n_3 \\ m_1,m_2,m_3=1} }^{N_i} \prod_{l=1}^{3}  \left (\alpha^{(i)}_{n_lm_l} \int_{[-b,b]}g_{\mu\nu}^{(j)}(\boldsymbol{x}_{l})T^{(i)}_{n_l}(\boldsymbol{x}_{l})d\boldsymbol{x}_{l}\int_{[-b,b]}g_{\kappa\lambda}^{(j')}(\boldsymbol{y}_{l})T^{(i)}_{m_l}(\boldsymbol{y}_l)d\boldsymbol{y}_{l} \right )}_{\approx \mathbf{F}^{(i)}(j,j')}.
\end{align}
\end{small}
We note that the expression of $\mathcal{B}_{LTEI-\lrttei}^{lr}(\mu,\nu,\kappa,\lambda)$ in \eqref{eq::approx4} involves the numerical evaluation of one dimensional integrals. We associate each such integral with the element of a matrix and obtain two matrices $\mathbf{W}^{(i,l)}_{\mu\nu} \in \mathbb{R}^{I_{\mu\nu} \times N_i}$ and $\mathbf{W}^{(i,l)}_{\kappa\lambda} \in \mathbb{R}^{I_{\kappa\lambda} \times N_i}$ defined entry-wise as
 \begin{equation}
\label{eq::defV}
    \mathbf{W}^{(i,l)}_{\mu\nu}(j,n_l)=\int_{[-b,b]}g_{\mu\nu}^{(j)}(\boldsymbol{x}_{l})T^{(i)}_{n_l}(\boldsymbol{x}_{l})d\boldsymbol{x}_{l} \text{  and  }  \mathbf{W}^{(i,l)}_{\kappa\lambda}(j',m_l)=\int_{[-b,b]}g_{\kappa\lambda}^{(j')}(\boldsymbol{y}_{l})T^{(i)}_{m_l}(\boldsymbol{y}_{l})d\boldsymbol{y}_{l}.
\end{equation}
We use one-dimensional Gaussian quadrature rule for the evaluation of \eqref{eq::defV}. Their approximation is denoted by $\tilde{\mathbf{W}}^{(i,l)}_{\mu\nu}(j,n_l) (resp. \tilde{ \mathbf{W}}^{(i,l)}_{\kappa\lambda}(j',m_l))$.  We further define matrices $\mathbf{F}^{(i)}, i \in  \left \{ 1,\cdots,N_{q_1} \right \}$, as displayed in \eqref{eq::approx4}. By replacing the expressions of $\tilde{ \mathbf{W}}^{(i,l)}_{\mu \nu}$ and $\tilde{ \mathbf{W}}^{(i,l)}_{\kappa\lambda}$, we obtain 
\begin{eqnarray}
\label{eq::fi_fact}
    \mathbf{F}^{(i)}(j,j')
    =\sum_{  \substack{n_1,n_2,n_3 \\ m_1,m_2,m_3=1} }^{N_i} \prod_{l=1}^{3}  \left (\alpha^{(i)}_{n_lm_l} \tilde{\mathbf{W}}^{(i,l)}_{\mu\nu}(j,n_l) \tilde{ \mathbf{W}}^{(i,l)}_{\kappa\lambda}(j',m_l) \right ).
\end{eqnarray}
By changing the order of summation in \eqref{eq::fi_fact} and exploiting  Khatri-Rao as well as Kronecker structures (see their definitions in Section~\ref{sec::def}), we obtain the factorized representation of $\mathcal{B}_{LTEI-\lrttei}^{lr}$ as given in the following theorem.
\begin{theorem}
\label{prop::fi_fact}
The long-range two-electrons integrals has a factorized representation that writes
\begin{equation}\label{approx}
       \mathcal{B}_{LTEI-\lrttei}^{lr}(\mu,\nu,\kappa,\lambda)
      =\frac{\omega}{\sqrt \pi} \sum_{i=1}^{N_{q_1}} w_{i}\sum_{  j =1}^{I_{\mu\nu} }\sum_{  j' =1}^{I_{\kappa\lambda} } c_jc_{j'}\mathbf{F}^{(i)}(j,j'),
\end{equation}
where  $ \mathbf{F}^{(i)}$ $\in \mathbb{R}^{I_{\mu\nu} \times I_{\kappa\lambda}}$  
\begin{equation}\label{Fi}
    \mathbf{F}^{(i)}=(\diamond_{l=1}^{3}\mathbf{\tilde{W}}^{(i,l)}_{\mu\nu})(\otimes_{l=1}^{3}\mathbf{A}^{(i)})(\diamond_{l=1}^{3}\mathbf{\tilde{W}}^{(i,l)}_{\kappa\lambda})^\top=\odot_{l=1}^{3}\mathbf{\tilde{W}}^{(i,l)}_{\mu\nu}\mathbf{A}^{(i)}\mathbf{\tilde{W}}^{(i,l)\top}_{\kappa\lambda},
\end{equation}
 where  $\mathbf{A}^{(i)} \in  \mathbb{R}^{N_i \times N_i}$ are the Chebyshev coefficients matrices such that  $\mathbf{A}^{(i)}(n_l,m_l)=  \alpha^{(i)}_{n_lm_l}  \text{ for } n_l,m_l\in \left [1,\cdots ,N_i  \right ],l \in  \left \{ 1,2,3 \right \}$ with $\alpha^{(i)}_{n_lm_l}$ defined in \eqref{chebcheb},$\mathbf{\tilde{W}}^{(i,l)}_{\mu\nu} \in \mathbb{R}^{I_{\mu\nu} \times N_i}$ and $\mathbf{\tilde{W}}^{(i,l)}_{\kappa\lambda} \in \mathbb{R}^{I_{\kappa\lambda} \times N_i}$ are the numerical approximation of the one-dimensional integrals  defined   in \eqref{eq::defV}.
\end{theorem}
Algorithm~\ref{alg:capp} computes the approximated entries $\mathcal{B}_{LTEI-\lrttei}^{lr}(\mu, \nu, \kappa, \lambda)$ \eqref{approx} given the coefficient matrix obtained from the two-dimensional Chebyshev interpolation $\mathbf{A}^{(i)}$  $\in \mathbb{R}^{N_i \times N_i}, \text{ for } i \in  \left \{ 1,\ldots,N_{q_1} \right \}$ and for any   given pairs of $\mu, \nu, \kappa,\lambda$. This approach allows  to reduce the storage complexity (resp. arithmetic complexity) to $\mathcal{O}\left(  \sum_{i=1}^{N_{q_1}}N_i(N_i+I_{\mu\nu}+I_{\kappa\lambda}) \right)\sim \mathcal{O}\left( N_{q_1}N^{\frac{1}{3}}(N^{\frac{1}{3}}+I_{\mu\nu}+I_{\kappa\lambda}) \right)$ (resp. $\mathcal{O}\left( \sum_{i=1}^{N_{q_1}}N_iI_{\kappa\lambda}(N_i+I_{\mu\nu}) \right)\sim\mathcal{O}\left( N_{q_1}N^{\frac{1}{3}}I_{\kappa \lambda}(N^{\frac{1}{3}}+I_{\mu\nu}) \right)$), with   $N^{\frac{1}{3}}=(max(N_i)_{i \in \left \{ 1,\dots,N_{q_1} \right \}})$, instead of $\mathcal{O}(N(N+I_{\mu\nu}+I_{\kappa\lambda}))$ (resp. $\mathcal{O}(NI_{\kappa\lambda}(N+I_{\mu\nu}+I_{\kappa\lambda}))$), using naïve tensorized three dimensional quadrature on the computational box $[-b,b]^3$. Numerical results for this element-wise factorization are  summarized in Section~\ref{sec::numerical}.
\begin{algorithm}[H]
\caption{Compute $\mathcal{B}_{LTEI-\lrttei}^{lr}(\mu,\nu,\kappa,\lambda)$} \label{alg:capp}
\begin{algorithmic}
\Require Chebyshev coefficient matrices $\mathbf{A}^{(i)}$, $\mu,\nu$,$\kappa, \lambda$, $w_i$ for $i \in \left \{ 1,\cdots,N_{q_1} \right \}.$
\Ensure $\mathcal{B}_{LTEI-\lrttei}^{lr}(\mu,\nu,\kappa,\lambda)$
\State{$Compute ~\mathbf{\tilde{W}}_{\mu\nu}^{(i,1)}, \mathbf{\tilde{W}}_{\mu\nu}^{(i,2)}, \mathbf{\tilde{W}}_{\mu\nu}^{(i,3)}    I_{\mu\nu} \times N_i \text{ matrices }~(\text{according to } \eqref{eq::defV})$.}
\State{$Compute~\mathbf{\tilde{W}}_{\kappa\lambda}^{(i,1)}, \mathbf{\tilde{W}}_{\kappa\lambda}^{(i,2)}, \mathbf{\tilde{W}}_{\kappa\lambda}^{(i,3)}    I_{\kappa\lambda} \times  N_i  \text{ matrices }
~~~~(\text{according to } \eqref{eq::defV})$.}
\State{Set $s=0.$}
\For{i=1 to $N_{q_1}$}
\State{$\mathbf{F}^{(i)}=\odot_{l=1}^{3}\mathbf{\tilde{W}}^{(i,l)}_{\mu\nu}\mathbf{A}^{(i)}\mathbf{\tilde{W}}^{(i,l)\top}_{\kappa\lambda}$.}
\State{$s=s+w_i\sum_{  j =1}^{I_{\mu\nu} }\sum_{  j' =1}^{I_{\kappa\lambda} } c_jc_{j'}\mathbf{F}^{(i)}(j,j')$.}
\EndFor
\State{$\mathcal{B}_{LTEI-\lrttei}^{lr}(\mu,\nu,\kappa,\lambda)=\frac{\omega}{\sqrt \pi} s.$}
\end{algorithmic}
\end{algorithm}

\subsection{Error bound of the two-electron integrals numerical approximation}
\label{subsubsection:eb}
In what follows, we give a theoretical  error bound associated with  the element-wise numerical approximation of  $\mathcal{B}^{lr}(\mu,\nu,\kappa,\lambda)$ introduced in \eqref{approx}.
\begin{prop}
The element-wise  error $\epsilon$ between the long-range two-electron integrals  $\mathcal{B}^{lr}(\mu,\nu,\kappa,\lambda)$, given a finite box $[-b,b]^3$,   and it's approximation $\mathcal{B}_{LTEI-\lrttei}^{lr}$ can be bounded as follows
\begin{equation}
      \left | \epsilon \right |:=  \left | \mathcal{B}^{lr}(\mu,\nu,\kappa,\lambda)-\mathcal{B}_{LTEI-\lrttei}^{lr}(\mu,\nu,\kappa,\lambda) \right | \leq  c_1 \sup_{\boldsymbol{x}, \boldsymbol{y} \in \left [ -b,b \right ]^3}\left( \left \| \frac{d^{2N_{q_1}}}{d s^{2N_{q_1}}}f(s,\boldsymbol{x},\boldsymbol{y}) \right \|_{\infty,[0,\omega]} \right)+\frac{\omega}{\sqrt{\pi}} c_2,
\end{equation}
where we define the multivariate function 
\begin{equation}
\label{eq::h_def}
    f(s,\boldsymbol{x},\boldsymbol{y})=exp(-s^2 \left \| \boldsymbol{x}-\boldsymbol{y}    \right \|^2), s \in  \left [ 0,\omega \right ], \boldsymbol{x}, \boldsymbol{y} \in \left [ -b,b \right ]^3.
\end{equation}
$\epsilon$ is the approximation error, $N_{q_1}$ is the number of quadrature points, $c_1$ and $c_2$  are defined in the following proof.
\end{prop}
\begin{proof}
We start by introducing the following function
\begin{eqnarray}
h(z_i)&=&\int_{[-b,b]^{3}}\int_{[-b,b]^{3}}g_{\mu \nu}(\boldsymbol{x}) g_{\kappa \lambda}(\boldsymbol{y})  exp\left (-(\frac{\omega}{2}+\frac{\omega}{2}z_{i})^{2}\left \| \boldsymbol{x}-\boldsymbol{y}    \right \|^{2}  \right )d\boldsymbol{x}d\boldsymbol{y}\\
&=&\sum_{j=1}^{I_{\mu\nu}}\sum_{j'=1}^{I_{\kappa\lambda}}c_jc_{j'}\left ( \prod_{l=1}^{3}\int_{[-b,b]^2}g^{(j)}_{\mu \nu}(x_l)g^{(j')}_{\kappa\lambda}(y_l) exp\left (-(\frac{\omega}{2}+\frac{\omega}{2}z_{i})^{2} (x_l-y_l)^{2}  \right )dx_ldy_l\right ),
\end{eqnarray}
with $z_i, i \in \left \{ 1,\cdots,N_{q_1} \right \}$ being the Gaussian quadrature nodes. The upper bound of $\epsilon$ can be found as follows
\begin{small}
\begin{equation}
     \left |  \epsilon\right |=\left |\mathcal{B}^{lr}(\mu,\nu,\kappa,\lambda)- \mathcal{B}_{LTEI-\lrttei}^{lr}(\mu,\nu,\kappa,\lambda)\right | \leq   \underbrace{\left |\mathcal{B}^{lr}(\mu,\nu,\kappa,\lambda)-  \frac{\omega}{\sqrt \pi}\sum_{i=1}^{N_{q_1}} w_{i} h(z_i)\right |}_{\epsilon_1}
     + \underbrace{\left |\frac{\omega}{\sqrt \pi}\sum_{i=1}^{N_{q_1}} w_{i} h(z_i)-\mathcal{B}_{LTEI-\lrttei}^{lr}(\mu,\nu,\kappa,\lambda)  \right |}_{\epsilon_2},
\end{equation}
\end{small}
Using Proposition~\ref{def::gauss}, triangle inequality, and Stirling formula given by $n! \approx \sqrt{2 \pi n}\left(\frac{n}{\mathrm{e}}\right)^{n}$, $\epsilon_1$ is bounded as follows
\begin{equation}
    \epsilon_1  \leqslant c_1\sup_{\boldsymbol{x}, \boldsymbol{y} \in \left [ -b,b \right ]^3}\left( \left \| \frac{d^{2N_{q_1}}}{d s^{2N_{q_1}}}f(s,\boldsymbol{x},\boldsymbol{y}) \right \|_{\infty,[0,\omega]} \right)\text{, }  c_1=\frac{2e_{N_{q_1}} }{\sqrt{\pi}}b^6 \left \| g_{\mu\nu}\right \|_{\infty,\left [ -b,b \right ]^3}\left \| g_{\kappa\lambda}\right \|_{\infty,\left [ -b,b \right ]^3},
\end{equation}
 with $e_{N_{q_1}}=\frac{\omega^{2N_{q_1}+1}e^{2N_{q_1}} (N_{q_1}\pi)^{\frac{1}{2}}}{2^{6N_{q_1}+1}N_{q_1}^{2N_{q_1}}(2N_{q_1}+1)}$. 
The error bound of $\epsilon_2$ needs a more detailed  explanation. We replace $\mathcal{B}_{LTEI-\lrttei}^{lr}(\mu,\nu,\kappa,\lambda)$ by its expression defined in \eqref{approx} such that 
\begin{equation}
     \left| \epsilon_2 \right|=\left| \frac{\omega}{\sqrt{\pi}}\sum_{i=1}^{N_{q_1}}w_i\left( h(z_i)-\sum_{  j =1}^{I_{\mu\nu} }\sum_{  j' =1}^{I_{\kappa\lambda} } c_jc_{j'}\mathbf{F}^{(i)}(j,j') \right) \right|\le \frac{\omega}{\sqrt{\pi}} \sum_{i=1}^{N_{q_1}}\left| w_i \right|\left| h(z_i)-\sum_{  j =1}^{I_{\mu\nu} }\sum_{  j' =1}^{I_{\kappa\lambda} } c_jc_{j'}\mathbf{F}^{(i)}(j,j') \right|,
 \end{equation}
 with $\mathbf{F}^{(i)}(j,j')$ being defined in \eqref{eq::fi_fact}. Using the triangle inequality, the expression of $\left| h(z_i)-\sum_{  j =1}^{I_{\mu\nu} }\sum_{  j' =1}^{I_{\kappa\lambda} } c_jc_{j'}\mathbf{F}^{(i)}(j,j') \right|$, for $i \in  \left\{1,\cdots,N_{q_1}  \right\}$, can be bounded as follows
  \begin{small}
 \begin{eqnarray}
 \label{eq::first_err}
    \left| h(z_i)-\sum_{  j =1}^{I_{\mu\nu} }\sum_{  j' =1}^{I_{\kappa\lambda} } c_jc_{j'}\mathbf{F}^{(i)}(j,j') \right| &\le& \sum_{j=1}^{I_{\mu\nu}}\sum_{j'=1}^{I_{\kappa\lambda}}c_jc_{j'}\left| \prod_{l=1}^{3}\int_{[-b,b]^2}g_{\mu\nu}^{^{(j)}}(x_l)g_{\kappa\lambda}^{^{(j')}}(y_l)e^{-(\frac{\omega}{2}+\frac{\omega}{2}z_i)^2(x_l-y_l)^2}dx_ldy_l-\mathbf{F}^{(i)}(j,j')\right|.
\end{eqnarray}
 \end{small}
 In order to evaluate the bound of \eqref{eq::first_err}, one needs to evaluate the error bound of the following expression using Propositon~\ref{def::Chebyupperbound} and Proposition~\ref{def::gauss}
 \begin{align}
\label{eq::second_err}
 \left|\int_{[-b,b]^2}g_{\mu\nu}^{^{(j)}}(x_l)g_{\kappa\lambda}^{^{(j')}}(y_l)e^{-(\frac{\omega}{2}+\frac{\omega}{2}z_i)^2(x_l-y_l)^2}-\sum_{n_l,m_l}^{N_i}\alpha_{n_lm_l}^{(i)}\tilde{\mathbf{W}}^{(i,l)}_{\mu\nu}(j,n_l) \tilde{ \mathbf{W}}^{(i,l)}_{\kappa\lambda}(j',m_l) \right|
&\le \beta_i, l \in  \left \{ 1,2,3 \right \},
\end{align}
where for $i \in  \left\{ 1,\cdots,N_{q_1} \right\},j \in \left\{ 1,\cdots,I_{\mu\nu} \right\},$ and $j' \in  \left\{ 1,\cdots, I_{\kappa\lambda} \right\}$, $\beta_i$ is defined as follows
\begin{small}
    \begin{align}
\beta_i&=(2b)^2\left \| g_{\mu\nu}^{(j)}\right \|_{\infty,[-b,b]}\left \| g_{\kappa\lambda}^{(j')}\right \|_{\infty,[-b,b]}e_{N_i}
\\ & +e_{N_{q_2}}\left ( \left \| g_{\kappa\lambda}^{(j')}\right \|_{\infty,[-b,b]}\left \| \frac{d^{2N_{q_2}} \left (g_{\mu\nu}^{(j)} T_{n_1}^{(i)}  \right )(x)}{dx^{2N_{q_2}}}    \right \|_{\infty,[-b,b]} +N_{q_2} \max_{1\leq i\leq N_{q_2}}(w_i)\left \| g_{\mu\nu}^{(j)}\right \|_{\infty,[-b,b]} \left \| \frac{d^{2N_{q_2}}\left (g_{\kappa\lambda}^{(j')} T_{m_1}^{(i)}  \right )(y) }{d y^{2N_{q_2}}}    \right \|_{\infty,[-b,b]}\right ),
\end{align}
\end{small}
with $e_{N_{q_2}}=\frac{(2b)^{2N_{q_2}+1}e^{2N_{q_2}}(N_{q_2}\pi)^{\frac{1}{2}}}{2^{6N_{q_2}+1}N_{q_2}^{2N_{q_2}}(2N_{q_2}+1)}$. The term $e_{N_i}$ is  defined as follows
\begin{equation}
    e_{N_i}=\frac{b^{N_i+1}}{2^{1+N_i}}\frac{e^{ -N_i+1}}{\sqrt{2 \pi \left (N_i+1  \right )} \left (1+N_i^{\frac{1}{3} \left (N_i+1  \right )}  \right )}\left [ \max _{  -b\le \xi \le b} \left|\frac{\partial^{N_i+1} \mathscr{F} (z_i,\xi,y)}{\partial \xi^{N_i+1}}\right|+\delta\max _{-b\le \xi,\eta \le b} \left |\frac{\partial^{N_i+1} \mathscr{F}(z_i,\xi, \eta)}{\partial \eta^{N_i+1}}  \right |\right ],
\end{equation}
where $\mathscr{F}(z_i,x,y)=e^{ -(\frac{\omega}{2}+\frac{\omega}{2}z_i)^2(x-y)^2 }$ with $z_i$ being the Gaussian quadrature points and $\delta$ is defined in \eqref{def::Chebyupperbound}. Now, by factorizing \eqref{eq::first_err} and using \eqref{eq::second_err}, one arrives at the desired error bound of $\epsilon_2$
\begin{equation}
     \epsilon_2 \leq \frac{\omega}{\sqrt{\pi}}c_2,
 \end{equation}
 with 
 \begin{small}
 \begin{equation}
   c_2=N_{q_1} \sup_{ \substack{1\leq i\leq N_{q_1}} }  \left(   \left| w_i \right|\sum_{j}^{I_{\mu\nu}}\sum_{j'}^{I_{\kappa\lambda}}c_jc_{j'}
(2b)^4\left \| g_{\mu\nu}^{(j)}\right \|^2_{\infty,[-b,b]}\left \| g_{\kappa\lambda}^{(j')}\right \|^2_{\infty,[-b,b]}\sup_{1\le n_1,m_1\le N_i } \left( (1+N_i\alpha_{n_1,m_1}^{(i)}+\left( N_i\alpha_{n_1,m_1}^{(i)} \right)^2)\beta_i   \right)\right ).
 \end{equation}
 \end{small}
 \end{proof}
 All along this work, we consider a fixed number of quadrature points $N_{q_2}$ for the evaluation of \eqref{eq::defV}  and the study is not being done on the parameter $N_{q_2}$ since the computations using these one-dimensional Gaussian quadrature to evaluate \eqref{eq::defV} are involved in the precomputation steps.
As we notice here, the approximation error depends on the value of $\omega$, the number of quadrature points $N_{q_1}$, the regularity of the function $f$ , the Gaussian-type functions and on the dimension of the hypercube.
\subsection{A new decomposition of TEI tensor $\mathcal{B}^{lr}$ through {\lrttei} approach }
As already discussed in the introduction, one of the main steps in many methods in quantum chemistry involves the application of
the two-electron integrals tensor $\mathcal{B}^{lr}  \in  \mathbb{R}^{N_b \times N_b \times N_b \times N_b}$ to a vector with $N_b^2$ elements or a set of such vectors. 
To perform efficiently this contraction operation, we introduce in this section a factorized representation of the fourth-order two-electron integrals tensor $\mathcal{B}^{lr}$ that expands the factorized representation of its elements summarized in Theorem~\ref{prop::fi_fact}. We show also that the obtained tensorized structure is beneficial to accelerate contraction operations.
\subsubsection{Factorized expression of $\mathcal{B}^{lr}$ }
In what follows we derive the factorized representation of $\mathbf{B}^{lr}$(mode-(1,2) matricization of $\mathcal{B}^{lr} \in \mathbb{R}^{N_b \times N_b  \times N_b\times N_b}$). We slightly modify  the expression of the approximation of  the two-electron integrals (see  Theorem~\ref{prop::fi_fact}) by changing the order of summation to obtain
\begin{equation}
\label{eq::approx33}
    \mathcal{B}_{LTEI-\lrttei}^{lr}(\mu,\nu,\kappa,\lambda) = \frac{\omega}{\sqrt \pi} \sum_{i=1}^{N_{q_1}} w_{i} \left[  \sum_{  \substack{n_1,n_2,n_3\\ m_1,m_2,m_3=1}}^{N_i}\left( \sum_{  j =1}^{I_{\mu\nu} }c_j\prod_{l=1}^{3}   \tilde{ \mathbf{W}}^{(i,l)}_{\mu\nu} (j,n_l) \right)   \prod_{l=1}^{3} \alpha^{(i)}_{n_lm_l}   \left( \sum_{  j' =1}^{I_{\kappa\lambda} }  c_{j'}\prod_{l=1}^{3} \tilde{ \mathbf{W}}^{(i,l)}_{\kappa\lambda}(j',m_l) \right)   \right].
\end{equation}
\label{subsubsec::first}
We introduce the matrices $\mathbf{M}_{\lrttei}^{(i)} \in \mathbb{R}^{N_b^2 \times N_i^3}$ with single entries $\sum_{  j =1}^{I_{\mu\nu} }c_j\prod_{l=1}^{3}   \tilde{ \mathbf{W}}^{(i,l)}_{\mu\nu} (j,n_l), \mu, \nu \in  \left\{ 1,\cdots,N_b \right\}, n_l \in   \left\{ 1,\cdots,N_i \right\}, l \in \left\{ 1,2,3\right\}$ such that the approximation of   mode-(1,2) matricization of $\mathcal{B}_{LTEI-\lrttei}^{lr}$, referred to as $\mathbf{B}_{LTEI-\lrttei}^{lr}$, writes
\begin{small}
\begin{equation}\label{eq::tensors}
    \mathbf{B}_{LTEI-\lrttei}^{lr}=\frac{\omega}{\sqrt \pi}\sum_{i=1}^{N_{q_1}}w_i\mathbf{M}_{\lrttei}^{(i)}(\otimes_{l=1}^{3}\mathbf{A}^{(i)})\mathbf{M}_{\lrttei}^{(i) \top} \in \mathbb{R}^{N_b^2 \times N_b^2} \text{ and }  \mathbf{M}_{\lrttei}^{(i)}(\mu\nu,n_1n_2n_3)=\sum_{  j =1}^{I_{\mu\nu} }c_j\prod_{l=1}^{3}   \tilde{ \mathbf{W}}^{(i,l)}_{\mu\nu} (j,n_l), i \in \left \{ 1,...,N_{q_1} \right \},
\end{equation}
\end{small}
where $\mathbf{A}^{(i)}=\left ( \alpha^{(i)}_{n_lm_l} \right )_{n_l,m_l\in \left [1,\cdots ,N_i  \right ]},l \in  \left \{ 1,2,3 \right \}$ are  the coefficient matrices obtained from the two-dimensional Chebychev  interpolation (see Definition~\ref{chebcheb}).
\subsubsection{Fast evaluation of tensor products}
\label{sec::fasteval}
In practice, we only need to compute the matrix $\mathbf{M}_{\lrttei}^{(i)}$  with the maximum number of interpolation points $N$. We denote this matrix by  $\mathbf{M}_{\lrttei,max} \in \mathbb{R}^{N_b^2 \times N }$. In fact, the other matrices  $\mathbf{M}_{\lrttei}^{(i)} \in \mathbb{R}^{N_b^2 \times N_i^3}, N_i^3  \le N, i \in  \left\{ 1,\cdots,N_{q_1} \right\}$ have common entries with $\mathbf{M}_{\lrttei,max}$. For example, given the two following matrices $\mathbf{M}_{\lrttei}^{(i)} \in \mathbb{R}^{N_b^2 \times N_i}$ and $\mathbf{M}_{\lrttei}^{(j)} \in \mathbb{R}^{N_b^2 \times N_j}$ with $N_i  < N_j$ and $i,j \in  \left \{1,\cdots,N_{q_1}  \right \}$, we have
\begin{equation}
    \mathbf{M}_{\lrttei}^{(i)}(\mu\nu,n_1n_2n_3)=\mathbf{M}_{\lrttei}^{(j)}(\mu\nu,n_1n_2n_3), n_1,n_2,n_3 \in  \left \{1,\cdots,N_{i}  \right \}.
\end{equation}
This can also be illustrated in Figure~\ref{fig:tensor_contraction}. Therefore, the storage complexity for storing $\mathbf{M}_{\lrttei,max}$ is $\mathcal{O}(N N_b^2 )$. 
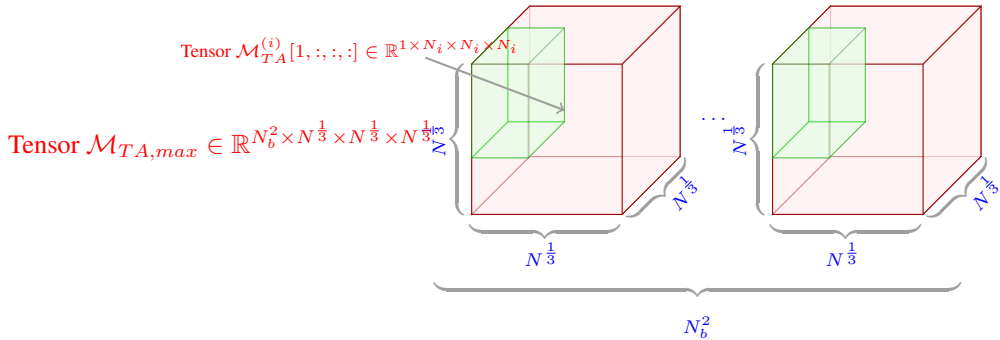
\begin{figure}[H]
\newcommand{\Depth}{2}
\newcommand{\Height}{2}
\newcommand{\Width}{2}
\newcommand{\xx}{1}
\newcommand{\yy}{1}
\newcommand{\zz}{1}
\begin{tikzpicture}
\coordinate (O) at (4,0,0);
\coordinate (A) at (4,\Width,0);
\coordinate (B) at (4,\Width,\Height);
\coordinate (C) at (4,0,\Height);
\coordinate (D) at (4+\Depth,0,0);
\coordinate (E) at (4+\Depth,\Width,0);
\coordinate (F) at (4+\Depth,\Width,\Height);
\coordinate (G) at (4+\Depth,0,\Height);
\draw[red!60!black,fill=red!5] (O) -- (C) -- (G) -- (D) -- cycle;
\draw[red!60!black,fill=red!5] (O) -- (A) -- (E) -- (D) -- cycle;
\draw[red!60!black,fill=red!5] (O) -- (A) -- (B) -- (C) -- cycle;
\draw[red!60!black,fill=red!5,opacity=0.8] (D) -- (E) -- (F) -- (G) -- cycle;
\draw[red!60!black,fill=red!5,opacity=0.6] (C) -- (B) -- (F) -- (G) -- cycle;
\draw[red!60!black,fill=red!5,opacity=0.8] (A) -- (B) -- (F) -- (E) -- cycle;
\coordinate (O) at (4,0.75\Width,0.75\Height);
\coordinate (A) at (4,\Width,0.75\Height);
\coordinate (B) at (4,\Width,\Height);
\coordinate (C) at (4,0.75\Height,\Height);
\coordinate (G) at (4+0.75\Depth,0.75\Height,\Height);
\coordinate (D) at (4+0.75\Depth,0.75\Height,0.75\Height);
\coordinate (E) at (4+0.75\Depth,\Width,0.75\Depth);
\coordinate (F) at (4+0.75\Depth,\Width,\Height);
\draw[green!80!black,fill=green!10,opacity=0.6] (O) -- (C) -- (G) -- (D) -- cycle;
\draw[green!80!black,fill=green!10,opacity=0.6] (O) -- (A) -- (E) -- (D) -- cycle;
\draw[green!80!black,fill=green!10,opacity=0.6] (O) -- (A) -- (B) -- (C) -- cycle;
\draw[green!40!black,fill=green!10,opacity=0.2] (D) -- (E) -- (F) -- (G) -- cycle;
\draw[green!40!black,fill=green!10,opacity=0.2] (C) -- (B) -- (F) -- (G) -- cycle;
\draw[green!40!black,fill=green!10,opacity=0.2] (A) -- (B) -- (F) -- (E) -- cycle;
\draw (0.5+2,0.5,0) node {\scriptsize{\color{blue!100}$\dots$  }};
\draw (0.2+4,-1.3,0) node {\scriptsize{\color{blue}$N^{\frac{1}{3}}$  }};
\draw (0.2+4,-1,0) node[rotate = 0] {{\color{gray!75}$\underbrace{\hspace{2cm}}$}};
\draw (-0.5+4,1,2) node[rotate = 90] {\scriptsize{\color{blue}$N^{\frac{1}{3}}$}};
\draw (-0.2+4,1,2) node[rotate = 270] {{\color{gray!75}$\underbrace{\hspace{2cm}}$}};
\draw (2.5+4,0,1.2) node[rotate = 45] {\scriptsize{\color{blue} $N^{\frac{1}{3}}$}};
\draw (2.2+4,0,1.2) node[rotate = 45] {{\color{gray!75}$\underbrace{\hspace{1.1cm}}$}};
\coordinate (O) at (0,0,0);
\coordinate (A) at (0,\Width,0);
\coordinate (B) at (0,\Width,\Height);
\coordinate (C) at (0,0,\Height);
\coordinate (D) at (\Depth,0,0);
\coordinate (E) at (\Depth,\Width,0);
\coordinate (F) at (\Depth,\Width,\Height);
\coordinate (G) at (\Depth,0,\Height);
\draw[red!60!black,fill=red!5] (O) -- (C) -- (G) -- (D) -- cycle;
\draw[red!60!black,fill=red!5] (O) -- (A) -- (E) -- (D) -- cycle;
\draw[red!60!black,fill=red!5] (O) -- (A) -- (B) -- (C) -- cycle;
\draw[red!60!black,fill=red!5,opacity=0.8] (D) -- (E) -- (F) -- (G) -- cycle;
\draw[red!60!black,fill=red!5,opacity=0.6] (C) -- (B) -- (F) -- (G) -- cycle;
\draw[red!60!black,fill=red!5,opacity=0.8] (A) -- (B) -- (F) -- (E) -- cycle;
\coordinate (O) at (0,0.75\Width,0.75\Height);
\coordinate (A) at (0,\Width,0.75\Height);
\coordinate (B) at (0,\Width,\Height);
\coordinate (C) at (0,0.75\Height,\Height);
\coordinate (G) at (0.75\Depth,0.75\Height,\Height);
\coordinate (D) at (0.75\Depth,0.75\Height,0.75\Height);
\coordinate (E) at (0.75\Depth,\Width,0.75\Depth);
\coordinate (F) at (0.75\Depth,\Width,\Height);
\draw[green!80!black,fill=green!10,opacity=0.6] (O) -- (C) -- (G) -- (D) -- cycle;
\draw[green!80!black,fill=green!10,opacity=0.6] (O) -- (A) -- (E) -- (D) -- cycle;
\draw[green!80!black,fill=green!10,opacity=0.6] (O) -- (A) -- (B) -- (C) -- cycle;
\draw[green!40!black,fill=green!10,opacity=0.2] (D) -- (E) -- (F) -- (G) -- cycle;
\draw[green!40!black,fill=green!10,opacity=0.2] (C) -- (B) -- (F) -- (G) -- cycle;
\draw[green!40!black,fill=green!10,opacity=0.2] (A) -- (B) -- (F) -- (E) -- cycle;
\draw (0.2,-1.3,0) node {\scriptsize{\color{blue}$N^{\frac{1}{3}}$  }};
\draw (0.2,-1,0) node[rotate = 0] {{\color{gray!75}$\underbrace{\hspace{2cm}}$}};
\draw (-0.5,1,2) node[rotate = 90] {\scriptsize{\color{blue}$N^{\frac{1}{3}}$}};
\draw (-0.2,1,2) node[rotate = 270] {{\color{gray!75}$\underbrace{\hspace{2cm}}$}};
\draw (2.5,0,1.2) node[rotate = 45] {\scriptsize{\color{blue} $N^{\frac{1}{3}}$}};
\draw (2.2,0,1.2) node[rotate = 45] {{\color{gray!75}$\underbrace{\hspace{1.1cm}}$}};
\draw [draw=gray!75,thick,->] (-1,1.7,1) -- (0.85,1,1) node [below] {};
\draw (-2,1.8,1) node {\scriptsize{\color{red} Tensor $ \mathcal{M}_{\lrttei}^{(i)}[1,:,:,:] \in \mathbb{R}^{1 \times  N_i \times N_i \times N_i }$}};
\draw (-3.7,0.6,1) node {{\color{red}Tensor $\mathcal{M}_{\lrttei,max}\in\mathbb{R}^{N_b^2\times N^{\frac{1}{3}} \times N^{\frac{1}{3}} \times N^{\frac{1}{3}} }$}};
\draw (3,-1,2) node[rotate = 0] {{\color{gray!75}$\underbrace{\hspace{7cm}}$}};
\draw (3,-1.5,2) node[rotate = 0]
{\scriptsize{\color{blue} $N_b^2$}};
\end{tikzpicture}
    \caption{$\mathcal{M}_{\lrttei,max} \in \mathbb{R}^{N_b^2 \times N^{\frac{1}{3}} \times N^{\frac{1}{3}} \times N^{\frac{1}{3}} }$ is the tensorization of $\mathbf{M}_{\lrttei,max}  \in \mathbb{R}^{N_b^2 \times N  }$.}
    \label{fig:tensor_contraction}
\end{figure}
By doing so, we can extract $\mathcal{M}_{\lrttei}^{(i)} \in\mathbb{R}^{N_b^2\times N_i \times N_i \times N_i }$ tensors that we unfold back to matrices $\mathbf{M}_{\lrttei }^{(i)}\in\mathbb{R}^{N_b^2\times N_i^3}$ by mode-1 matricization  defined in \eqref{eq::mode}.
We can exploit the tensorized structure of the factorized long-range two-electron integral tensor in equation \eqref{eq::tensors}  to reduce the application cost of the product between the tensorized form $\otimes_{l=1}^{3}\mathbf{A}^{(i)} \in \mathbb{R}^{N_i^{3} \times N_i^{3}}$ and $\mathbf{M}_{\lrttei}^{(i)} \in \mathbb{R}^{N_{b}^{2} \times N_{i}^{3}}$ from $\mathcal{O}\left(N_i^{6}N_b^2\right)$ to $\mathcal{O}\left(N_i^{4}N_b^2\right)$. Given the Definition~\ref{eq::mode}, the product $\left(\stackrel{}{\otimes}_{l=1}^{3} \mathbf{A}^{(i)}\right)\mathbf{M}_{\lrttei}^{(i)^\top} $ can be  defined entry-wise by
\begin{eqnarray}\label{ig1}
   \left (  \left(\stackrel{}{\otimes}_{l=1}^{3} \mathbf{A}^{(i)}\right)\mathbf{M}_{\lrttei}^{(i)^\top}   \right ) (n,j) \nonumber
    &=&\sum_{m_1,m_2,m_3=1}^{N_i} \left(\prod_{l=1}^{3} \mathbf{A}^{(i)}(n_l, m_l)\right) \mathbf{M}_{\lrttei}^{(i)^\top} (m_1m_2m_3,j)
    \\&
=&\sum_{m_1,m_2,m_3=1}^{N_i} \left(\prod_{l=1}^{2}  \mathbf{A}^{(i)}(n_l, m_{l})\right)\left( \mathbf{A}^{(i)}\mathbf{M}^{(i)^\top}_{{\lrttei},(4)}\right)\left (n_3,m_1m_2j  \right ),
\end{eqnarray}
where $\mathbf{M}^{(i)^\top}_{{\lrttei},(4)}$ is the mode-4 matricization (see Definition~\ref{eq::mode}) of the fourth order tensor $ \mathcal{M}_{\lrttei}^{(i)}  \in \mathbb{R}^{N_b^2 \times N_i \times N_i \times N_i}$. From \eqref{ig1}, we notice that we need to perform three times the matrix-matrix products of sizes $N_i \times N_i$ and $N_i \times N_i^2N_b^2$, leading to an overall time complexity of $\mathcal{O}\left(3N_i^{4}N_b^2\right) \sim \mathcal{O}\left(N_i^{4}N_b^2\right)$. If we want to compute the whole tensor, we need to sum over $i \in \left \{ 1,\dots,N_{q1} \right \}$ which yields to a complexity of $\mathcal{O}\left(N^{\frac{4}{3}}N_b^2\right)$ with $N=\left (max(N_i)_{i \in \left \{ 1,\ldots,N_{q_1} \right \}}  \right )^3$.

Indeed, in practical applications the whole two-electron integrals tensor  does not need to be evaluated but  it is rather kept in its tensorized structure to benefit from fast matrix operations when applying it to vectors or matrices. We will discuss in more details an application case in  Section~\ref{sec::application}. An important point when implementing these tensor product evaluations is that the presented method can benefit from BLAS  operations  \cite{blas}. Indeed, \eqref{ig1} can be interpreted as the application of a sequence of products of permutation matrices and block-diagonal matrices (with the same blocks $\mathbf{A}^{(i)}$ along the diagonal) to $\mathbf{M}_{\lrttei,(k)}^{(i)}, k \in  \left\{2,3,4 \right\}$. Matrix-vector products with block-diagonal matrices of this form can be numerically reformulated as matrix-matrix products between one of these diagonal blocks and a matrix composed of the concatenation of subvectors of the original one \cite{igor}. Since matrix-matrix products can be performed more efficiently than matrix-vector products using BLAS routines (namely BLAS-3 instead of BLAS-2), this optimization results in  efficient implementations. In our case, we have even larger concatenation of subvectors because we apply these tensor products to matrices (not simply vectors), resulting in even better exploitation of BLAS-3 routines.

\section{Long-range TEI tensor factorization through Fast Multipoles Methods ({\fmm})}
\label{subsub::fmm}
In what follows, we recall briefly Fast Multipole Methods FMM and its application in our problem after demonstrating that our kernel is \textit{asymptotically smooth}. As  many methods taking advantage of tree space decomposition \cite{Barnes,kernel,GREENGARD1987325}, FMM rely on an important property of usual kernels. We discuss also the similarities and differences between LTEI-{\lrttei} and {\fmm} approaches to approximate $\mathcal{B}^{lr}$.

\begin{definition}[Definition~5.1 in \cite{chaillat:hal-01543919}]
\label{definition_as}
 A kernel $K(., .) \quad: \mathbb{R}^{3} \times \mathbb{R}^{3} \rightarrow \mathbb{R}$ is said to be asymptotically smooth if there exist two constants $c_{1}, c_{2}$ and a singularity degree $\sigma \in \mathbb{N}_{0}$ such that $\forall z \in\left\{\boldsymbol{x}_{l}, \boldsymbol{y}_{l}\right\} \in \mathbb{R} , \forall n \in \mathbb{N}_{0}, \forall \boldsymbol{x} \neq \boldsymbol{y}$,
$$
\left|\frac{\partial^n }{\partial z^n}K(\boldsymbol{x}, \boldsymbol{y})\right| \leq n ! c_{1}\left(c_{2}\left \| \boldsymbol{x}-\boldsymbol{y}    \right \|\right)^{-n-\sigma} .
$$
\end{definition}
Based on this property, efficient hierarchical schemes can be derived for the evaluation of $N$-body problems involving asymptotically smooth kernels.
\subsection{Fast Multipole Methods}
\label{sect_fmm}
Considering two	point clouds with $N_\mathbf{x}, N_\mathbf{Y}$ points, where we denote these clouds by  $ \left \{ \mathbf{x}_n \right \}_{n=1}^{N_\mathbf{X}}, \left \{ \mathbf{y}_n \right \}_{n=1}^{N_\mathbf{Y}} \subset \mathbb{R}^3$ (whose elements are referred to as 3D points or \textit{particles}), $q\hspace{0.1cm}:\hspace{0.1cm}\left \{ \mathbf{y}_n \right \}_{n=1}^{N_\mathbf{Y}}\rightarrow\mathbb{C}$ and an asymptotically smooth function $K\hspace{0.1cm}:\hspace{0.1cm}\left(\mathbb{R}^3\times\mathbb{R}^3\right)\backslash \{\mathbf{0}\}\rightarrow \mathbb{C}$, one may express the associated $N$-body problem as the computation of $p\hspace{0.1cm}:\hspace{0.1cm}\left \{ \mathbf{x}_n \right \}_{n=1}^{N_\mathbf{X}}\rightarrow \mathbb{C}$ such that
\begin{equation}
\label{eq_n_body}
	p(\boldsymbol{x}) := \sum_{\boldsymbol{y}\in \left \{ \mathbf{y}_n \right \}_{n=1}^{N_\mathbf{Y}}}K(\boldsymbol{x},\boldsymbol{y})q(\boldsymbol{y}).
\end{equation}
Computing $p$ naively requires $\mathcal{O}(N^2)$ floating point operations,	with   $N = max(N_\mathbf{X},N_\mathbf{Y})$. Thanks to hierarchical methods, such as \textit{hierarchical matrices} or \textit{Fast Multipole Methods} (FMM), this complexity can be reduced to $\mathcal{O}(N\hspace{0.1cm} logN)$ or even $\mathcal{O}(N)$ (but at the cost of an error we can control). These methods rely on decompositions of $\left \{ \mathbf{x}_n \right \}_{n=1}^{N_\mathbf{X}}$ and $\left \{ \mathbf{y}_n \right \}_{n=1}^{N_\mathbf{Y}}$ into groups of particles whose interaction can be efficiently performed through low-rank matrix approximations if their distance is sufficiently large compared to their radius. For non-oscillatory kernels $K$, FMMs are able to reach the $\mathcal{O}(N)$ complexity, so that they are attractive algorithm for efficiently solving $N$-body problems.

Among the different formulation of FMMs, we seek for particular features needed for our application case. Indeed, the method has to:
\begin{itemize}
    \item perform efficiently (actually in a linear time with respect to the number of points) on highly non-uniform point distributions, such as the three-dimensional Chebyshev grids,
    \item handle the kernel $K$ (which is non-standard kernel in the FMM community),
    \item be able to reach the precision required in realistic chemistry applications.
\end{itemize}

\subsection{Application to two-electron integrals (TEI)}
First, in order to exploit FMM on the two-electron integrals, one has to check that the underlying kernel is asymptotically smooth (see Definition~\ref{definition_as}). In our case, we want the FMM to act on  the long-range kernel  $K(\boldsymbol{x},\boldsymbol{y}), \boldsymbol{x},\boldsymbol{y} \in \mathbb{R}^3$ (see \eqref{eq::kernel}), which leads us to demonstrate the result of Proposition~ \ref{proposition_erfas}.
\begin{prop}
\label{proposition_erfas} $K(\boldsymbol{x},\boldsymbol{y})= \frac{erf(\omega \left \| \boldsymbol{x}-\boldsymbol{y}    \right \| ))}{ \left \| \boldsymbol{x}-\boldsymbol{y}    \right \| }, \boldsymbol{x},\boldsymbol{y} \in \mathbb{R}^3, 0 \leq \omega <  \infty  $ is asymptotically smooth.
\end{prop}
\begin{proof}
Given the function $K(\boldsymbol{x},\boldsymbol{y})= \frac{erf(\omega \left \| \boldsymbol{x}-\boldsymbol{y}    \right \| ))}{ \left \| \boldsymbol{x}-\boldsymbol{y}    \right \| }, \boldsymbol{x},\boldsymbol{y} \in \mathbb{R}^3, 0 \leq \omega <  \infty  $, we want to evaluate the function's partial derivative upper bound  with respect to $x_1 \in \mathbb{R}$ such that $\forall n \in \mathbb{N}_{0}, \forall \boldsymbol{x} \neq \boldsymbol{y}$, the $n$th derivative of $K(\boldsymbol{x},\boldsymbol{y})$ with respect to $x_1$ writes
\begin{eqnarray}
\label{eq::nth}
    \frac{\partial^n }{\partial x_1^n}K(\boldsymbol{x},\boldsymbol{y})&=&   \frac{\partial^n }{\partial x_1^n}(\frac{2}{\sqrt \pi}\int_{0}^{\omega}exp\left (-s^{2} \left \| \boldsymbol{x}-\boldsymbol{y}    \right \|^{2}  \right )ds)
    \\ &=&\frac{2^{n+1}}{\sqrt{\pi}}n!\int_{0}^{\omega}\sum_{k=0}^{\left [ \frac{n}{2} \right ]}\frac{(-1)^{n-2k}2^{-2k}(x_1-y_1)^{n-2k}}{k!(n-2k)!}s^{2n-2k}exp\left (-s^2 \left \| \boldsymbol{x}-\boldsymbol{y}    \right \|^2  \right )ds.
\end{eqnarray}
If $n$ is even, the term under the integral in   \eqref{eq::nth} is positive. Otherwise, it is either negative or positive. Therefore, \eqref{eq::nth} can be bounded by the absolute value of the $n$th derivative of the Coulomb potential that writes
\begin{equation}
\label{eq::inequality}
     \frac{\partial^n }{\partial x_1^n}\frac{1}{\left \| \boldsymbol{x}-\boldsymbol{y} \right \|}= \frac{2^{n+1}}{\sqrt{\pi}}n!\int_{0}^{\infty }\sum_{k=0}^{\left [ \frac{n}{2} \right ]} \frac{(-1)^{n-2k}2^{-2k}\left (x_1-y_1\right )^{n-2k}}{k!(n-2k)!}s^{2n-2k}exp\left (-s^2 \left \| \boldsymbol{x}-\boldsymbol{y} ^2  \right \|  \right )ds,
\end{equation}
and 
\begin{equation}
      \frac{\partial^n }{\partial x_1^n}K(\boldsymbol{x},\boldsymbol{y})\leq \left |\frac{\partial^n }{\partial x_1^n}\frac{1}{\left \| \boldsymbol{x}-\boldsymbol{y} \right \|}  \right |.
\end{equation}
Since $\frac{1}{\left \| \boldsymbol{x}-\boldsymbol{y} \right \|}$ is asymptotically smooth \cite{Hierarchical,Wolfgang}, this shows that $K(\boldsymbol{x},\boldsymbol{y})$ is also asymptotically smooth.\\
This proof applies for all the other directions.
\end{proof}
Hence, thanks to the asymptotically smooth behavior of $K$, FMM can be applied to this kernel and the far field contribution of the $N$-body problem can be efficiently approximated, especially by exploiting polynomial interpolation.
Similar to the previous sections,  we  consider the  finite six-dimensional integral $\mathcal{B}^{lr}(\mu,\nu,\kappa,\lambda)$ defined in \eqref{eq::compact} on a truncated computational box $[-b,b]^3 \times [-b,b]^3, b \in \mathbb{R}$ as follows
\begin{eqnarray}\label{TEI3}
      \mathcal{B}^{lr}(\mu,\nu,\kappa,\lambda)&=&\int_{[-b,b]^{3}}\int_{[-b,b]^{3}}g_{\mu\nu}(\boldsymbol{x})K(\boldsymbol{x},\boldsymbol{y}) g_{\kappa\lambda}(\boldsymbol{y})d\boldsymbol{x}d\boldsymbol{y}.
\end{eqnarray}
Instead of applying Gaussian quadrature rule on the kernel $K(\boldsymbol{x},\boldsymbol{y})$ as we did in the previous Section~\ref{sec::LRTTEI}, we use  Chebyshev polynomials evaluated in a six-dimensional Chebyshev grid,  the low-rank approximation of  $K(\boldsymbol{x},\boldsymbol{y})$   can be written, as explained in \cite{kernel}, as follows
\begin{equation}\label{sec}
    K(\boldsymbol{x},\boldsymbol{y})=\sum_{i=1}^{N}L(\boldsymbol{x}_i,\boldsymbol{x})\underbrace{\sum_{j=1}^{N}K(\boldsymbol{x}_i,\boldsymbol{y}_j)L(\boldsymbol{y}_j,\boldsymbol{y})}_{N\textit{-body problem as in Eq. \ref{eq_n_body}}},
\end{equation}
where $N$ is the total number of Chebyshev interpolation points (we use the same $N$ as the one introduced in Section~\ref{sec::LRTTEI}), $\boldsymbol{x}_i=(\boldsymbol{x}_{i_1},\boldsymbol{x}_{i_2},\boldsymbol{x}_{i_3})$ and $\boldsymbol{y}_i=(\boldsymbol{y}_{i_1},\boldsymbol{y}_{i_2},\boldsymbol{y}_{i_3})$, for $i \in  \left \{ 1,2,\dots,N  \right \}$, are 3-vectors of Chebyshev points with $i_{l}, j_{l} \in\{1, \ldots, N \},l \in \left \{ 1,2,3 \right \}$. We also have
\begin{equation}
    L (\boldsymbol{x}_i, \boldsymbol{x})=L^{(1)} \left(\boldsymbol{x}_{i_1}, \boldsymbol{x}_{1}\right) L^{(2)}  \left(\boldsymbol{x}_{i_2}, \boldsymbol{x}_{2}\right) L^{(3)} \left(\boldsymbol{x}_{i_3}, \boldsymbol{x}_{3}\right).
\end{equation}
\begin{equation}
\label{eq::cheblagr}
         L^{(l)} \left(\boldsymbol{x}_{i_l}, \boldsymbol{x}_{l}\right) =\frac{1}{N^{\frac{1}{3}}}+\frac{2}{N^{\frac{1}{3}}}\sum_{k=2}^{N^{\frac{1}{3}}} T_k(\boldsymbol{x}_{i_l})T_k(\boldsymbol{x}_{l}),~l \in \left \{ 1,2,3 \right \}.
\end{equation}
One may notice that the equation \eqref{eq::cheblagr} appears as a simple reformulation of the interpolation presented in Definition~\eqref{def::Chebyspoints}, combining the equation~\eqref{eq::chebpoly} and the equation~\eqref{eq::chebcoeff}. The important point here is that we want the kernel to explicitly appear (evaluated on Chebyshev interpolation nodes) in the expression, so that a FMM algorithm can be derived, following \cite{kernel,igor}.    
Chebyshev polynomials are used here as interpolation basis and were already defined in Definition~\ref{def::Chebyspoints}. The long-range two-elctron integrals in \eqref{eq::compact} can be written as follows
\begin{eqnarray}\label{eq_interp_FMM}
      \mathcal{B}_{\fmm}^{lr}(\mu,\nu,\kappa,\lambda)&=&\sum_{i=1}^{N }\underbrace{ \int_{[-b,b]^3}g_{\mu\nu}(\boldsymbol{x})L (\boldsymbol{x}_i,\boldsymbol{x})d\boldsymbol{x}}_{\mathbf{Z}_{\mu\nu}(1,\boldsymbol{x}_i)}
\left (\sum_{j=1}^{N }K(\boldsymbol{x}_i,\boldsymbol{y}_j)
\underbrace{\int_{[-b,b]^3}g_{\kappa\lambda}(\boldsymbol{y})L(\boldsymbol{y}_j,\boldsymbol{y})d\boldsymbol{y}}_{\mathbf{Z}_{\kappa\lambda}(\boldsymbol{y}_j,1)}  \right )
\\&=&\sum_{i =1}^{N}\mathbf{Z}_{\mu\nu}(1,\boldsymbol{x}_i)\left ( \sum_{ j=1}^{N} \nonumber K(\boldsymbol{x}_i,\boldsymbol{y}_j)\mathbf{Z}_{\kappa\lambda}(\boldsymbol{y}_j,1) \right ).
\end{eqnarray}
Equation \eqref{eq_interp_FMM} can be written in matrix formulation as follows for fixed   $\mu,\nu,\kappa,\lambda \in   \left \{ 1,\ldots,N_b \right \}$
\begin{equation}\label{factoz}
     \mathcal{B}_{\fmm}^{lr}(\mu,\nu,\kappa,\lambda)= \mathbf{Z}_{\mu\nu}\mathbf{K}\mathbf{Z}_{\kappa\lambda}^\top,
\mathbf{Z}_{\mu\nu},\mathbf{Z}_{\kappa\lambda}  \in \mathbb{R}^{1 \times N},\mathbf{K} \in \mathbb{R}^{N \times N}
\end{equation}
with $\mathbf{K}(\boldsymbol{x}_i,\boldsymbol{y}_j)= \frac{erf(\omega \left \| \boldsymbol{x}_i-\boldsymbol{y}_j    \right \| ))}{ \left \| \boldsymbol{x}_i-\boldsymbol{y}_j    \right \| }, i,j \in  \{1, \ldots, N\}.$
The last term into parenthesis in \eqref{eq_interp_FMM}
corresponds to an $N$-body problem as in Equation (\ref{eq_n_body}), whose evaluation can be performed in $\mathcal{O}(N )$ FLOPS using FMM. One may notice that the FMM accuracy can be chosen accordingly to the interpolation error in equation (\ref{eq_interp_FMM}). For all $\mu,\nu,\kappa,\lambda \in   \left \{ 1,\ldots,N_b \right \}$, the factorized representation of the mode-(1,2) matricization of the  fourth-order tensor $\mathcal{B}_{\fmm}^{lr}$ \eqref{eq_interp_FMM} is then given by 
\begin{equation}\label{fmimi}
    \mathbf{B}_{\fmm}^{lr}=\mathbf{M}_{FMM}\mathbf{K}\mathbf{M}_{FMM}^\top \in \mathbb{R}^{ N_b^2 \times N_b^2 } ,~\mathbf{M}_{FMM} \in \mathbb{R}^{ N_b^2 \times N  },
\end{equation}
with $\mathbf{M}_{FMM}[ \mu\nu,:]=\mathbf{Z}_{\mu\nu} \in \mathbb{R}^{1 \times N}$, for $\mu,\nu \in  \left \{1,\cdots,N_b\right \}$. 
Hence, the entire computation of \eqref{fmimi}  requires the application of the FMM method to  each column  of $\mathbf{M}_{FMM}$, the overall evaluation complexity of FMM becomes $\mathcal{O}(N\times N_b^2)$ to compute $\mathbf{K}\mathbf{M}_{FMM}^\top$. 
\\
\begin{remarknn}
The FMM formulation we opted for relies on precomputations (at a linear cost with respect to the number of particles) for the construction of low-rank approximations (see Section~\ref{sect_fmm}) that depends only on the particle distribution. Because the interpolation points are the same for each $\mathbf{Z}_{\kappa\lambda}({\boldsymbol{y}_j},1)$, our particle distributions do not change, so these precomputations can be performed only once and reused for each FMM application.
\end{remarknn}
\subsection{Similarities and differences between {LTEI-\lrttei} and {\fmm} approaches}
In table \ref{table::sim} we summarize the approximated expressions of \eqref{eq::compact} obtained  through {LTEI-\lrttei} and  {\fmm} approaches.
\begin{table}[H]
\caption {Factorization of TEI } \label{table::sim} 
\begin{center}
\scalebox{0.8}{%
\begin{tabular}{|c|c|c|} 
  \hline
Approaches & {LTEI-\lrttei} & {\fmm} \\
\hline
Distribution & $N$ Chebyshev points & $N$ Chebyshev points\\
  \hline
Entry-wise evaluation: & 
      $\mathcal{B}_{LTEI-\lrttei}^{lr}(\mu,\nu,\kappa,\lambda)
      :=\frac{\omega}{\sqrt \pi} \sum_{i=1}^{N_{q_1}} \left (w_{i} \sum_{  j =1}^{I_{\mu\nu} }\sum_{  j' =1}^{I_{\kappa\lambda}}c_jc_{j'} \mathbf{F}^{(i)}(j,j')  \right ).$  & $\mathcal{B}_{\fmm}^{lr}(\mu,\nu,\kappa,\lambda):= \mathbf{Z}_{\mu\nu}\mathbf{K}\mathbf{Z}_{\kappa\lambda}^\top.$ \\ 
 \hline
 Factorized representation: & 
       $\mathbf{B}_{LTEI-\lrttei}^{lr}:=\frac{\omega}{\sqrt \pi}\sum_{i=1}^{N_{q_1}}w_i\mathbf{M}_{\lrttei}^{(i)}\left ({\otimes}_{l=1}^{3} \mathbf{A}^{(i)}  \right )\mathbf{M}_{\lrttei}^{(i) \top} \in \mathbb{R}^{ N_b^2 \times N_b^2 }.$
   & 
$\mathbf{B}_{\fmm}^{lr}:=\mathbf{M}_{FMM}\mathbf{K}\mathbf{M}_{FMM}^\top \in \mathbb{R}^{ N_b^2 \times N_b^2 }.$
 \\
  \hline
\end{tabular}}
\end{center}
\end{table}
We discuss here the differences and similarities between both approaches. On one hand, for {\lrttei} approach,  we start by applying a change of variable to the long-range kernel $K(\boldsymbol{x},\boldsymbol{y})$ \eqref{eq::kernel} in order to remove the term $ \frac{1}{ \left \| \boldsymbol{x}-\boldsymbol{y}    \right \|}$, then we apply one-dimensional Gaussian quadrature (see \eqref{TEI2}) with $N_{q_1}$ quadrature points. In addition to that, we apply two-dimensional Chebyshev interpolation which yields to obtaining a tensorized form obtained in \eqref{approx}, \eqref{eq::tensors}. Thus,  we need to evaluate $\mathbf{M}_{\lrttei,max} \in \mathbb{R}^{N_b^2 \times N }$ which involves the evaluation of one-dimensional integrals over $[-b,b]$. On the other hand,
when applying interpolation directly on the original kernel $K$, one ends up with a three dimensional N-body problem that can be efficiently handled using FMM approach.
Thus, we need to compute  $\mathbf{M}_{FMM} \in \mathbb{R}^{ N_b^2 \times N}$ which involves also the evaluation of one-dimensional integrals over $[-b,b]$. The similarities between both approaches consist in  employing Chebyshev interpolation with the same total number of interpolation points $N$.
\begin{remarknn}
One may mention that for low level optimisations (such as explicit formula for the polynomials or fast FFT-based assembling of the interpolation coefficients), we opted for slightly different interpolation nodes in the two methods. Indeed, \textit{Gauss-Chebyshev-Lobatto} nodes are used for {LTEI-\lrttei} method while \textit{Chebyshev} nodes are used for {\fmm}. These last points are defined as (showing only $\boldsymbol{x}_{i_ l}$ expression)
\begin{equation}
  \boldsymbol{x}_{i_ l}=\cos \left(\frac{2 k-1}{2 N^{\frac{1}{3}}} \pi\right), k \in \left \{ 1,\cdots,N^{\frac{1}{3}}\right \}, l \in \left \{ 1,2,3 \right \}.
\end{equation}
However, for both cases, the same number of interpolation nodes  is considered for a given targeted precision, $N^{\frac{1}{3}}$ per direction,  so that this detail does not impact the complexity estimates and the comparison between them.
\end{remarknn}

\section{Application to electronic structure calculations}
\label{sec::application}
We describe in what follows an application case for the two-electron integrals tensor using LTEI-{\lrttei} as well as {\fmm}. In quantum chemistry, one of the main steps in many methods is the construction of the Coulomb matrix \cite{doi:10.1021/ct501128u,doi:10.1137/19M1252855,Resolutions}. We define in the following the long-range Coulomb matrix in the molecular orbital basis $\phi_i$ that are represented (approximately) as \cite{hamiltonian}
\begin{equation}
\label{eq::LCAO}
    \phi_{i}=\sum_{\mu=1}^{N_{b}}q_{i\mu}g_{\mu}, i \in \left \{ 1,\cdots,N_{orb} \right \},
\end{equation}
with $q_{i\mu}$ being   the coefficients  of the linear combinations over the basis functions $\left\{g_{\mu}\right\}_{1 \leq \mu \leq N_{b}}$. In this molecular orbital  basis, the Coulomb  long-range integral reads
\begin{eqnarray}
\label{eq::coulombos}
    \mathbf{J}^{lr}(i,j)&=&\int_{\mathbb{R}^{3}}\int_{\mathbb{R}^{3}}K(\boldsymbol{x},\boldsymbol{y})  \sum_{i=1}^{N_{orb}}\left | \phi_{i}(\boldsymbol{x}) \right |^{2} \sum_{j=1}^{N_{orb}}\left | \phi_{j}(\boldsymbol{y}) \right |^{2} d\boldsymbol{x}d\boldsymbol{y}
    \\&=& \sum_{\mu,\nu,\kappa,\lambda=1}^{N_{b}}\sum_{i,j=1}^{N_{orb}}q_{i\mu}q_{i\nu}\left (\int_{\mathbb{R}^{3}}\int_{\mathbb{R}^{3}}K(\boldsymbol{x},\boldsymbol{y})  g_{\mu}(\boldsymbol{x})g_{\nu}(\boldsymbol{x})g_{\kappa}(\boldsymbol{y})g_{\lambda}(\boldsymbol{y}) q_{j\kappa}q_{j\lambda}d\boldsymbol{x}d\boldsymbol{y}  \right ).
\end{eqnarray}
Let us define the rectangular matrix $\mathbf{Q}$   $\in \mathbb{R}^{ N_{orb} \times  N_{b}^2}$ with entries $Q(i,\mu\nu)=q_{i\mu}q_{i\nu}$ such that  $\mathbf{J}^{lr}$ 
 writes in matrix notation as
\begin{equation}\label{Coul}
    \mathbf{J}^{lr}=\mathbf{Q}\mathbf{B}^{lr}\mathbf{Q}^\top  \in \mathbb{R}^{N_{orb} \times N_{orb}},
\end{equation}
where $\mathbf{B}^{lr} \in \mathbb{R}^{N_b^2 \times N_b^2}$ is the   mode-(1,2) matricization of $\mathcal{B}^{lr} $.  A naive approach to evaluate \eqref{Coul}, given $\omega$, the matrix $\mathbf{Q} \in \mathbb{R}^{N_{orb} \times  N_b^2}$, and the long-range two-electron integrals $\mathbf{B}^{lr}$, is to first compute the matrix product  $\mathbf{B}^{lr}\mathbf{Q}^\top$ and then perform $\mathbf{Q}\left (\mathbf{B}^{lr}\mathbf{Q}^\top  \right )$. The last has an arithmetic cost of $\mathcal{O}(N_b^4N_{orb})$. Given a truncated computational box $ \left [ -b,b \right ]^3$, one can use the factorized structure $\mathbf{B}^{lr}_{LTEI-\lrttei}$ defined in \eqref{eq::tensors} or  $\mathbf{B}^{(lr)}_{\fmm}$ defined in \eqref{fmimi} to evaluate \eqref{Coul} efficiently. Given the two approximation approaches ({LTEI-\lrttei} and \fmm), we arrive at the following matrix representations
\begin{equation}
\label{eq::coulapprox}
    \mathbf{J}_{LTEI-\lrttei}^{lr}=\frac{\omega}{\sqrt \pi}\sum_{i=1}^{N_{q_1}}w_i\left( \mathbf{Q}\mathbf{M}_{\lrttei}^{(i)} \right)\left ({\otimes}_{l=1}^{3} \mathbf{A}^{(i)}  \right ) \left( \mathbf{Q}\mathbf{M}_{\lrttei}^{(i)} \right)^\top \text{ and } \mathbf{J}_{\fmm}^{lr}=\left( \mathbf{Q}\mathbf{M}_{FMM} \right)\mathbf{K}\left( \mathbf{Q}\mathbf{M}_{FMM} \right)^\top.
\end{equation}
We present in Table~\ref{table::tabb} an overview of  the storage complexities obtained  through {LTEI-\lrttei} method as well as {\fmm} method to evaluate  entries of the  long-range two-electron integrals tensor and its application to evaluate the long-range Coulomb matrix defined in  \eqref{Coul}.
\begin {table}[H]
\caption {Storage complexity comparison  } \label{tab:title} 
 \begin{center}
\scalebox{1.}{%
\begin{tabular}{|l|c|r|r|r|r|r|}
  \hline
 & {LTEI-\lrttei} & {\fmm}\\
 \hline
Element-wise TEI  & $\mathcal{O}(N^{\frac{1}{3}}N_{q_1}(N^{\frac{1}{3}}+I_{\mu\nu}+I_{\kappa\lambda}))$ & $\mathcal{O}(N )$\\
\hline
Application (\ref{Coul}) & $\mathcal{O}(N^{\frac{2}{3}}(N_{orb}N^{\frac{1}{3}}+N_{q_1}))$ & $\mathcal{O}(N(1+N_{orb}))$\\
\hline
\end{tabular}
}
\label{table::tabb}
\end{center}
\end {table}
The storage complexity of the element-wise evaluation for {\fmm} is a consequence of \eqref{factoz}, i.e. linear with regard to the number of interpolation points $N$. The storage complexities for the evaluation of \eqref{Coul} are obtained as follows. For LTEI-{\lrttei} approach 
\begin{enumerate}
    \item Instead of forming all matrices $\mathbf{Q}\mathbf{M^{(i)}_{\lrttei }}$ for $i  \in\left \{ 1..N_{q_1} \right \}$, we form only (as explained in Section~\ref{sec::fasteval}) $\mathbf{Q}\mathbf{M_{\lrttei,max}}$  that requires $\mathcal{O}(N_{orb}N)$ storage.
    \item As discussed before, we keep $\left ({\otimes}_{l=1}^{3} \mathbf{A}^{(i)}  \right )$ in tensorized form. Hence, forming all coefficient matrices   $\mathbf{A}^{(i)}$ of size $N_i \times N_i$, for $i  \in\left \{ 1..N_{q_1} \right \}$  requires  $\mathcal{O}(\sum_{i=1}^{N_{q_1}}N_i^2) \sim   \mathcal{O}(N_{q_1}N^{\frac{2}{3}})$ storage. 
\end{enumerate}
So in total, the storage complexity is $\mathcal{O}(N^{\frac{2}{3}}(N_{orb}N^{\frac{1}{3}}+N_{q_1}))$. For {\fmm} approach
\begin{enumerate}
    \item Forming $\mathbf{Q}\mathbf{M_{FMM}}$  requires  $\mathcal{O}(N_{orb}\times N)$ of storage.
    \item Forming $\mathbf{K}$ requires $\mathcal{O}(N)$ of storage. 
\end{enumerate}
So in total, the storage complexity is $\mathcal{O}(N(1+N_{orb}))$. According to Table.\ref{table::tabb}, the storage demand for this evaluation seems lower (in order) for {LTEI-\lrttei} compared to \fmm. However, we cannot  conclude on the best method in terms of  storage complexity since  $N_{q_1}$ and $N^{\frac{1}{3}}$ depend on the value of $\omega$ and the chosen computational box $[-b,b]^3$. This motivates numerical comparisons between the two approaches for different parameters (see Section~\ref{sec::numerical}).

\section{Compression techniques for the factorized long-range TEI tensor}
\label{sub::LR}
One of the main precomputation steps required  to obtain  the factorized representation of $\mathbf{B}^{lr}$ is based on the evaluation of $\mathbf{M}_{\lrttei,max} \in \mathbb{R}^{N_b^2 \times N  }$ (resp.  $\mathbf{M}_{FMM} \in \mathbb{R}^{ N_b^2 \times N}$) matrix. This step tends to be  expensive  in terms of both computational and memory requirements for molecules of moderate size, as we consider in our experiments.  In this section we address this problem by discussing    different approaches to compress $\mathbf{M}_{\lrttei,max}$, some of which can be applied to $\mathbf{M}_{FMM}$.
\subsection{Compression by using low-rank methods}
\label{KR}
In many  cases, the matrix $\mathbf{M}_{\lrttei,max}$  is  numerically low-rank  as we will discuss in the numerical experiments section (see Figure~\ref{fig:low_rank}). It is possible to reduce its dimensions by exploiting its low rank structure. We recall the \textit{screening} technique \cite{meth} which consists in   simply discarding "negligible" pairs of Gaussian type basis functions as  explained in \ref{subsub:pres}. Low rank approximation methods such as truncated SVD \cite{chaillat} can be also applied directly on $\mathbf{M}_{\lrttei,max}$ to further reduce its dimensions. We introduce in this section a different compression method that exploits the  khatri-rao products and associated properties.  Let $\mathbf{\tilde{W}}^{(i,l)} \in \mathbb{R}^{N_b^2I_{\mu\nu,max} \times N_i}$ be defined by
\begin{equation}
\mathbf{\tilde{W}}^{(i,l)}=\left[\phantom{\begin{matrix}\mathbf{\tilde{W}}^{(i,l)}_{\mu\nu}\\
\mathbf{0}\end{matrix}}
\right. \hspace{-1em}
 \begin{matrix}
\mathbf{\tilde{W}}^{(i,l)}_{\mu\nu}\\
\mathbf{0}
\end{matrix} 
\hspace{-1.5em}
\left.\phantom{\begin{matrix}\mathbf{\tilde{W}}^{(i,l)}_{\mu\nu}\\
\mathbf{0} \end{matrix}}\right]
\begin{tabular}{l}
$\left.\lefteqn{\phantom{\begin{matrix}\mathbf{W}^{(i,l)}_{\mu\nu} \end{matrix}}}\right\}I_{\mu\nu} \times N_i$\\
$\left.\lefteqn{\phantom{\begin{matrix}0 \end{matrix}}} \right\} \left (I_{\mu\nu,max}-I_{\mu\nu}  \right ) \times N_i$, $I_{\mu\nu,max}=max(I_{\mu\nu})_{1\leq \mu,\nu\leq N_b}.$
\end{tabular}
\end{equation}
Its low rank $R_{i,l}$  approximation  can be written as:
\begin{equation}
\label{eq::SVD}
    \mathbf{\tilde{W}}^{(i,l)} \approx \mathbf{U}^{(i,l)}\mathbf{V}^{(i,l)\top},
\end{equation}
where  $\mathbf{U}^{(i,l)} \in \mathbb{R}^{N_b^2I_{\mu\nu,max} \times R_{i,l}}$ and $\mathbf{V}^{(i,l)\top} \in \mathbb{R}^{R_{i,l} \times N_i}$. Given the decomposition \eqref{eq::SVD}, Proposition~\ref{prop::relations} is used to obtain the following expression
\begin{small}
\begin{align}
     \left ( \diamond_{l=1}^3(\mathbf{\tilde{W}}^{(i,l)}) \right )\left ({\otimes}_{l=1}^{3} \mathbf{A}^{(i)}  \right )\left ( \ast_{l=1}^3(\mathbf{\tilde{W}}^{(i,l)})^\top  \right )&= \left ( \diamond_{l=1}^3(\mathbf{U}^{(i,l)}\mathbf{V}^{(i,l)}) \right )\left (\mathbf{A}^{(i)}\otimes \mathbf{A}^{(i)}\otimes \mathbf{A}^{(i)}  \right )\left ( \ast_{l=1}^3(\mathbf{U}^{(i,l)}\mathbf{V}^{(i,l)})^\top  \right )\\&= \left ( \diamond_{l=1}^3\mathbf{U}^{(i,l)} \right )(\otimes_{l=1}^3\mathbf{V}^{(i,l)})\left (\mathbf{A}^{(i)}\otimes \mathbf{A}^{(i)}\otimes \mathbf{A}^{(i)}  \right )(\otimes_{l=1}^3\mathbf{V}^{(i,l)\top} )\left ( \ast_{l=1}^3(\mathbf{U}^{(i,l)})^\top  \right ) \\ &=\left ( \diamond_{l=1}^3\mathbf{U}^{(i,l)} \right )\otimes_{l=1}^3(\mathbf{V}^{(i,l)}\mathbf{A}^{(i)} \mathbf{V}^{(i,l)\top} )\left ( \ast_{l=1}^3(\mathbf{U}^{(i,l)})^\top  \right ).
\end{align}
\end{small}
 By replacing the low rank approximation of the matrix $\mathbf{\tilde{W}}^{(i,l)}$ in the expression of $\mathbf{B}^{lr}_{LTEI-\lrttei}$ in equation~\eqref{eq::tensors}, we obtain 
 \begin{equation}
\label{eq::approxepsilon}
    \mathbf{B}_{LTEI-\lrttei}^{lr} \approx \frac{\omega}{\sqrt \pi}\sum_{i=1}^{N_{q_1}}w_i\mathbf{\tilde{U}}^{(i)} \otimes_{l=1}^3(\mathbf{V}^{(i,l)}\mathbf{A}^{(i)} \mathbf{V}^{(i,l)\top}) \mathbf{\tilde{U}}^{(i)\top},
\end{equation}
where $\mathbf{\tilde{U}}^{(i)}=\sum_{j=1}^{I_{\mu\nu,max}}c_j\mathcal{U}^{(i)}[j,:,:] \in \mathbb{R}^{N_b^2 \times \prod_{l=1}^{3}R_{i,l}}$ with $\mathcal{U}^{(i)} \in \mathbb{R}^{I_{\mu\nu,max} \times N_b^2 \times \prod_{l=1}^{3}R_{i,l}}$ the tensorization of $\left ( \diamond_{l=1}^3\mathbf{U}^{(i,l)} \right ) \in \mathbb{R}^{I_{\mu\nu,max}N_b^2 \times \prod_{l=1}^{3}R_{i,l}}$.
In practice, we compute only   the matrix $\mathbf{\tilde{U}}^{(i)} \in \mathbb{R}^{N_b^2 \times max(\prod_{l=1}^{3}R_{i,l})_{i \in  \left\{ 1,\cdots,N_{q_1} \right\}}}$  with the maximum rank $max(\prod_{l=1}^{3}R_{i,l})_{i \in  \left\{ 1,\cdots,N_{q_1} \right\}}$ as   discussed in Section~\ref{sec::fasteval}.
\subsection{Adaptive approach  for the choice of the integration domain $[-b,b]$}
\label{subsub::adap}
We discuss now an adaptive approach for the choice of the integration domain $[-b,b]$. For each pair of Gaussian functions, we identify its numerical support $[-b,b]$. We cluster together these numerical supports to obtain overall $N_{partitions}$ supports, $\left \{ [-b_s,b_s] \right \}_{s \in \left \{1,\cdots,N_{partitions} \right \}}$. For each pair of Gaussian functions $(\mu,\nu), \mu,\nu \in  \left \{ 1,\cdots,N_b \right \}$, we proceed as follows: Given  the general Gaussian product rule 
(Definition~\ref{def::Gaussrule}), the product of two primitive Gaussian type functions $g_{\mu\nu}^{(j)}(\boldsymbol{x}_{l})= g_{\mu}^{(j_1)}(\boldsymbol{x}_l)g_{\nu}^{(j_2)}(\boldsymbol{x}_l)$ is
\begin{equation}\label{gauss_prod}
   g_{\mu\nu}^{(j)}(\boldsymbol{x}_{l})= g_{\mu}^{(j_1)}(\boldsymbol{x}_l)g_{\nu}^{(j_2)}(\boldsymbol{x}_l)=(\boldsymbol{x}_l-\mathbf{r}_{l})^{p_{\mu_l}}
(\boldsymbol{x}_l-\mathbf{r}'_{l})^{p_{\nu_l}}exp\left (-\frac{\mu_{j_1}\nu_{j_2}}{\mu_{j_1}+\nu_{j_2}} (\mathbf{r}_l-\mathbf{r}'_l)^{2}  \right ) \sigma_{\mu_{j_1}\nu_{j_2}}(\boldsymbol{x}_l),
\end{equation}
where 
\begin{equation}
    \sigma_{\mu_{j_1}\nu_{j_2}}(\boldsymbol{x}_l)
=exp\left (-(\mu_{j_1}+\nu_{j_2}) (\boldsymbol{x}_l-\frac{\mu_{j_1} \mathbf{r}_l+\nu_{j_2} \mathbf{r}'_l}{\mu_{j_1}+\nu_{j_2}})^{2}  \right ), j_1 \in \left \{ 1,\cdots,I_{\mu} \right \}, j_2 \left \{ 1,\cdots,I_{\nu} \right \},
\end{equation}
and  (see \eqref{eq::primi})
\begin{equation}
    g_{\mu}^{(j_1)}(\boldsymbol{x}_l)=(\boldsymbol{x}_{l}-\mathbf{r}_{l})^{p_{\mu_l}}exp\left (-\mu_{j_1} (\boldsymbol{x}_l-\mathbf{r}_l)^{2}  \right ) \text{ and }  g_{\nu}^{(j_2)}(\boldsymbol{x}_l)=(\boldsymbol{x}_l-\mathbf{r}'_{l})^{p_{\nu_l}} exp\left (-\nu_{j_2} (\boldsymbol{x}_l-\mathbf{r}'_l)^{2}  \right ),
\end{equation}
with $j=(j_1,j_2) \in \left \{ 1,..,I_{\mu\nu} \right \}$, $I_{\mu\nu}=I_{\mu}I_{\nu}$,   l $\in \left \{1,2,3\right \}$, $\mu,\nu \in \left \{1,..,N_b\right \}.$ The numerical support $[-b,b]$ is chosen according to  a cutoff threshold $\tau_{adaptive} >0$ such that 
\begin{equation}
  \sigma_{\mu_{j_1}\nu_{j_2}}(\boldsymbol{x}_l) \leqslant  \tau_{adaptive}, l \in \left \{1,2,3  \right \}.
\end{equation}
To illustrate this adaptive approach, for a given pair $(\mu,\nu)$,
   we represent  in Figure~\ref{fig:cluster} (left) the exponential terms
   $\sigma_{\mu_{j_1}\nu_{j_2}}(\boldsymbol{x}_1)$ in the expression  \eqref{gauss_prod} for  $j \in \left \{ 1,\cdots,I_{\mu\nu} \right \}$ with respect to the first direction (l=1). The exponential decay of these functions enables us to limit the range of the numerical grid according to a chosen  threshold $\tau_{adaptive}$. Through this adaptive technique, Figure~\ref{fig:cluster} (right)  illustrates the distribution of the numerical support (dimension $b$). Each bar represents the percentage of Gaussian function pairs $(\mu,\nu)$ associated to the exponential terms  $\sigma_{\mu_{j_1}\nu_{j_2}}(\boldsymbol{x}_l)$ lying in the range $[-b,b]$. It is showed that the distribution depends on the molecule choice as well as the number of basis functions $N_b$.
 \begin{figure}[H]
\centering
\begin{subfigure}{.5\textwidth}
  \centering
  \includegraphics[scale=0.08]{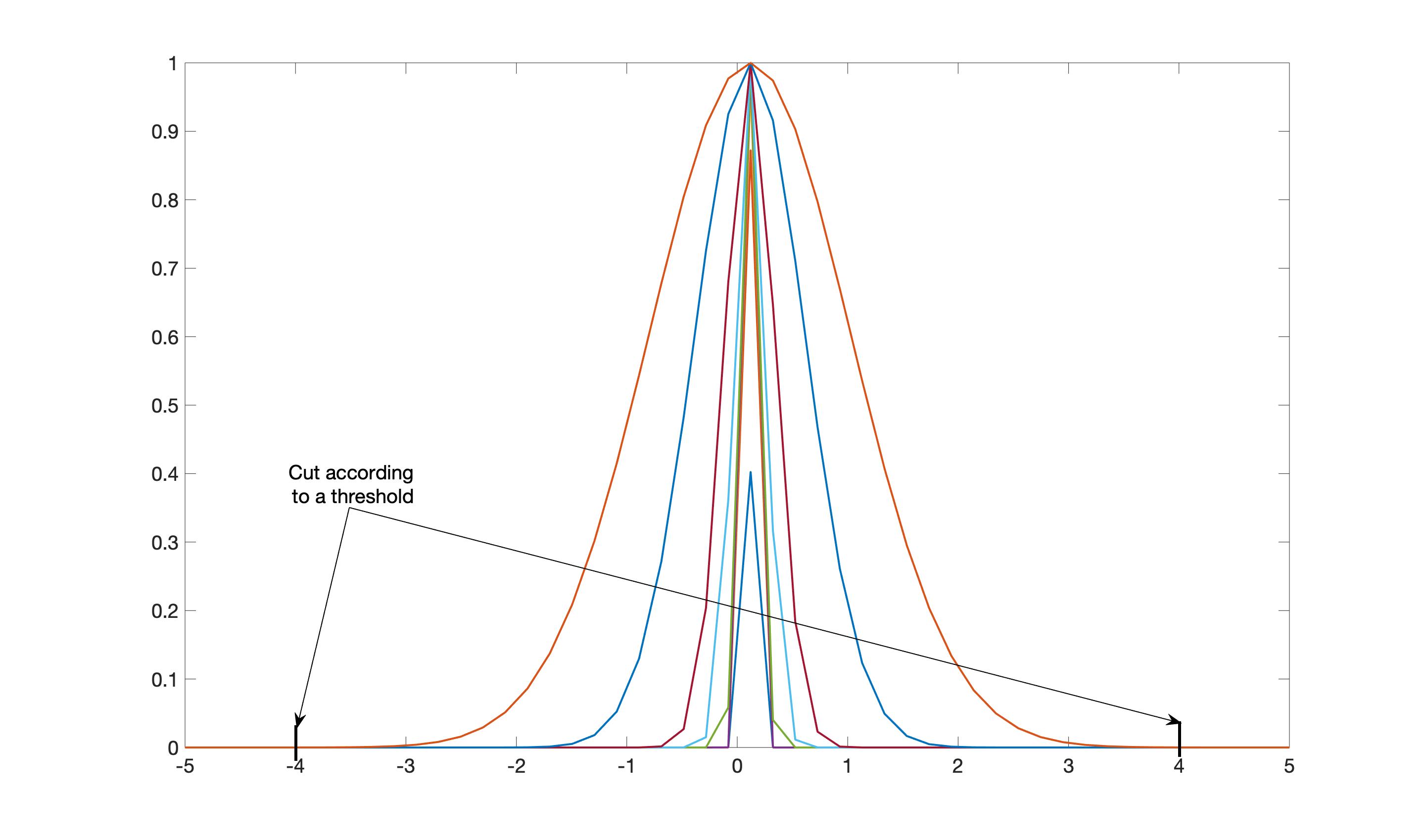}
\end{subfigure}%
\begin{subfigure}{0.5 \textwidth}
  \centering
  \includegraphics[scale=0.08]{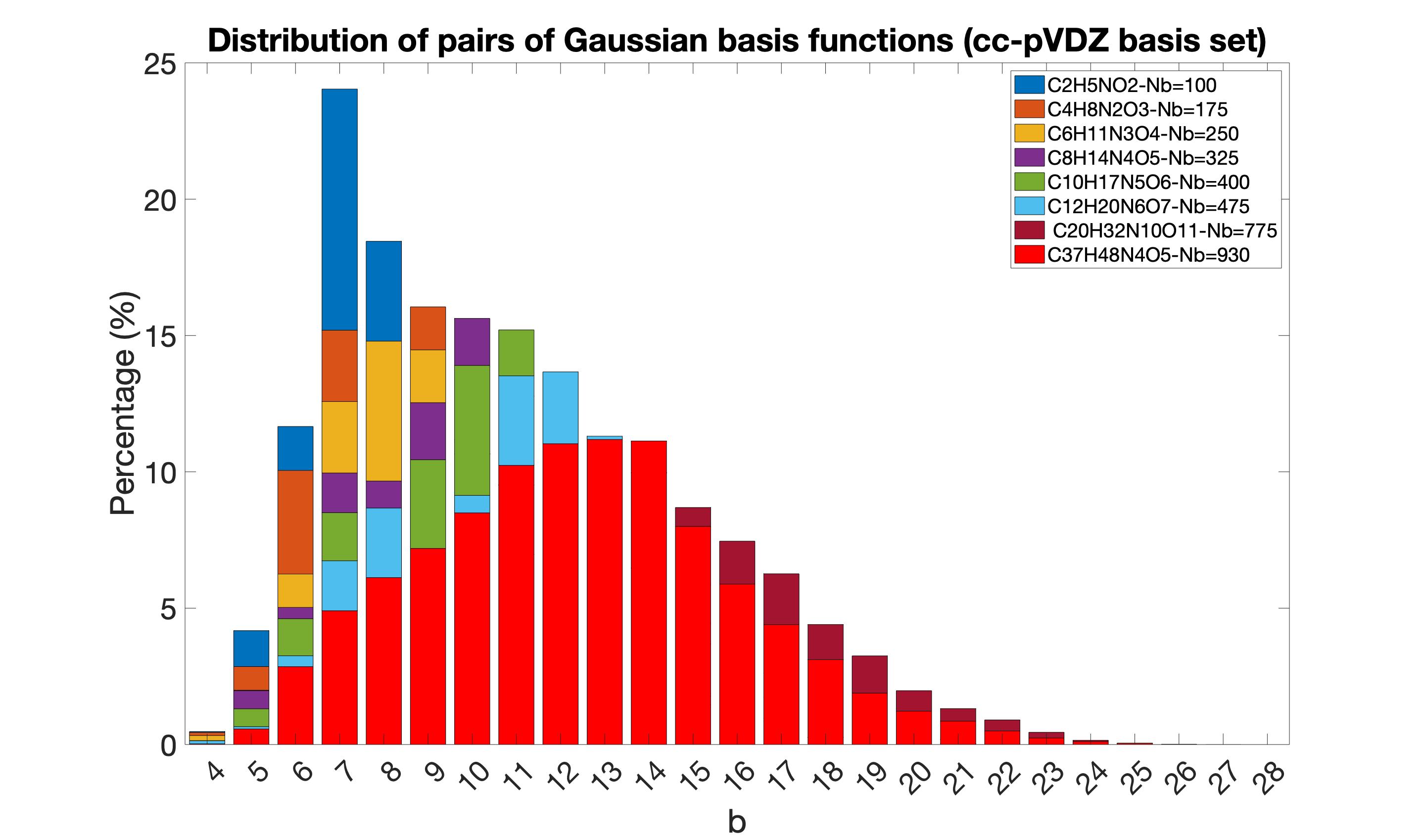}
\end{subfigure}
\caption{(Left) Identifying numerical supports of different pairs of Gaussian functions. Each color in the plot represents the   exponential term  $\sigma_{\mu_{j_1}\nu_{j_2}}(\boldsymbol{x}_1)$. Here the selected numerical support is $[-4,4]$. (Right) Distribution of numerical supports $[-b,b]$ for a given threshold $\tau_{adaptive}=10^{-20}$, the x-axis shows the dimension b of the box, the  y-axis shows the percentage of the Gaussian function pairs. }
\label{fig:cluster}
\end{figure}
The advantages of using this approach is  that there is no need to fix in advance the size of the numerical box $b$ since it depends on the Gaussian functions. Moreover, it is possible to reduce the storage demand since instead of storing the matrix $\mathbf{M}_{\lrttei,max} \in  \mathbb{R}^{N_b^2 \times N}$, smaller matrices of sizes $N_{b,s}^2\times N_s$ are stored, where 
$N_{b,s}^2$ are the pairs of Gaussian functions associated to  the integration domain $[-b_s,b_s]$ and $N_{s}$ is the maximum number   of Chebyshev interpolation points in the interval $[-b_s,b_s]$.
We must point out that   by using  this adaptive method, multiple tensor contraction calculations need to be performed  to compute \eqref{eq::coulapprox} which will depend on the number of partitions $\mathcal{P}_s$. This can be costly if we consider a sequential algorithm.  However, this adaptive approach offers a possibility to parallelize the evaluation of  \eqref{eq::coulapprox}.
\subsection{Compression by using Screening}
\label{subsub:pres}
It is possible to further reduce the dimensions of $\mathbf{M}_{\lrttei,max}$ by exploiting the properties of the Gaussian type basis functions. In fact, given the product of two-primitive Gaussians  introduced in \eqref{gauss_prod}, we notice that $g_{\mu\nu}^{(j)}(\boldsymbol{x}_{l})=g_{\nu\mu}^{(j)}(\boldsymbol{x}_{l})$, for $j  \left\{ 1,\cdots,I_{\mu\nu} \right\}, \mu,\nu \in \left\{ 1,\cdots, N_b \right\}$ and $l \in \left\{ 1,2,3\right\}$.  Therefore,  there are only $\frac{N_b}{2}(N_b+1)$ choices for $N_b^2$ combinations of $\mu$ and $\nu$. We also apply  the screening technique that is  often used by chemists to reduce the computational cost of the evaluation of integrals \cite{meth}. From the Gaussian product rule \eqref{gauss_prod}, the higher the exponent of a primitive Gaussian, the faster the products with primitives from other centers decay with distance and the sooner they become negligible. Therefore, for large enough molecules, it is possible to discard a consistent number of pairs of primitive Gaussians which is illustrated in the numerical experiment section in Figure~\ref{fig:pairs}. In practice, we discard the  primitive pair that satisfies the following condition for a given  threshold $\tau_{screening}$
\begin{equation}
   exp\left( -\frac{\mu_{j_1}\nu_{j_2}}{\mu_{j_1}+\nu_{j_2}}\sum_{l=1}^{3}(\mathbf{r}_l-\mathbf{r}'_l)^{2} \right) \leq \tau_{screening}.
\end{equation}

\section{Numerical results}
\label{sec::numerical}
In this section, we evaluate numerically our novel method LTEI-\lrttei \footnote{https://github.com/sbadred/LTEI\_TA.jl.git} by using a  prototype implementation in Julia language version 1.5.3. We also compare it with {\fmm} method   using \textit{defmm} library \cite{igor}. The $\textit{defmm}$ library is a C++ code\footnote{https://github.com/IChollet/defmm} that is particularly well-suited for the two-electronic integrals context since it implements various important features with $\mathcal{O}(N)$ complexity on non-oscillatory kernels in both precomputation and application cost. More precisely, $\textit{defmm}$ is
\begin{itemize}
    
    \item \textit{kernel-independent}, meaning that the user has to provide only a routine evaluating $K(\boldsymbol{x},\boldsymbol{y})$ to use the code and the handling of $erf$ function can be added at minimal implementation effort,
    \item \textit{adaptive}, meaning that the algorithm automatically adapts to the potential non-uniformity of the particle distribution. Similar  performance was observed for \textit{defmm} using non-oscillatory kernels applied on uniform and highly non-uniform distributions \cite{igor} (such as our tensorized Chebyshev grids),
    \item \textit{convergent} for any asymptotically smooth kernel, including our kernel $K(\boldsymbol{x},\boldsymbol{y})$ (see Proposition~\ref{proposition_erfas}), as proven in \cite{chollet:hal-03563005}.
\end{itemize}
An example of a call to  \textit{defmm}  library  is provided in  \ref{Appen::B}.
 \textit{defmm} is compiled using the intel C++ compiler (version 19.1.2.254) and FFTW3 (since \textit{defmm} relies on FFTs for the far field compression/evaluation). We remind that the evaluation algorithm in {LTEI-\lrttei}, which is written in Julia, is based on matrix-matrix products, performed with optimized BLAS operations (see Section \ref{subsubsec::first})  for the dense linear algebra computations. Hence, the effect of the programming language choice has a negligible impact for {LTEI-\lrttei}. This justifies the comparison between c++ calls (\textit{defmm}) and our implementation of {LTEI-\lrttei} in Julia. We are also aware that results presented in the following correspond to prototypes in which we simply link \textit{defmm} with outputs from our Julia code, regardless of further possible optimizations. All the calculations  are carried out using Cleps cluster from Inria, Paris, France. This machine has 4 partitions. We use cpu-homogen partition  which contains 20 nodes with hyper-threading such that we can allocate a maximum of 64 logical cores per node (Intel(R) Xeon(R) Silver 4214 CPU @ 2.20GHz) with a  memory of 6GB per core.
We start always by the data initialization step which consists  in  reading input files generated from \textit{quantum package}. These files contain molecular properties: number of atoms, number of basis functions,  coordinates of the nuclei, basis set parameters. For all molecules we use the “cc-pVDZ” Gaussian basis set \cite{basis}.\raggedbottom
\subsection{Approximation  error and computational cost}
The following numerical results present the approximation errors with respect to  different parameters $\omega$, $N$, and $N_{q_1}$. We start by providing the approximation error for the element-wise evaluation of the long-range two-electron integrals tensor and then we provide the numerical error convergence obtained for the evaluation of the long-range Coulomb matrix as defined in  \eqref{Coul} using  both methods:  {LTEI-\lrttei} and {\fmm}.
\subsubsection{Approximation error}
First we provide convergence results of LTEI-\lrttei~method for the evaluation of the long-range two-electron integrals given in equation~\eqref{approx}. For the following numerical tests, we consider small sized molecules : $NH_3$ and $CO_2$, where we represent the mean relative error of $10^3$ randomly chosen elements  from the tensor $\mathcal{B}^{lr}.$ On the left of the Figure \ref{fig::both},   the maximum number of Chebyshev interpolation points $N$ is fixed  while on the right of the Figure~\ref{fig::both} the number of quadrature points $N_{q_1}$ is fixed.
\begin{figure}[H]
    \centering
    \includegraphics[scale=0.25]{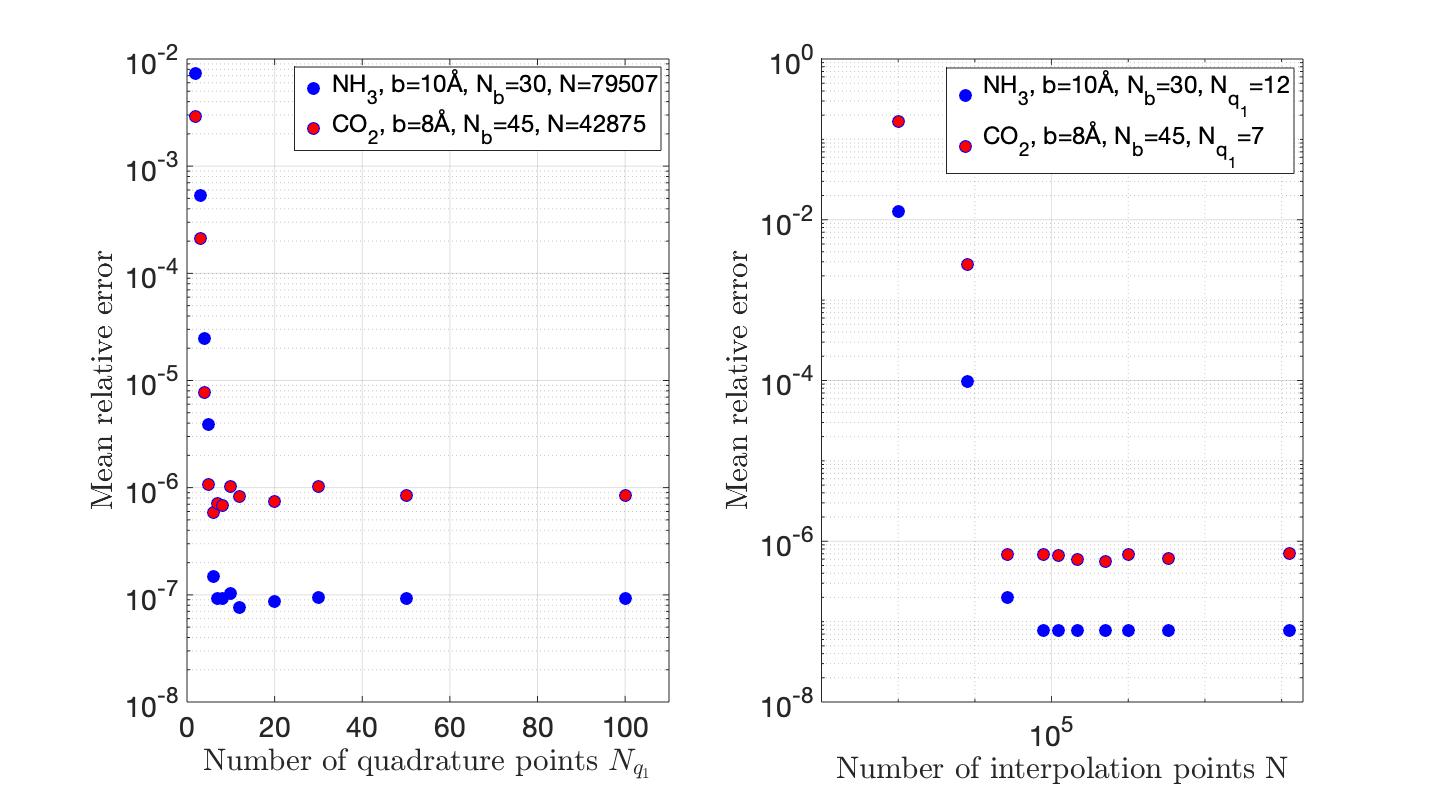}
    \caption{Approximation error of the long-range two-electron integrals using LTEI-{\lrttei}, $\omega=0.5$.}
     \label{fig::both}
\end{figure}
\begin{figure}[H]
\centering
\begin{subfigure}{.5\textwidth}
  \centering
  \includegraphics[width=.9\linewidth, height=120px]{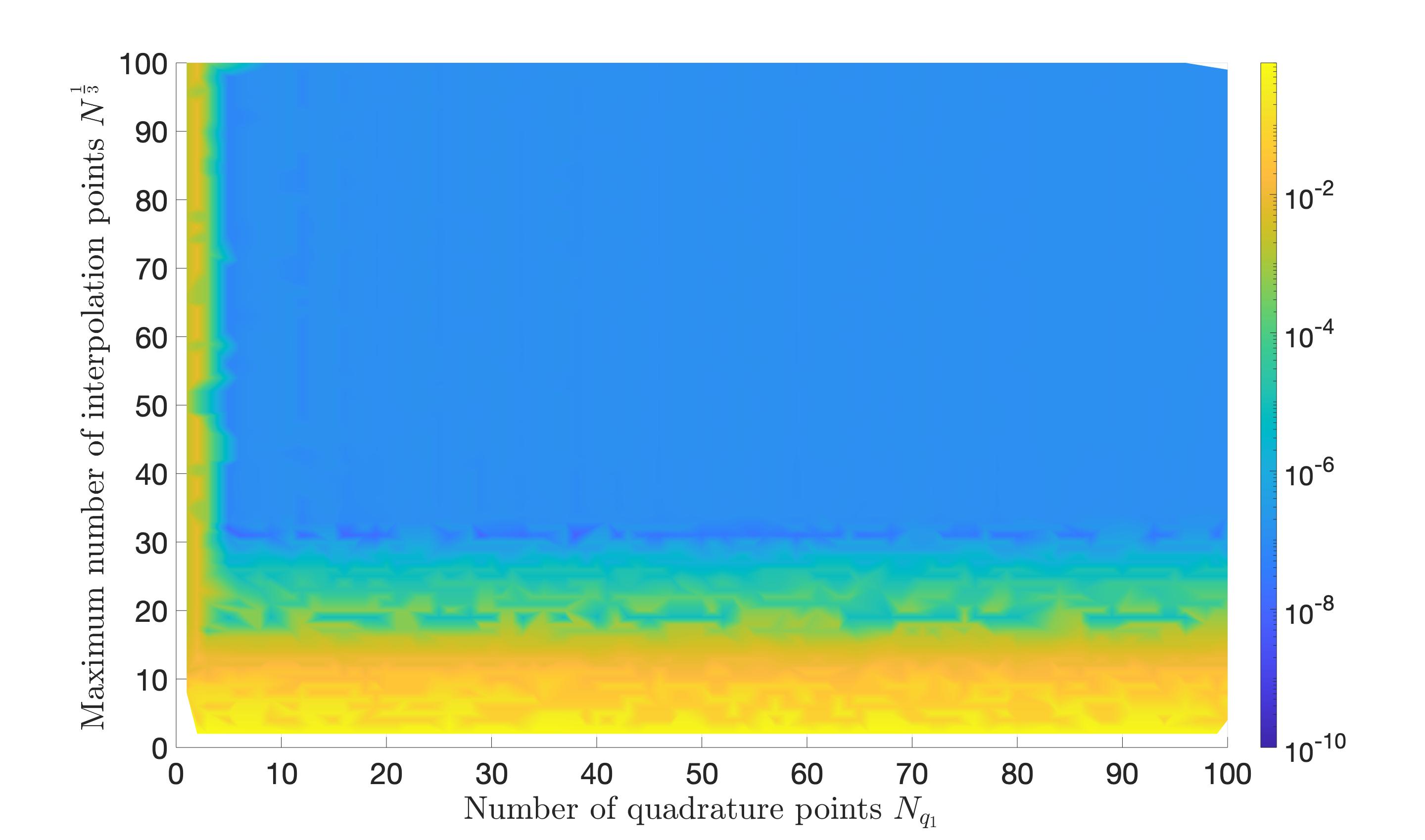}
   \subcaption{$\omega=0.5$}
\end{subfigure}%
\begin{subfigure}{.5 \textwidth}
  \centering
  \includegraphics[width=.9\linewidth, height=120px]{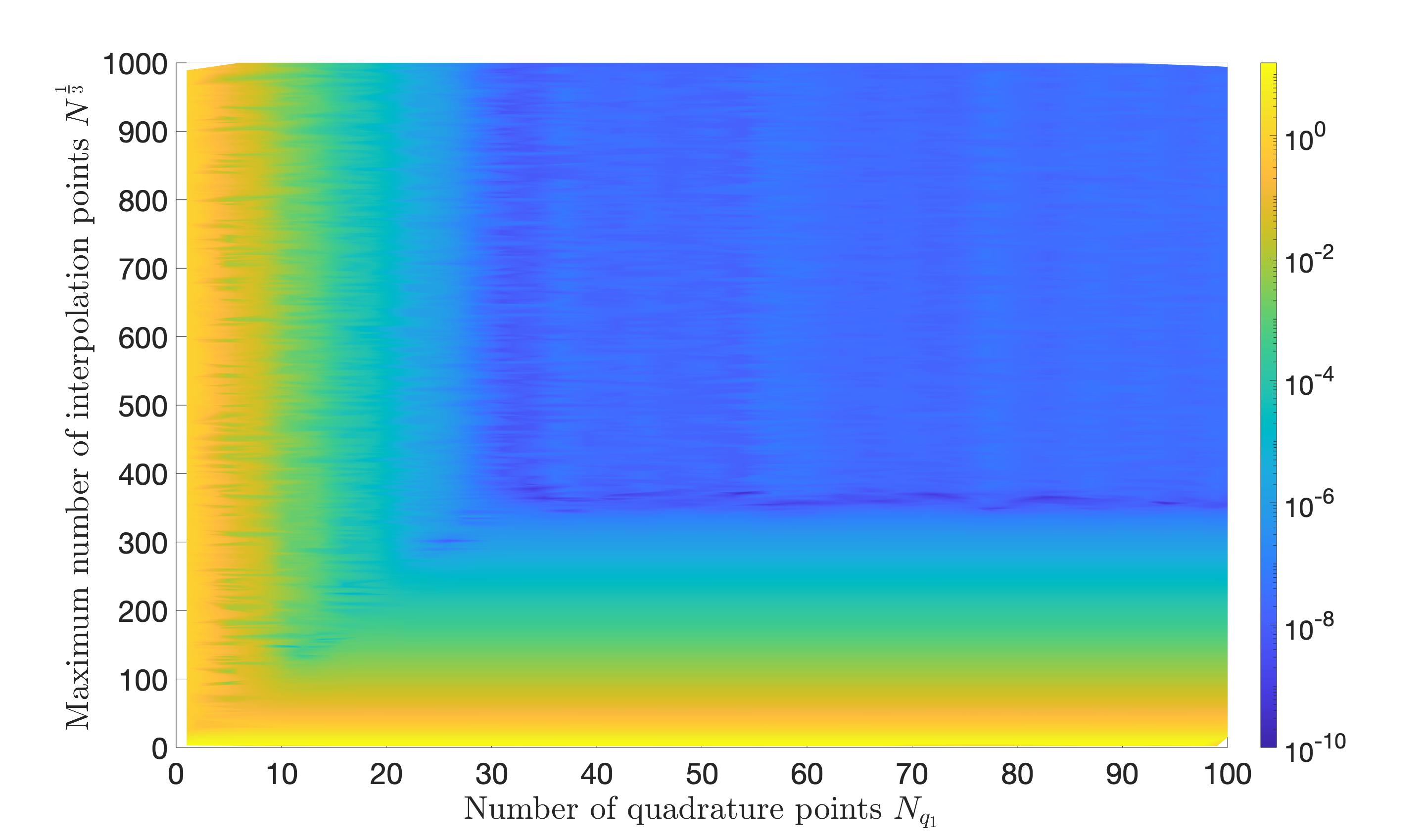}
   \subcaption{$\omega=5$}
\end{subfigure}
\caption{ Approximation error of the element-wise evaluation of the two-electron integrals   \eqref{approx} with respect to ($\#$interpolation points per direction, $\#$quadrature points ) $\equiv$ ($N^{\frac{1}{3}}$,$N_{q_1}$) for the optimal accuracy using NH3 molecule in the cc-pVDZ basis set for different values of $\omega$. The colorbar shows the mean relative approximation error.}
\label{fig::cross}
\end{figure}
In Figure~\ref{fig::both}, with fixed $\omega=0.5$, we notice the fast convergence of the relative error towards the value  of $1e^{-7}$ for both subfigures such that  the analytical results, generated from \textit{quantum package}, and numerical results are in reasonably good agreement  for both molecules. We note  that the stagnation of the error is a consequence of the approximations used (Chebyshev interpolation and Gaussian quadrature rule), hence in order to optimize our method for a desired accuracy, we need to find a good compromise between the parameters $N^{\frac{1}{3}}$ and $N_{q_1}$, as shown in Figure~\ref{fig::cross}. Indeed, we note that in Figure~\ref{fig::cross}, for each number of interpolation points, there is a number of quadrature points  that allows to reach a small relative error (up to $1e^{-10}$). 
One may also notice that the minimal error is constrained by the choice of $b$, i.e. of the integration box, since the support of the primitive Gaussians are truncated.
\begin{figure}[H]
\begin{center}
\includegraphics[width=0.9\linewidth, height=200px]{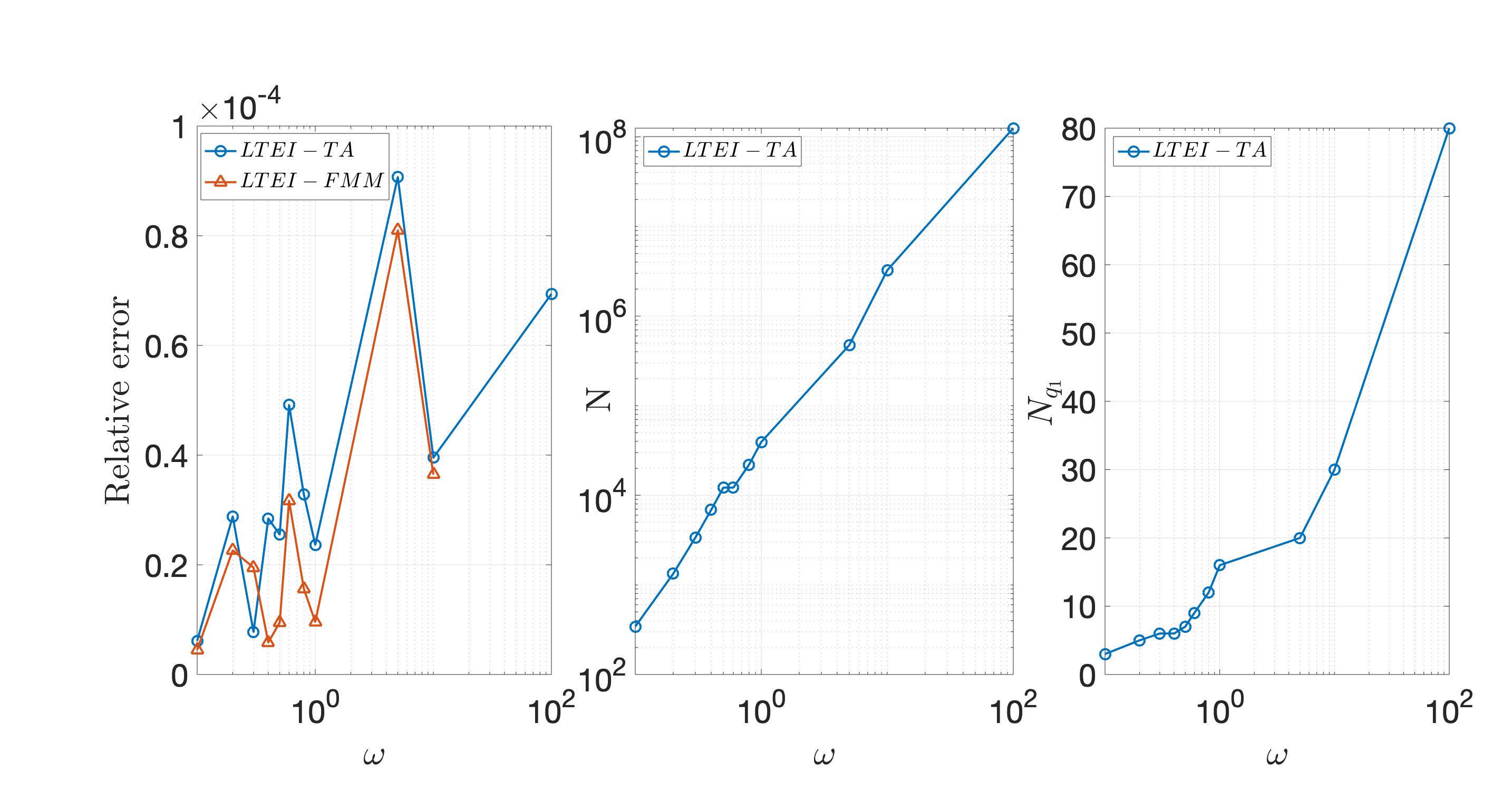}
\caption{(Leftmost figure) The approximation error of the element-wise evaluation of the two-electron integrals with respect to $\omega$ for both approaches: LTEI-{\lrttei} and {\fmm}. (Middle figure) The number of interpolation points $N$ needed to reach the imposed accuracy (relative error smaller than $1e^{-4}$) with respect to $\omega$. (Rightmost figure) The number of quadrature points $N_{q_1}$ needed to reach the imposed accuracy (relative error smaller than $1e^{-4}$) with respect to $\omega$.}
\label{fig::ele}
\end{center}
\end{figure}
Figure~ \ref{fig::ele} displays  the number of interpolation points $N$ (middle figure) and the number of quadrature points $N_{q_1}$ (rightmost figure) with respect to $\omega$ for computing a single entry of the  long-range two-electron integrals tensor  $\mathcal{B}^{lr}$  through LTEI-\lrttei~and {\fmm} approaches. The entry $\mathcal{B}^{lr}(\mu,\nu,\kappa,\lambda)$ is chosen randomly and we impose that  the relative error is smaller than $10^{-4}$, where the relative error is  defined as  $\frac{\left | \mathcal{B}^{lr}(\mu,\nu,\kappa,\lambda)-\mathcal{B}_{LTEI-\lrttei}^{lr}(\mu,\nu,\kappa,\lambda)  \right | }{ \left |\mathcal{B}^{lr}(\mu,\nu,\kappa,\lambda)  \right |}$ for LTEI-\lrttei, and as $\frac{\left | \mathcal{B}^{lr}(\mu,\nu,\kappa,\lambda)-\mathcal{B}_{\fmm}^{lr}(\mu,\nu,\kappa,\lambda)  \right | }{ \left |\mathcal{B}^{lr}(\mu,\nu,\kappa,\lambda)  \right |}$ for {\fmm}, respectively.
We observe  that the number of  interpolation points $N$ and the number of quadrature points $N_{q_1}$ needed to reach the desired accuracy grow with $\omega$, as it can be seen in the  middle and rightmost figures. This is explained by the fact that when $\omega \rightarrow \infty$, {LTEI-\lrttei} needs to approximate a nearly singular kernel, which increases its cost.  
 The leftmost figure also shows that the  accuracy  of {LTEI-\lrttei} and {\fmm}  for the evaluation of an element of  $\mathcal{B}^{lr}$ is comparable for the same number of interpolation points $N$. This is because both approaches are based on Chebyshev interpolation. 
 We note that the quadrature in {LTEI-\lrttei} is chosen to be at least as precise as the interpolation and the FMM error is controlled by a parameter \cite{igor} whose value is practically calibrated so that this error equals the numerical interpolation. Both methods thus lead to the expected accuracy.
 \begin{figure}[H]
\begin{center}
\includegraphics[scale=0.15]{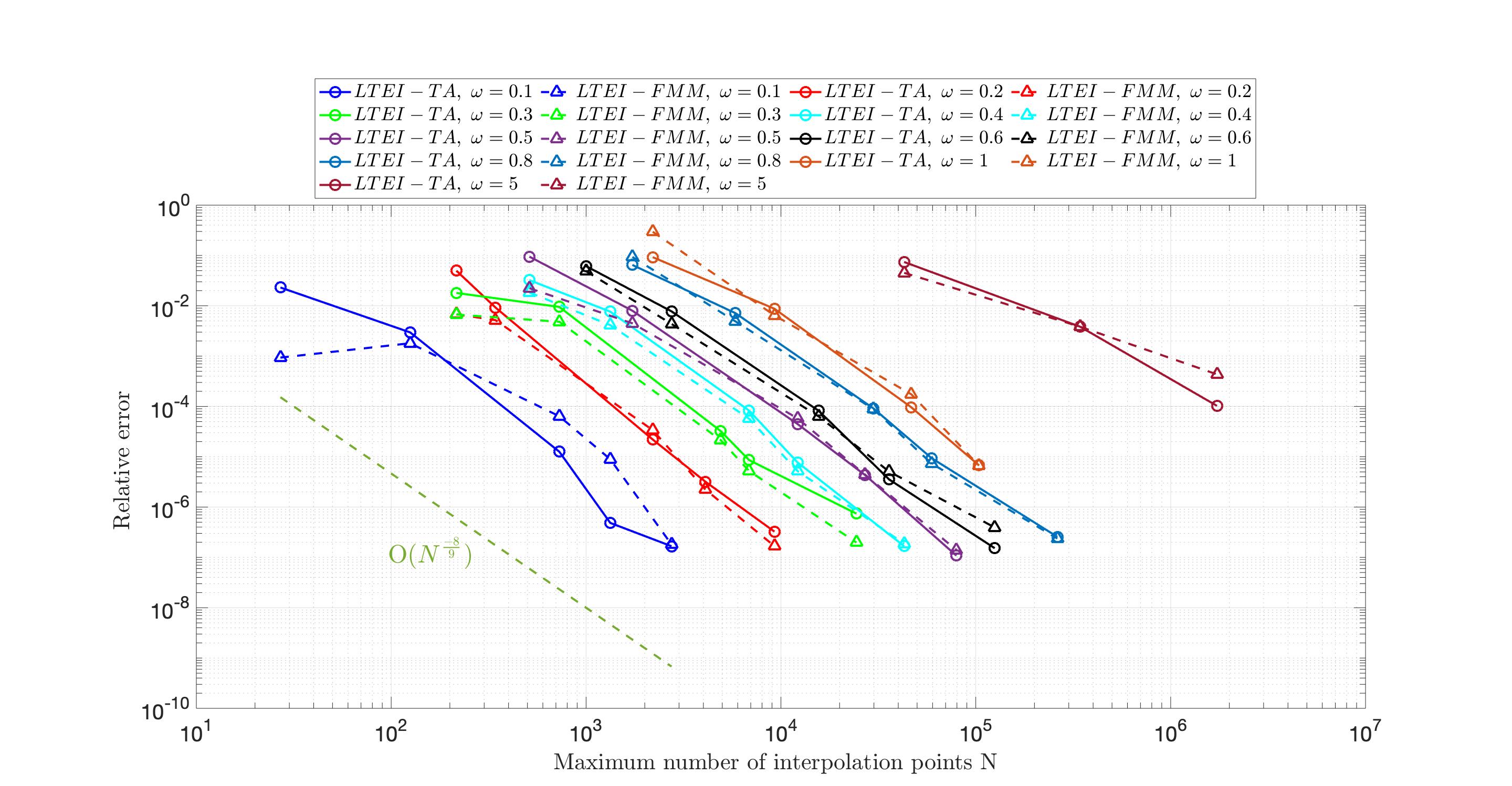}
\caption{Approximation error of the evaluation of the long-range Coulomb matrix  using LTEI-{\lrttei} and {\fmm} with respect   to the number of interpolation points $N$ for various values of $\omega$: convergence rate estimation. These calculations were carried for the glycine molecule with $N_b=100$ in the cc-pVDZ basis set.}
\label{fig::tendancy}
\end{center}
\end{figure} 
Figure~ \ref{fig::tendancy} considers the evaluation of the long-range Coulomb matrix   using LTEI-{\lrttei} and {\fmm} approaches as described in \eqref{eq::coulapprox}. It  displays the relative error   with  respect to the number of interpolation points $N$ for different values of $\omega$, where the relative error    of {LTEI-\lrttei}   (resp. {\fmm} ) is $\frac{  \left  \| \mathbf{J}^{lr} -   \mathbf{J}_{LTEI-\lrttei}^{lr}     \right \|_2   }{ \left \| \mathbf{J}^{lr}   \right \|_2}$ (resp. $\frac{  \left  \| \mathbf{J}^{lr} -   \mathbf{J}_{\fmm}^{lr}     \right \|_2   }{ \left \| \mathbf{J}^{lr}   \right \|_2}$). We note  that we were not able to  evaluate theoretically the  convergence rate of this evaluation with respect to the  number of interpolation points  $N$. We observe, however, that the numerical error seems to have an almost linear-scaling in the 3D tensorized interpolations grid size $N$ for small values  $\omega \in (0,1)$. However, this   scaling is lost for larger $\omega$. Indeed, we expect our method to be far less efficient for very large $\omega$ since the underlying kernel tends to the (singular) Coulomb one  when  $\omega\rightarrow +\infty$.

\subsubsection{Computational cost}
We first discuss   the execution time  required  for the evaluation of  an element  of the long-range tensor $\mathcal{B}^{lr}$, as displayed in Figure \ref{fig::timefmm} . The computational  complexity of this evaluation is of order $\mathcal{O}(N_{q_1}N^{\frac{1}{3}}I_{\kappa \lambda}(N^{\frac{1}{3}}+I_{\mu\nu}))$ as discussed in Section~\ref{subsubsection:ewf}. For small values of $N_{q_1}$ and a few number of interpolations points $N^{\frac{1}{3}}$, we   obtain linear scaling  with respect to $N^{\frac{1}{3}}$ as shown in Figure~ \ref{fig::timefmm}. This is explained by the fact that the term $\mathcal{I_{\kappa\lambda}I_{\mu\nu}}$ dominates the overall complexity for small $\omega$. However, when $\omega$ increases, a quadratic complexity is observed with respect to $N^{\frac{1}{3}}$, which  correponds to  $\mathcal{O}(N_{q_1}N^{\frac{1}{3}}I_{\kappa \lambda}(N^{\frac{1}{3}}+I_{\mu\nu}))$. We also compare LTEI-{\lrttei} with {\fmm}  and  with a naive numerical computation such that the two-electron integrals are computed with an integration over    $N\times N \times N$ tensorized three dimensional Cartesian grids. We notice here that the {\fmm}  approach has a linear scaling with regards to the number of interpolation points $N$ as expected. We conclude that for the element-wise evaluation, LTEI-{\lrttei}  is the most efficient method.
\begin{figure}[H]
\begin{center}
\includegraphics[scale=0.15]{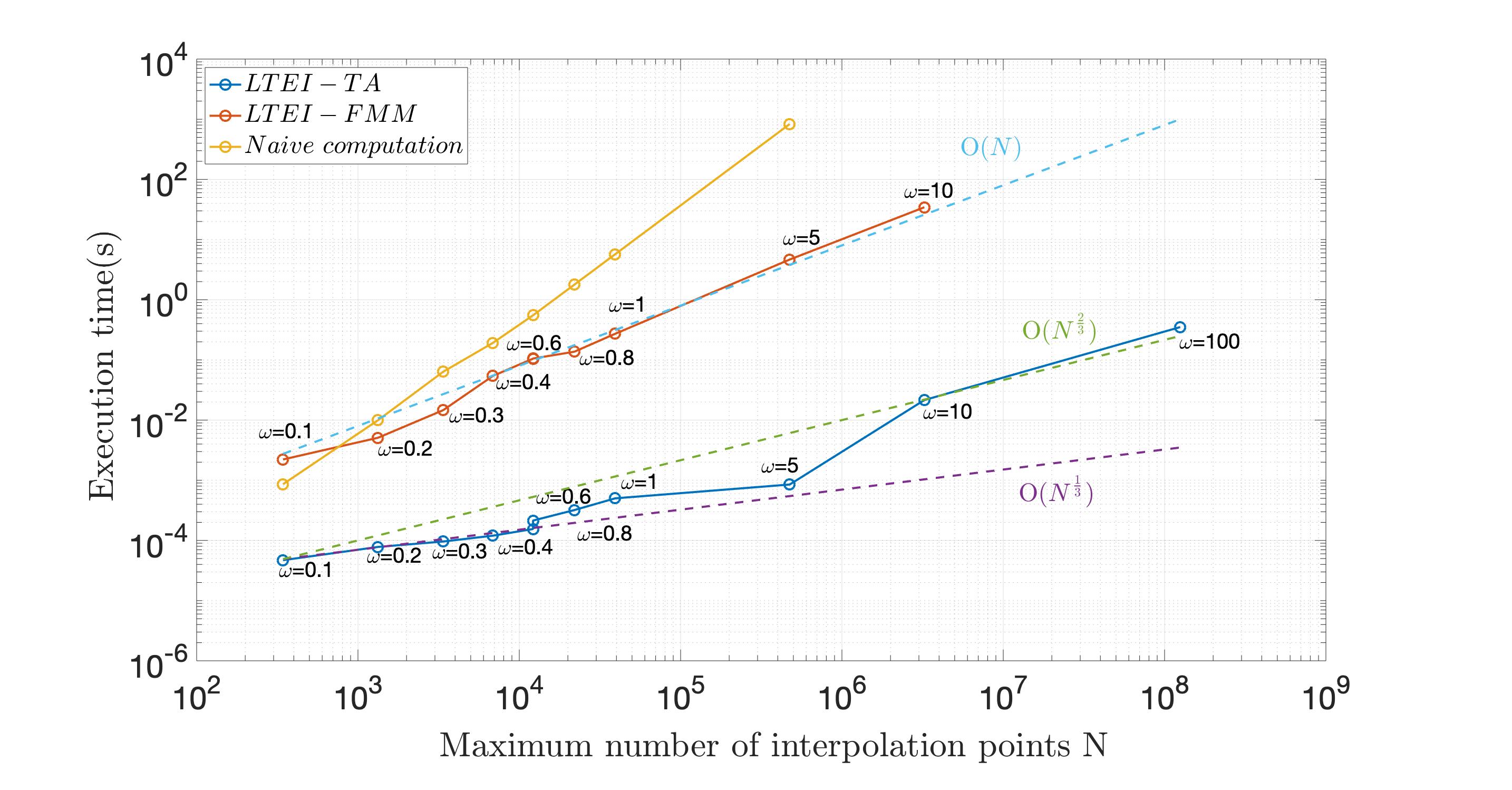} 
\caption{Computational time versus the maximum number of interpolation points $N$ for different values of $\omega$ for the evaluation of the long-range two-electron integrals with relative error smaller than  $\leqslant 1e^{-4}$.}
\label{fig::timefmm}
\end{center}
\end{figure} 
\begin{figure}[H]
\begin{center}
\includegraphics[scale=0.15]{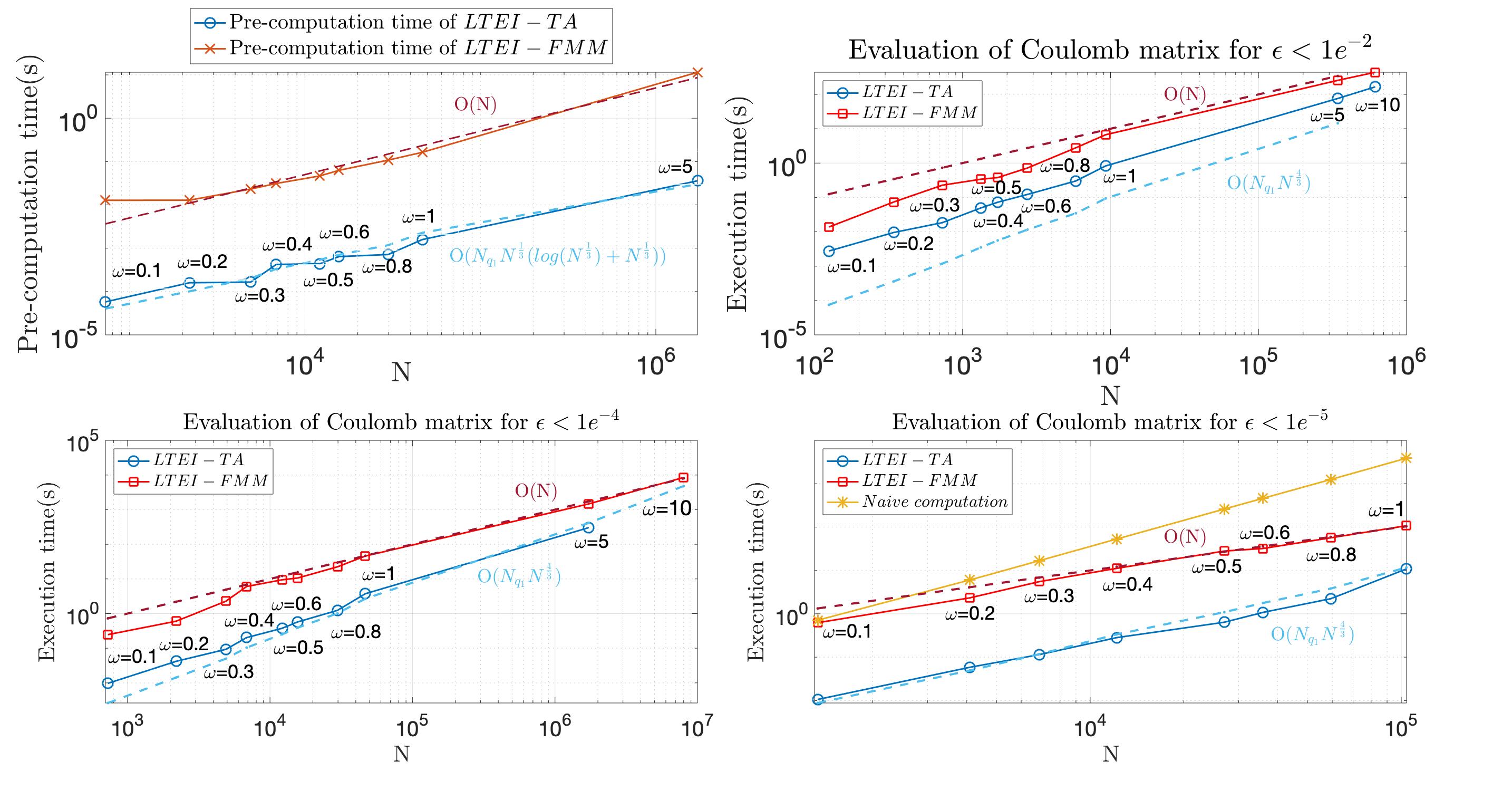}   
\caption{The leftmost plot represents the precomputation time for each approximation approach (LTEI-{\lrttei} and {\fmm}) with respect to  the maximum number of Chebyshev interpolation points  $N$. We impose here that the relative error denoted by $\epsilon$ is smaller than   $\leqslant 1e^{-4}$. We provide in the other plots a comparison in terms of the computational time required for  the evaluation of \eqref{Coul} between both approaches by varying the error bound $\epsilon$ and $\omega$ . We use the Glycine molecule $C_2H_5NO_2$ with fixed  $N_b=100$ and $N_{orb}=95$ in the cc-pVDZ basis set.}
\label{fig::times}
\end{center}
\end{figure}
Second, we compare the precomputation cost required to approximate the long-range kernel $K(x,y)$, as given in \eqref{eq::kernel}, by using both approaches {LTEI-\lrttei} and {\fmm} and by varying $\omega$ from $0.1$ to $5$. The results are displayed in the leftmost part of Figure~\ref{fig::times}. We observe that the runtime of  {\fmm} depends linearly on the total number of interpolation points $\mathcal{O}(N)$, independently of the value of $\omega$. {LTEI-\lrttei} has also a precomputation time in accordance with the theory $\mathcal{O}(N_{q_1}N^{\frac{1}{3}}(log(N^{\frac{1}{3}})+N^{\frac{1}{3}}))$ as explained in Section~\ref{sec::LRTTEI} . We observe that  {LTEI-\lrttei}  is two orders of magnitude faster than {\fmm}  for all the considered values of $\omega$ (which is a consequence of its small precomputation complexity).

Third, we discuss the time required  to evaluate the long-range Coulomb matrix, as given in equation  \eqref{Coul}, which involves the multiplication of the  matricization of $\mathcal{B}^{lr}$ with a matrix. Figure~\ref{fig::times} illustrates the execution time with respect to the number of  interpolation points $N$ needed to achieve  different relative errors for various values of $\omega$ for the evaluation of the Coulomb matrix. The relative error of {LTEI-\lrttei}   (resp. {\fmm} ) is $\frac{  \left  \| \mathbf{J}^{lr} -   \mathbf{J}_{LTEI-\lrttei}^{lr}     \right \|_2   }{ \left \| \mathbf{J}^{lr}   \right \|_2}$ (resp. $\frac{  \left  \| \mathbf{J}^{lr} -   \mathbf{J}_{\fmm}^{lr}     \right \|_2   }{ \left \| \mathbf{J}^{lr}   \right \|_2}$). We observe in Figure~ \ref{fig::times} that the evaluation of the long-range Coulomb matrix using {\fmm} approach scales linearly with the number of interpolation  points $\mathcal{O}(N)$, but more than linearithmically for {LTEI-\lrttei}. This reflects the complexity analysis of {LTEI-\lrttei} method, $\mathcal{O}(N_{q_1}N^{\frac{4}{3}})$, provided in Section~\ref{sec::fasteval}. However, LTEI-{\lrttei} is still faster than {\fmm} for relatively small values of $\omega$ and for different relative errors. This numerical gain can be explained by the important prefactor of the {\fmm} approach: even if the complexity is linear, there is an important constant hidden in the big $\mathcal{O}$ notations \cite{LU20071348}. While for small values of $\omega$, $N_{q_1}$ is small and hence {LTEI-\lrttei} is more efficient. However, {LTEI-\lrttei}  is not   asymptotically  competitive with respect to {\fmm} approach. Indeed, as $\omega$  controls the regularity of the  \textit{erf}-interaction function, when   $\omega$ increases, {LTEI-\lrttei} needs a  larger  number of interpolation points $N$ as well as  quadrature points $N_{q_1}$ to achieve a given accuracy. As a consequence, {LTEI-\lrttei} becomes more costly and less efficient than {\fmm}.

To summarize, these results demonstrate two major things: first, {LTEI-\lrttei} is a numerically highly efficient method, able to outperfom {\fmm} on tested cases. Second, we are able to reach the linear complexity (with regard to the total number of interpolation points) by exploiting {\fmm}, which allows to deal with more singular cases (with large values of $\omega$). In the following, we want to study the efficiency of our numerical approaches for variable $N_b$.
\begin{figure}[H]
\begin{center}
\includegraphics[scale=0.15]{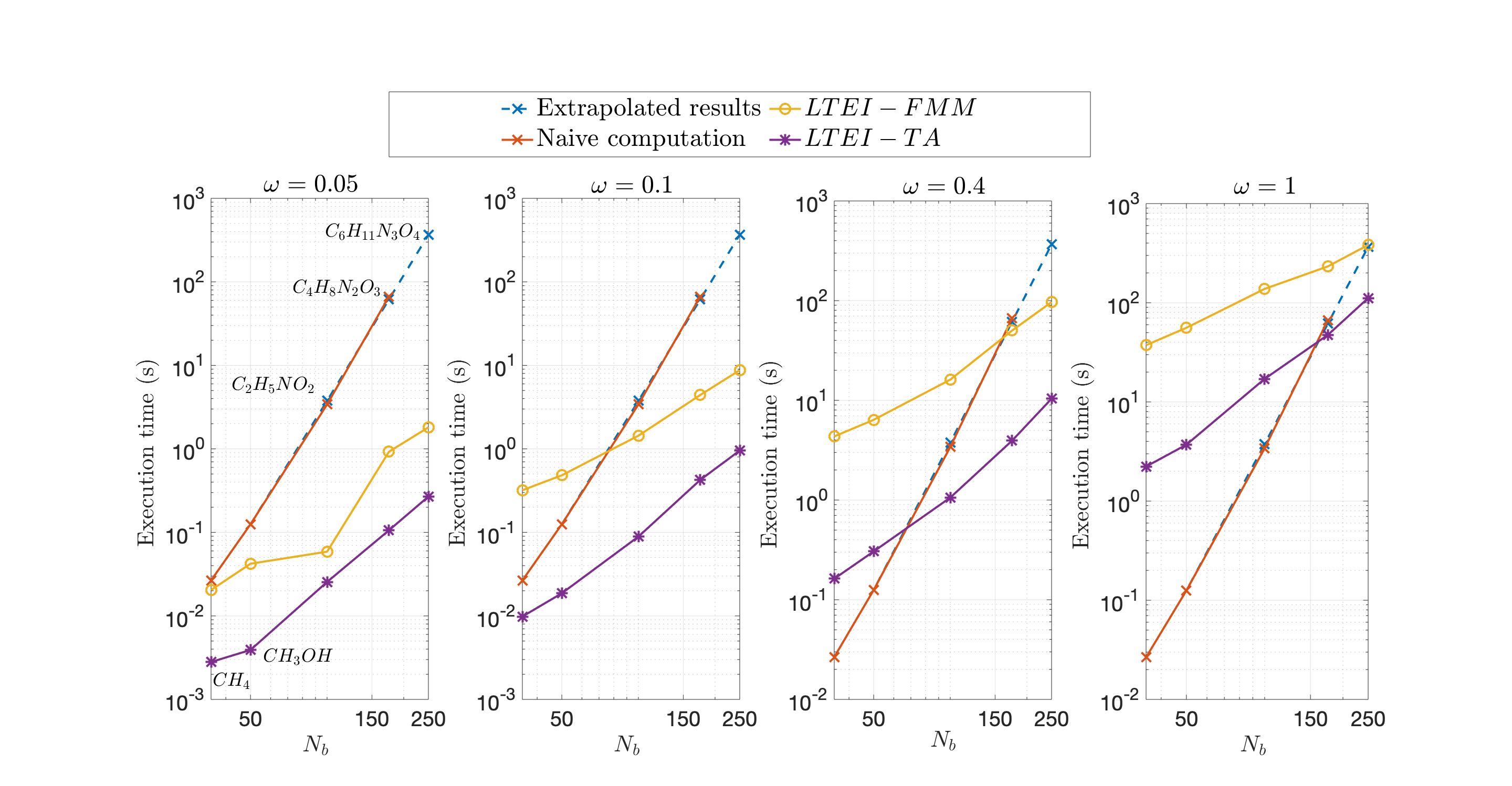}  
\caption{Execution time(s) required for the evaluation of  \eqref{Coul} using the TEI tensor $\mathbf{B}^{lr}$ for different values of $N_b$, for $\omega=0.05, \omega=0.1,\omega=0.4,$ and $\omega=1$ with  imposed relative error smaller than  $1e^{-5}$.}
\label{fig::application}
\end{center}
\end{figure} 
Figure~\ref{fig::application} displays the  execution times required to evaluate \eqref{Coul} with respect to the number of basis functions $N_b$ for different values of $\omega$ and different molecules. We impose here that the relative errors of {LTEI-\lrttei} and {\fmm} approaches for the evaluation of the long-range Coulomb matrix are  smaller than $1e^{-5}$. We compare the running times between three approaches: the first approach is a direct computation of \eqref{Coul} given the matricization of $\mathcal{B}^{lr}$ denoted by   $\mathbf{B}^{lr} \in  \mathbb{R}^{N_b^2 \times N_b^2}$ (times for $N_b > 175$ are obtained by extrapolation). The second (resp. third) approach  exploits  the factorized structure of $\mathbf{B}^{lr}$ obtained through {LTEI-\lrttei} (resp. {\fmm})   to compute \eqref{Coul}. For small $\omega$, we notice that a faster computation of \eqref{Coul} is obtained through   {LTEI-\lrttei} and {\fmm} methods: LTEI-{\lrttei}  is about one order of magnitude faster than {\fmm}. For important values of $\omega$ ($\omega$=1), the new introduced approaches, {LTEI-\lrttei} and  {\fmm},  are less efficient    given the high number of interpolation points $N$ needed as well as the number of quadrature points $N_{q_1}$ for LTEI-{\lrttei} method as we notice in Figure~\ref{fig::times}. However,  when $N_b$ increases,  the tensor contractions using the direct method  will be expensive and will have high memory demands (sometimes $\mathbf{B}^{lr}$ is too large to store in memory) . Therefore in some cases, it would be beneficial to use one of the new factorization methods  to  reduce the computational and storage cost. The numerical results are obtained for different molecules with different topologies.  Therefore, in order to preserve the accuracy, in practice, we choose the size of the computational box  $[-b,b]$  depending on the size of the molecule as well as the Gaussian functions decay as explained previously in Section~\ref{sec::LRTTEI}.\raggedbottom
\subsection{Tensor compression techniques}
In this section we study numerically compression techniques  to reduce  the computation and storage cost of $\mathbf{M}_{\lrttei,max} \in  \mathbb{R}^{N_b^2 \times N}$ or $\mathbf{M}_{FMM} \in  \mathbb{R}^{N_b^2 \times N}$ in order to speed up the evaluation of the Coulomb matrix \eqref{Coul}. These techniques were   discussed in Section~\ref{sub::LR}. First, the number of basis functions $N_b$ can be reduced by using screening techniques that exploit the symmetries of the pairs of basis functions as well as the properties of Gaussian type-functions. Indeed, Figure~\ref{fig:pairs} shows that the number of pairs of Gaussian type basis functions $N_b^2$ can be reduced by using screening. Second, for small values of $\omega$ and different numbers of basis functions $N_b$, Figure~\ref{fig:low_rank} shows that the singular values of   $\mathbf{M}_{\lrttei,max}$ decay quickly   , so $\mathbf{M}_{\lrttei,max}$ can be approximated by a low-rank matrix. Therefore,  we had recourse to three different approaches for the compression of $\mathbf{M}_{\lrttei,max}$: the first approach, denoted by SVD, consists in approximating  $\mathbf{M}_{\lrttei,max}$  using $\epsilon$-truncated SVD;
the second approach, denoted by KR, exploits  the Khatri-Rao product properties as discussed in Section~\ref{KR};  and  the third approach, denoted by ADAP+KR, includes the partitioning of pairs of basis functions in terms of their numerical supports   combined with KR approach as explained in Section~\ref{subsub::adap}.

Figure~\ref{fig::compression} (resp. Figure~\ref{fig::compression2}) displays the compression rate obtained between uncompressed $\mathbf{M}_{\lrttei,max}$ matrix  (resp.  screened $\mathbf{M}_{\lrttei,max}$ matrix  )  and its compressed representation, for different molecules with different number of basis functions $N_b$ in the basis set cc-pVDZ. We notice that 
the best compression rate, i.e $(1-\frac{size~of~ compressed~version}{size~of~original})*100$, is obtained through the ADAP+KR approach as observed in Figure~\ref{fig::compression2} ($86\%$ for $N_b=175$) compared to the other approaches SVD ($75\%$ for $N_b=175$) and KR ($83\%$ for $N_b=175$). We observe that for SVD, the larger $N_b$ ($N_b \geq 50$), the better the compression.  While screening techniques reduce the storage requirements of the matrix  $\mathbf{M}_{\lrttei,max}$ \cite{meth},  better compression results  are obtained when they are  combined with additional techniques introduced here. Figure~\ref{fig::timecompress}, shows the computational time required  for the   compression of  $\mathbf{M}_{\lrttei,max}$. The worst execution time is obtained  for  SVD method, in particular  for large values of $N_b$ ($N_b \geq 100$).

In summary, the adaptive approach leads to the best reduction in terms of  storage   while being the fastest among the tested methods. Moreover, the choice of the dimension of the computational box $b$ does not  have to be fixed in advance,  since it depends on the pairs of Gaussian type-functions \eqref{basis}. We further investigate the accuracy of this method in Table \ref{table::accu}.  We display in this table the relative error obtained when approximating  the Coulomb matrix \eqref{Coul} by using either $M_{\lrttei,max}$ compressed by the adaptive approach or a fixed computational box $[-b,b]$. The results show that the adaptive approach is more accurate than the ones obtained by fixing the computational box in advance.
However, by using the adaptive method, the computation of the Coulomb matrix requires multiple  matrix-matrix multiplications, and this can be more costly than fixing the computational box $[-b,b]$ in advance. However, since these multiplication can be performed in parallel, parallelization might be a key component to speed up the computation of the long-range Coulomb matrix \eqref{Coul}.
\begin{figure}[H]
\centering
\begin{minipage}{.5\textwidth}
  \centering
  \includegraphics[width=8cm,height=4.5cm]{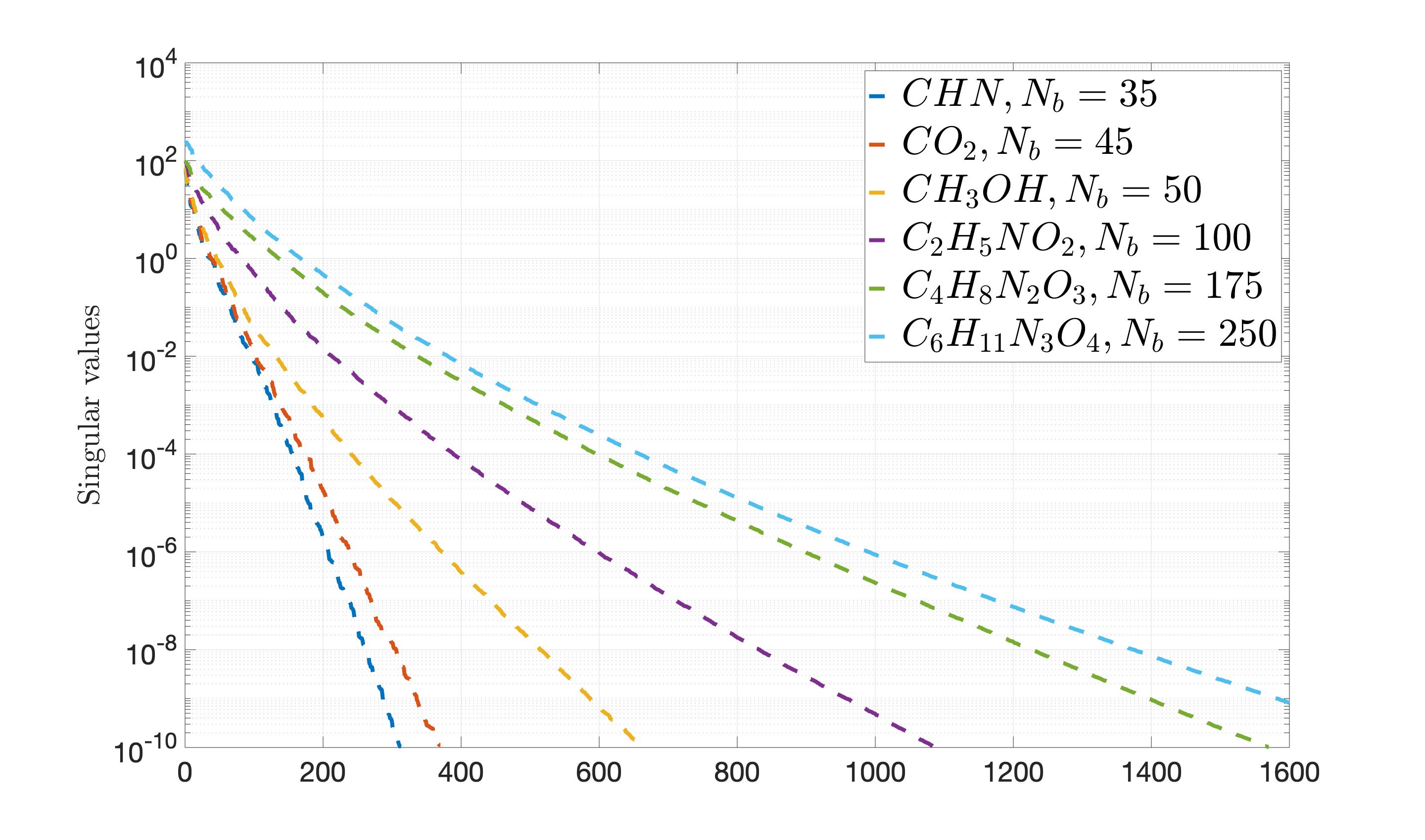}
  \subcaption{}
  \label{fig:low_rank}
\end{minipage}%
\begin{minipage}{.5\textwidth}
  \centering
  \includegraphics[width=9cm,height=5.2cm]{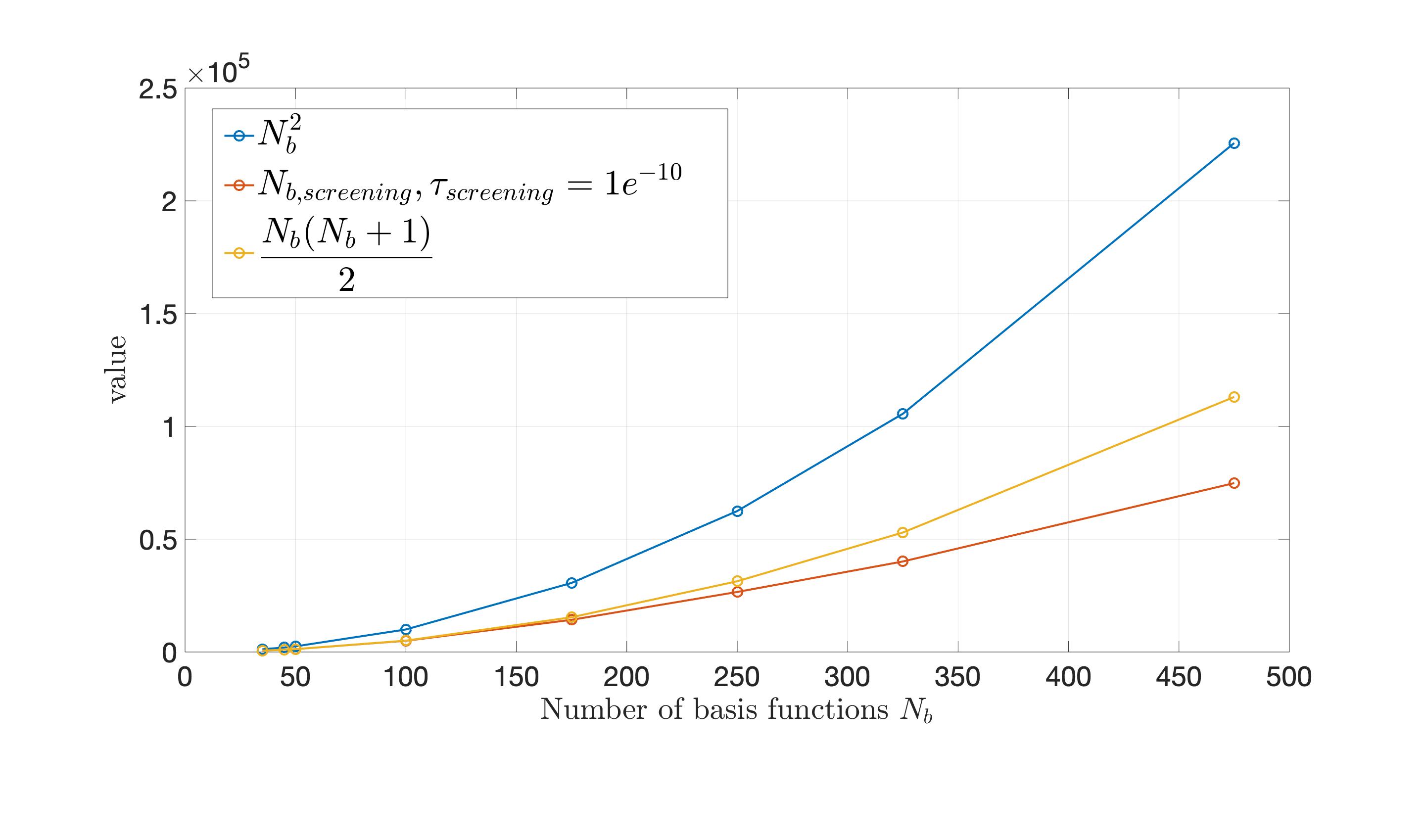}
  \subcaption{}
  \label{fig:pairs}
\end{minipage}
\centering
\caption{(a) Singular values of $\mathbf{M}_{\lrttei,max} \in \mathbb{R}^{N_b^2 \times N }$  for different molecules with $\omega=0.1$ and $N_{q_1}=3$. (b) Number of reduced pairs of basis functions obtained by exploiting  symmetry (yellow curve), as well as symmetry+properties of  Gaussian type functions with  $\tau_{screening}$=1e-10 (red curve).}
\end{figure}
\begin{figure}[H]
\centering
\begin{minipage}{.45\textwidth}
  \centering
  \includegraphics[width=1.\linewidth, height=150px]{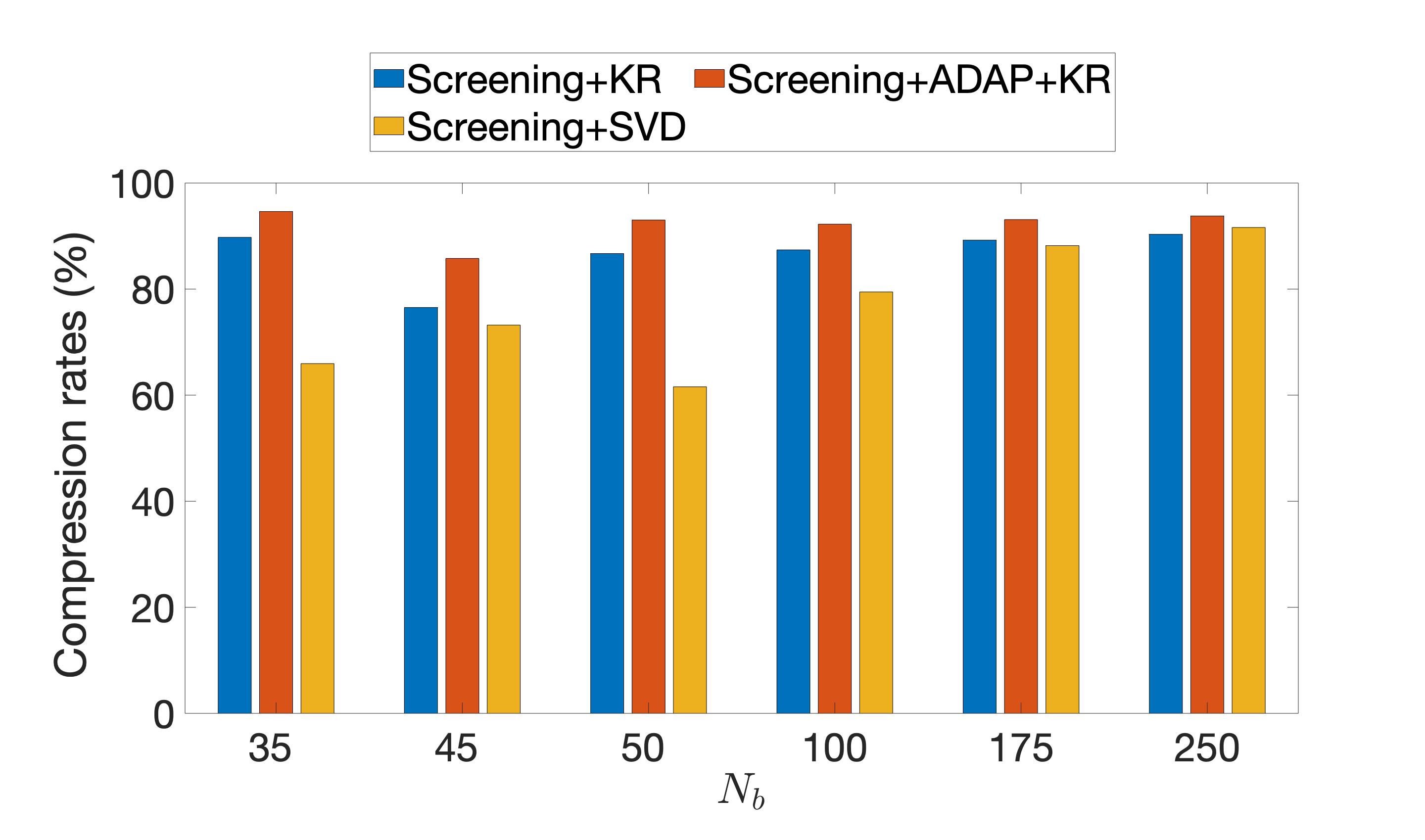}
  \caption{Compression rate between the original computed $\mathbf{M}_{\lrttei,max}$ matrix and its compressed representation for $\omega=0.3$ for different values of $N_b$, for the different molecules displayed in Figure~\ref{fig:low_rank}.}
  \label{fig::compression}
\end{minipage}%
\hfill 
\begin{minipage}{.45 \textwidth}
  \centering
  \includegraphics[width=1. \linewidth, height=150px]{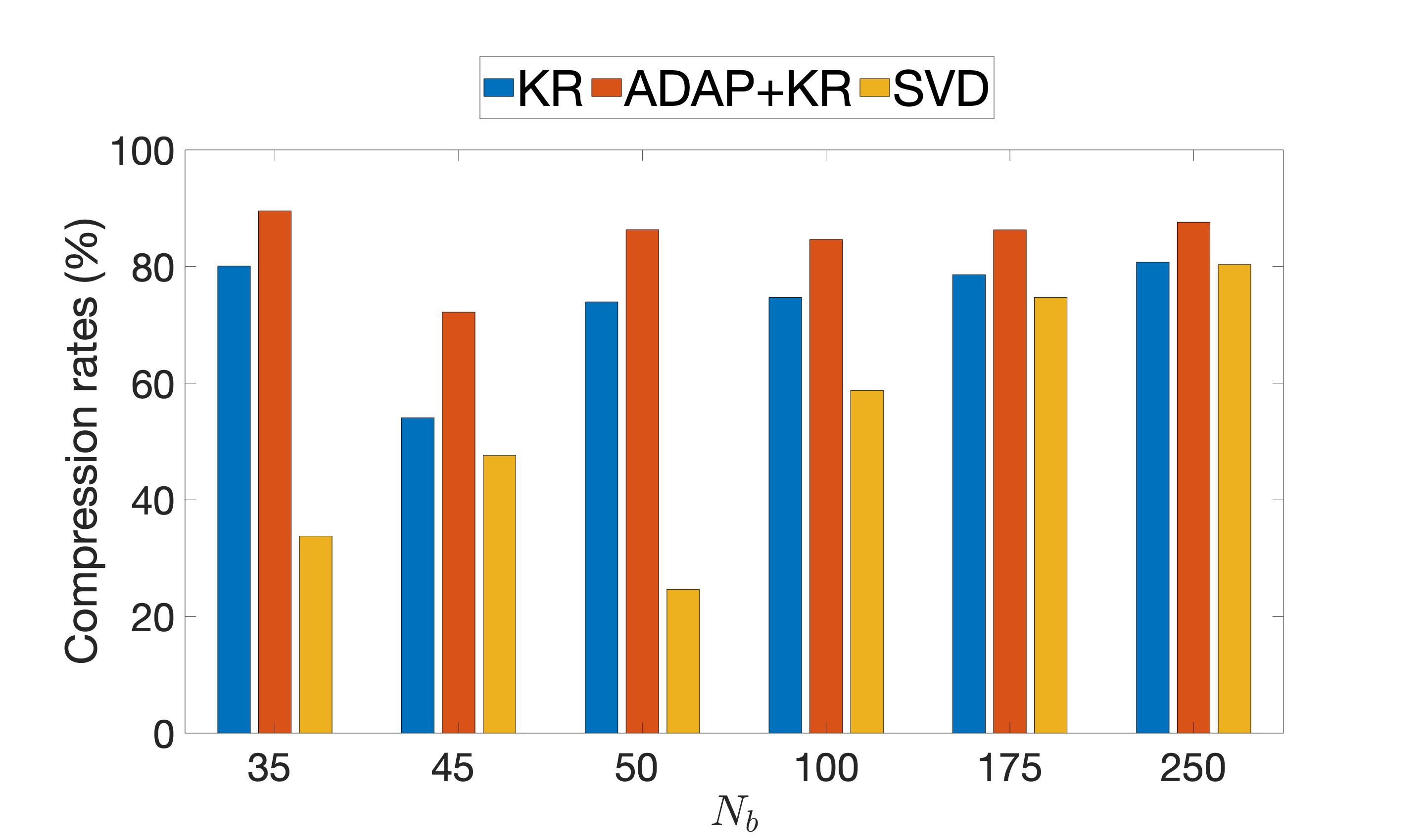}
  \caption{Compression rate between   $\mathbf{M}_{\lrttei,max}$ matrix (after screening)    and its compressed representation for $\omega=0.3$ for different values of $N_b$, for the different molecules displayed in Figure~\ref{fig:low_rank}.}
  \label{fig::compression2}
\end{minipage}
\end{figure}
\begin{figure}[H]
\centering
\begin{minipage}{.5\textwidth}
  \centering
  \includegraphics[width=1.\linewidth, height=150px]{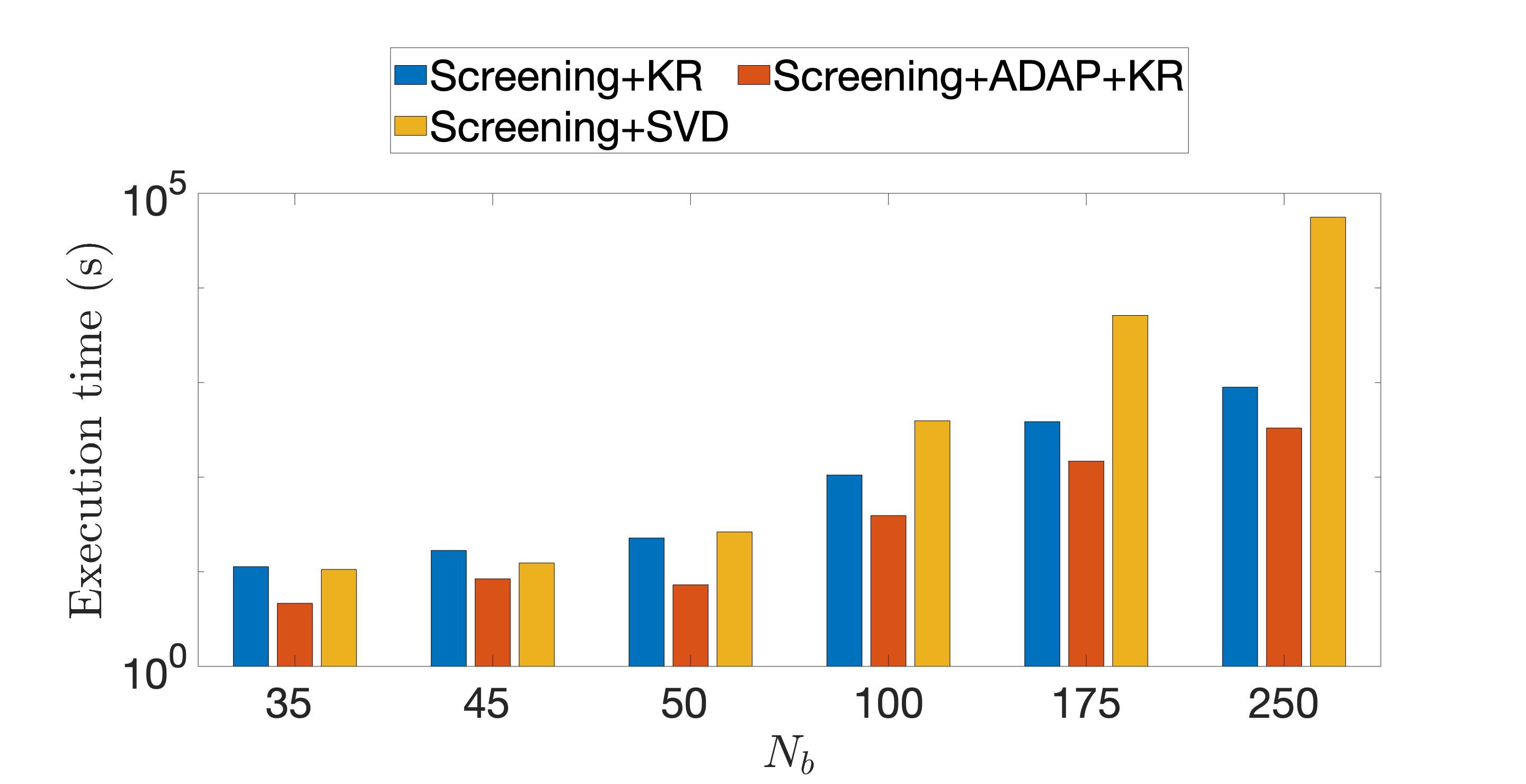}
  \caption{Execution time(s) of different compression methods defined in Section~\ref{sub::LR} for $\omega=0.3$, for different values of $N_b$, for the different molecules displayed in Figure~\ref{fig:low_rank}.}
  \label{fig::timecompress}
\end{minipage}%
\end{figure}

\begin{table}[H]
\caption {Adaptive method, $\omega=0.5$} \label{table::accu} 
\begin{center}
\scalebox{0.9}{%
\begin{tabular}{|l|c|r|r|r|r|r|}
  \hline
Molecule & $C_2H_5NO_2$ & $C_4H_8N_2O_3$ & $C_6H_{11}N_3O_4$ \\
  \hline
$N_b$ &  100  & 175 & 250 \\ 
 \hline
Adaptive approach & 1.0354e-7 & 2.4882e-8 & 4.587e-7\\
  \hline
$b=15$ & 1.6058e-7 & 1.8332e-7 & 8.2245e-7\\ 
\hline
$b=10$ & 3.7359e-07 & 0.001 & 0.02068\\
\hline
\end{tabular}}
\end{center}
\end{table}

\section{Conclusion}
 This paper introduces two new compression methods for the long-range kernel $K$ and the approximation  of the long-range six-dimensional two-electron integrals tensor. The first approach , referred to as {LTEI-\lrttei}, relies on  two-dimensional Chebyshev interpolation, Gaussian quadrature for numerical integration, and FFT  for computing Chebyshev  coefficients.
 The approximation of the long-range two-electron integrals tensor $\mathcal{B}^{lr}$ by using this  method  allows to exploit a tensorized structure that   leads to an  efficient application of the matricization of $\mathcal{B}^{lr}$  to evaluate the long-range Coulomb matrix for fixed $N_b$ and $N_{orb}$, with $\mathcal{O}(N_{q_1} N^\frac{4}{3})$ complexity, where $N$ is the number of Chebyshev interpolation points and $N_{q_1}$ is the number of quadrature points. The second approach, referred to as {\fmm}, relies on kernel-independent Fast Multipole Methods, with $\mathcal{O}(N)$ complexity. It exploits the asymptotically smooth behaviour of the long-range kernel $K$. The storage and time complexity of the presented methods were analysed and  compared numerically, exhibiting both the high efficiency of {LTEI-\lrttei} and the linear complexity of {\fmm}. We further investigated the compression of $\mathcal{B}^{lr}$ by using  screening techniques, low-rank methods, and an adaptive approach. {LTEI-\lrttei} approach is particularly efficient for small values of $\omega$, where $\omega$ is the separation parameter that controls the regularity of $K$. However,  for large values of $\omega$, in order to preserve accuracy, the number of interpolation points  as well as the number of quadrature points   becomes important for  {LTEI-\lrttei} and  thus {\fmm} becomes more efficient.

 As future work, we are planning to explore the potential of {LTEI-\lrttei}  for small values of $\omega$ in a range of quantum chemical contexts as post-HF models or hybrid approaches such as (long-range) DMRG–short-range DFT \cite{Hedeg_rd_2015}. Concerning {\fmm} it would be interesting to consider more singular kernels than the one in this paper (such as $\frac{erfc(\omega|\boldsymbol{x}-\boldsymbol{y}|)}{|\boldsymbol{x}-\boldsymbol{y}|}$ or the Coulomb kernel directly), thus extending {\fmm} to the evaluation of the short-range two-electron integrals by studying appropriate singular quadrature. Such work might be also  beneficial for Particle Mesh Ewald methods \cite{unknown}. 

\section*{Acknowledgments}
The authors are grateful to Julien Toulouse (Sorbonne university and CNRS), Emmanuel Giner( Sorbonne university) and  Xavier Claeys (Sorbonne university) for valuable discussions. We are thankful to Emmanuel Giner for his assistance with the configuration of quantum package and the extraction of molecular data.
The authors are also grateful to the CLEPS infrastructure from the Inria of Paris for providing resources and support. This project has received funding from the European Research Council (ERC) under the European Union’s Horizon 2020
research and innovation program (grant agreement No 810367).

\appendix
\section{The \textit{defmm} library}
\label{Appen::B}
The \textit{defmm} library (https://github.com/IChollet/defmm) is a easy to use C++ implementation of the directional interpolation-based Fast Multipole Method exploiting equispaced interpolation combined with Fast Fourier Transforms. Mainly, \textit{defmm} ensures a $\mathcal{O}(N)$ complexity independently of the particle distribution. Here, we provide an example of a short program calling \textit{defmm}: only five lines are needed to construct and apply the FMM matrix to a vector.
\begin{lstlisting}
#include "path to defmm/include/interface.hpp"
using namespace defmm;
int main(){

  const int  DIM   = 3    ; // Dimension
  const int  ORDER = 4    ; // Interpolation order
  const int  NCRIT = 32   ; // Number of particle per leaf cell
  const flt  KAPPA = 0.   ; // Wavenumber (for oscillatory kernels)
  const int  N     = 41334; // Number of points
  
  // Get random charge vector
  Vecc Q(N), P(N);
  for(int n = 0; n < N; n++){
    Q[n] = cplx(urand);}
  
  IBFMM_Mat<DIM> A;                   // FMM matrix
  A.addSourceParticlesINP("Y.inp",N); // Read source particles in Y.inp
  A.addTargetParticlesINP("X.inp",N); // Read target particles in X.inp
  A.prcmpt(ORDER,NCRIT,KAPPA);        // Precompute
  gemv(A,Q,P);                     // Compute P = A Q
 
  return 0;
}
\end{lstlisting}
As a header-only library, \textit{defmm} does not need to be compiled before calling. However, our library calls both BLAS and the FFTW3 library \cite{FFTW}.

Input files for the listing of source and target particles (that can be the same) are given as a sequence of particle coordinates (one particle per line, coordinates separated by blanks).


\section{The Hartree-Fock exchange}
In computational quantum chemistry, the  efficient construction of the long-range exchange matrix in the Fock matrix is also interesting \cite{Resolutions,khorom,Resolutions}. This matrix is calculated by using the long-range two-electron integrals tensor $\mathcal{B}^{lr}$. The long-range exchange matrix is given by
\begin{equation}
\label{eq::exchange}
    \mathbf{K}^{lr}(\mu,\nu)=2\sum_{j=1}^{N_{orb}}\sum_{\lambda,\kappa=1}^{N_b^2}q_{j\lambda}q_{ j\kappa}\mathcal{B}^{lr}(\mu,\lambda,\kappa,\nu), \mu, \nu \in \left \{1,\cdots,N_b\right \},
\end{equation}
with $q_{j\lambda}$, $q_{ j\kappa}$, and $N_{orb}$ being defined in Section~\ref{sec::application}. Using the long-range two-electron integrals tensor $\mathcal{B}^{lr}$, The evaluation of $\mathbf{K}^{lr}(\mu,\nu)$  costs $\mathcal{O}(N_b^2N_{orb})$. One can use the factorized structure $\mathbf{B}^{lr}_{LTEI-{\lrttei}}$ defined in \eqref{eq::tensors} to reduce the computational cost to $\mathcal{O}(NN_{orb}(N_b+N_{q_1}N^{\frac{1}{3}})$ for LTEI-{\lrttei}
approach with $N$ being the number of Chebyshev interpolation points and $N_{q_1}$ being the number of quadrature points. We obtain the following representation
\begin{equation}
    \mathbf{K}^{lr}_{LTEI-{\lrttei}}(\mu,\nu)= 2\sum_{j=1}^{N_{orb}}\left ( \sum_{i=1}^{N_{q_1}}w_i\left (\mathbf{Q}_j\mathbf{M}^{(i)}_{TA,\mu}  \right ) \otimes_{l=1}^3A^{(i)}  \left ( \mathbf{Q}_j\mathbf{M}^{(i)}_{TA,\nu}   \right )^\top \right ), 
\end{equation}
where for a fixed $j \in  \left \{1,\cdots,N_{orb} \right \} $ and $\lambda \in \left \{1,\cdots,N_{b} \right \}$, we have $\mathbf{Q}_j \in \mathbb{R}^{N_b}$ and $\mathbf{Q}_j(\lambda)=q_{j\lambda}$. The matrices $\mathbf{M}^{(i)}_{TA,\mu} \in \mathbb{R}^{N_b \times N_i^3}, i \in  \left \{1,\cdots,N_{q_1} \right \}$   are obtained by fixing the index $\mu \in  \left \{1,\cdots,N_b \right \}$ in the  tensorized representation of $\mathbf{M}^{(i)}_{TA} \in \mathbb{R}^{N_b
^2 \times N_i^3}$. These tensor representations are denoted by  $\mathcal{M}^{(i)}_{TA} \in \mathbb{R}^{N_b \times N_b \times N_i^3}$   such that 
\begin{equation}
    \mathcal{M}^{(i)}_{TA}[\mu,:,:]=\mathbf{M}^{(i)}_{TA,\mu}.
\end{equation}
Figure \ref{fig::Exch} displays the execution times required to evaluate the long-range exchange matrix \eqref{eq::exchange} with respect to the number of basis functions $N_b$, for small values of $\omega \in \left \{0.05,0.1  \right \}$. We impose that the relative error of LTEI-{\lrttei} approach for this evaluation is smaller than $1e^{-5}$ and we compare the running times between a direct computation of \eqref{eq::exchange} given $\mathbf{B}^{lr} \in  \mathbb{R}^{N_b^2 \times N_b^2}$ and the factorized structure of $\mathbf{B}^{lr}$ using $\mathbf{B}^{lr}_{LTEI-{\lrttei}}$. It can be seen that in the case of small values of $\omega$, we notice that a faster construction of \eqref{eq::exchange} is obtained through LTEI-{\lrttei}. Compression techniques introduced in Section~\ref{sub::LR}, can be used here to get better  running times.
\begin{figure}[H]
\begin{center}
\includegraphics[scale=0.15]{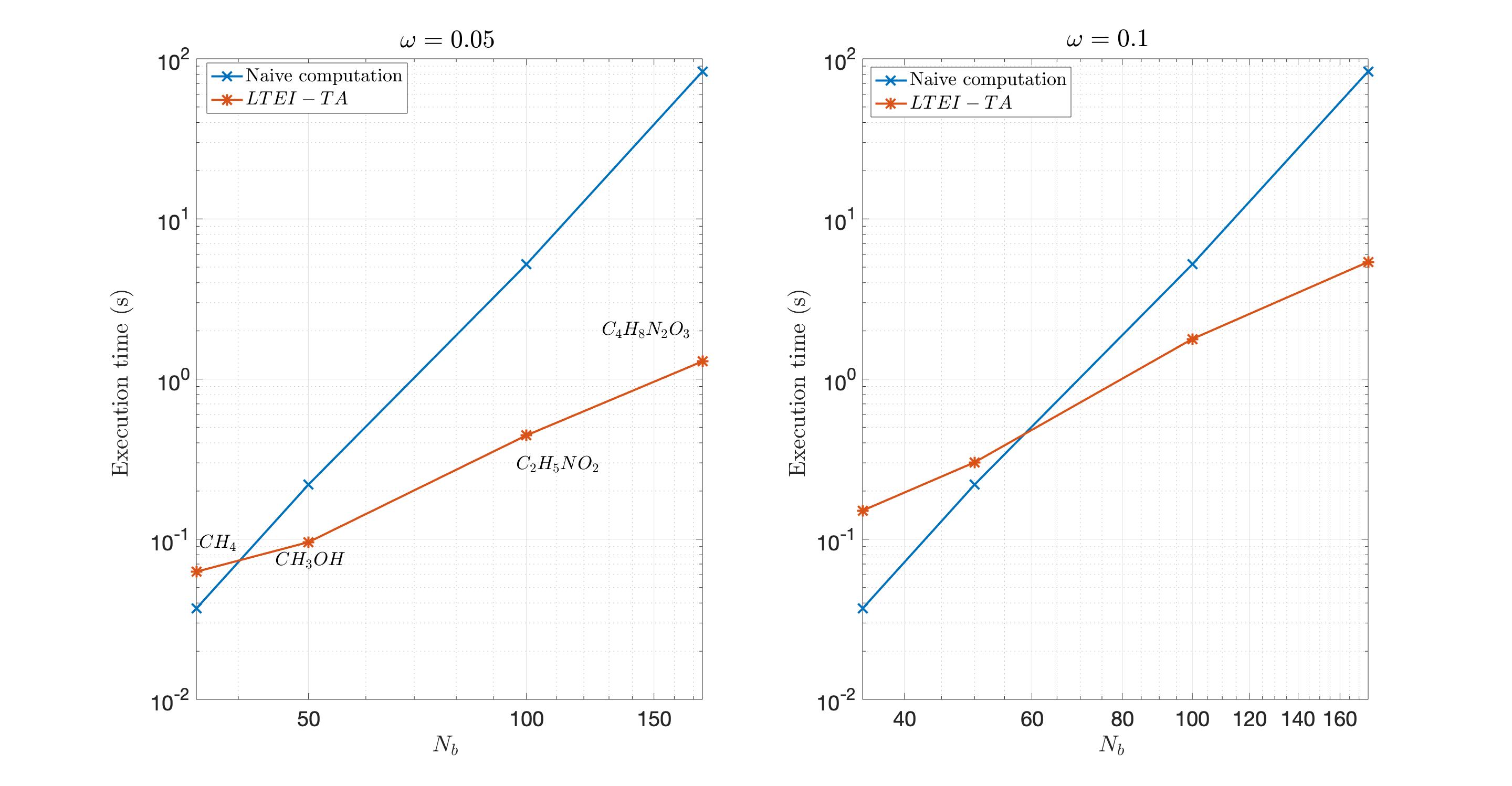}  
\caption{Execution time(s) required for the evaluation of  \eqref{Coul} using the TEI tensor $\mathbf{B}^{lr}$ for different values of $N_b$, for $\omega=0.05$ and $\omega=0.1$  with  imposed relative error smaller than  $1e^{-5}$.}
\label{fig::Exch}
\end{center}
\end{figure}

\bibliographystyle{model1-num-names}
\bibliography{refs}  

\begin{thebibliography}{43}
\expandafter\ifx\csname natexlab\endcsname\relax\def\natexlab#1{#1}\fi
\providecommand{\url}[1]{\texttt{#1}}
\providecommand{\href}[2]{#2}
\providecommand{\path}[1]{#1}
\providecommand{\DOIprefix}{doi:}
\providecommand{\ArXivprefix}{arXiv:}
\providecommand{\URLprefix}{URL: }
\providecommand{\Pubmedprefix}{pmid:}
\providecommand{\doi}[1]{\href{http://dx.doi.org/#1}{\path{#1}}}
\providecommand{\Pubmed}[1]{\href{pmid:#1}{\path{#1}}}
\providecommand{\bibinfo}[2]{#2}
\ifx\xfnm\relax \def\xfnm[#1]{\unskip,\space#1}\fi
\bibitem[{Khoromskaia and Khoromskij(2015)}]{box}
\bibinfo{author}{V.~Khoromskaia}, \bibinfo{author}{B.~N. Khoromskij},
\newblock \bibinfo{title}{{{Tensor numerical methods in quantum chemistry: from
  Hartree–Fock to excitation energies}}},
\newblock \bibinfo{journal}{Phys. Chem. Chem. Phys.} \bibinfo{volume}{17}
  (\bibinfo{year}{2015}) \bibinfo{pages}{31491--31509}.
\bibitem[{Cancès et~al.(2003)Cancès, Defranceschi, Kutzelnigg, {Le Bris}, and
  Maday}]{CANCES20033}
\bibinfo{author}{E.~Cancès}, \bibinfo{author}{M.~Defranceschi},
  \bibinfo{author}{W.~Kutzelnigg}, \bibinfo{author}{C.~{Le Bris}},
  \bibinfo{author}{Y.~Maday},
\newblock \bibinfo{title}{Computational quantum chemistry: A primer},
\newblock in: \bibinfo{booktitle}{Special Volume, Computational Chemistry},
  volume~\bibinfo{volume}{10} of \textit{\bibinfo{series}{Handbook of Numerical
  Analysis}}, \bibinfo{publisher}{Elsevier}, \bibinfo{year}{2003}, pp.
  \bibinfo{pages}{3--270}. \URLprefix
  \url{https://www.sciencedirect.com/science/article/pii/S1570865903100038}.
  \DOIprefix\doi{https://doi.org/10.1016/S1570-8659(03)10003-8}.
\bibitem[{Ashworth(2012)}]{hamiltonian}
\bibinfo{author}{S.~Ashworth},
\newblock \bibinfo{title}{Molecular quantum mechanics, 5th edn., by peter
  atkins and ronald friedman},
\newblock \bibinfo{journal}{Contemporary Physics - CONTEMP PHYS}
  \bibinfo{volume}{53} (\bibinfo{year}{2012}) \bibinfo{pages}{1--2}.
\bibitem[{Szabo et~al.(1996)Szabo, Attila, and SOstlund}]{quatummech}
\bibinfo{author}{Szabo}, \bibinfo{author}{Attila},
  \bibinfo{author}{N.~SOstlund},
\newblock \bibinfo{title}{Modern quantum chemistry : introduction to advanced
  electronic structure theory},
\newblock in: \bibinfo{booktitle}{Special Volume, Computational Chemistry},
  \bibinfo{publisher}{Mineola (N.Y.) : Dover publications},
  \bibinfo{year}{1996}, p. \bibinfo{pages}{481 / 481}. \URLprefix
  \url{http://lib.ugent.be/catalog/rug01:000906565}.
\bibitem[{Khoromskaia et~al.(2013)Khoromskaia, Khoromskij, and
  Schneider}]{khorom}
\bibinfo{author}{V.~Khoromskaia}, \bibinfo{author}{B.~N. Khoromskij},
  \bibinfo{author}{R.~Schneider},
\newblock \bibinfo{title}{Tensor-structured factorized calculation of
  two-electron integrals in a general basis},
\newblock \bibinfo{journal}{SIAM Journal on Scientific Computing}
  \bibinfo{volume}{35} (\bibinfo{year}{2013}) \bibinfo{pages}{A987--A1010}.
\bibitem[{Toulouse(2005)}]{theo1}
\bibinfo{author}{J.~Toulouse}, \bibinfo{title}{{Extension
  multid{\'e}terminantale de la m{\'e}thode de Kohn-Sham en th{\'e}orie de la
  fonctionnelle de la densit{\'e} par d{\'e}composition de l'interaction
  {\'e}lectronique en contributions de longue port{\'e}e et de courte
  port{\'e}e}}, \bibinfo{type}{Theses}, {Universit{\'e} Pierre et Marie Curie -
  Paris VI}, \bibinfo{year}{2005}. \URLprefix
  \url{https://tel.archives-ouvertes.fr/tel-00550772}.
\bibitem[{Savin(1996)}]{theo2}
\bibinfo{author}{A.~Savin}, \bibinfo{title}{On degeneracy, near-degeneracy and
  density functional theory}, volume~\bibinfo{volume}{4}, \bibinfo{year}{1996},
  pp. \bibinfo{pages}{327--357}. \DOIprefix\doi{10.1016/S1380-7323(96)80091-4}.
\bibitem[{Toulouse et~al.(2004)Toulouse, Colonna, and Savin}]{separation}
\bibinfo{author}{J.~Toulouse}, \bibinfo{author}{F.~m.~c. Colonna},
  \bibinfo{author}{A.~Savin},
\newblock \bibinfo{title}{Long-range--short-range separation of the
  electron-electron interaction in density-functional theory},
\newblock \bibinfo{journal}{Phys. Rev. A} \bibinfo{volume}{70}
  (\bibinfo{year}{2004}) \bibinfo{pages}{062505}.
\bibitem[{Giner(2021)}]{Sepemma}
\bibinfo{author}{E.~Giner},
\newblock \bibinfo{title}{A new form of transcorrelated hamiltonian inspired by
  range-separated dft},
\newblock \bibinfo{journal}{The Journal of Chemical Physics}
  \bibinfo{volume}{154} (\bibinfo{year}{2021}) \bibinfo{pages}{084119}.
\bibitem[{Savin(2020)}]{Sepsavin}
\bibinfo{author}{A.~Savin},
\newblock \bibinfo{title}{Models and corrections: Range separation for
  electronic interaction—lessons from density functional theory},
\newblock \bibinfo{journal}{The Journal of Chemical Physics}
  \bibinfo{volume}{153} (\bibinfo{year}{2020}) \bibinfo{pages}{160901}.
\bibitem[{Toulouse et~al.(2005)Toulouse, Gori-Giorgi, and Savin}]{short}
\bibinfo{author}{J.~Toulouse}, \bibinfo{author}{P.~Gori-Giorgi},
  \bibinfo{author}{A.~Savin},
\newblock \bibinfo{title}{{A short-range correlation energy density functional
  with multi-determinantal reference}},
\newblock \bibinfo{journal}{{Theoretical Chemistry Accounts: Theory,
  Computation, and Modeling}} \bibinfo{volume}{114} (\bibinfo{year}{2005})
  \bibinfo{pages}{305}.
\bibitem[{Lee et~al.(1997)Lee, Taylor, Dombroski, and Gill}]{FMM}
\bibinfo{author}{A.~Lee}, \bibinfo{author}{S.~Taylor},
  \bibinfo{author}{J.~Dombroski}, \bibinfo{author}{P.~Gill},
\newblock \bibinfo{title}{Optimal partition of the coulomb operator},
\newblock \bibinfo{journal}{Physical Review A - PHYS REV A}
  \bibinfo{volume}{55} (\bibinfo{year}{1997}) \bibinfo{pages}{3233--3235}.
\bibitem[{Ferté(2018)}]{portee}
\bibinfo{author}{A.~Ferté}, \bibinfo{title}{Théorie de la fonctionnelle de la
  densité avec une fonction d’onde multiréférence : Développement
  d’approximations pour la fonctionnelle de corrélation à courte portée
  utilisant la densité de paires à coalescence}, \bibinfo{year}{2018}.
  \bibinfo{note}{Unpublished}.
\bibitem[{{Lecours, Michael}(2021)}]{Lecours}
\bibinfo{author}{{Lecours, Michael}}, \bibinfo{title}{Compact Sparse Coulomb
  Integrals using a Range-Separated Potential}, Ph.D. thesis, University of
  Waterloo, \bibinfo{year}{2021}. \URLprefix
  \url{http://hdl.handle.net/10012/17516}.
\bibitem[{Limpanuparb et~al.(2013)Limpanuparb, Milthorpe, Rendell, and
  Gill}]{Resolutions}
\bibinfo{author}{T.~Limpanuparb}, \bibinfo{author}{J.~Milthorpe},
  \bibinfo{author}{A.~Rendell}, \bibinfo{author}{P.~Gill},
\newblock \bibinfo{title}{Resolutions of the coulomb operator: Vii. evaluation
  of long-range coulomb and exchange matrices},
\newblock \bibinfo{journal}{Journal of Chemical Theory and Computation}
  \bibinfo{volume}{9} (\bibinfo{year}{2013}) \bibinfo{pages}{863--867}.
\bibitem[{Simmonett et~al.(2022)Simmonett, Brooks, and Darden}]{unknown}
\bibinfo{author}{A.~Simmonett}, \bibinfo{author}{B.~Brooks},
  \bibinfo{author}{T.~Darden}, \bibinfo{title}{Efficient and scalable
  electrostatics via spherical grids and treecode summation},
  \bibinfo{year}{2022}. \DOIprefix\doi{10.26434/chemrxiv-2022-6xzql},
  \bibinfo{note}{unpublished}.
\bibitem[{Demel et~al.(2021)Demel, Lecours, Habrovský, and
  Nooijen}]{PMID:34686052}
\bibinfo{author}{O.~Demel}, \bibinfo{author}{M.~J. Lecours},
  \bibinfo{author}{R.~Habrovský}, \bibinfo{author}{M.~Nooijen},
\newblock \bibinfo{title}{Toward laplace mp2 method using range separated
  coulomb potential and orbital selective virtuals},
\newblock \bibinfo{journal}{The Journal of chemical physics}
  \bibinfo{volume}{155} (\bibinfo{year}{2021}) \bibinfo{pages}{154104}.
\bibitem[{Knowino(2010)}]{set}
\bibinfo{author}{Knowino}, \bibinfo{title}{Gaussian type orbitals ---
  knowino{,} an encyclopedia}, \bibinfo{year}{2010}. \URLprefix
  \url{http://knowino.org/w/index.php?title=Gaussian_type_orbitals&oldid=3278}.
\bibitem[{Pritchard et~al.(2019)Pritchard, Altarawy, Didier, Gibson, and
  Windus}]{basis}
\bibinfo{author}{B.~P. Pritchard}, \bibinfo{author}{D.~Altarawy},
  \bibinfo{author}{B.~Didier}, \bibinfo{author}{T.~D. Gibson},
  \bibinfo{author}{T.~L. Windus},
\newblock \bibinfo{title}{New basis set exchange: An open, up-to-date resource
  for the molecular sciences community},
\newblock \bibinfo{journal}{Journal of Chemical Information and Modeling}
  \bibinfo{volume}{59} (\bibinfo{year}{2019}) \bibinfo{pages}{4814--4820}.
  \bibinfo{note}{PMID: 31600445}.
\bibitem[{Fong and Darve(2009)}]{kernel}
\bibinfo{author}{W.~Fong}, \bibinfo{author}{E.~Darve},
\newblock \bibinfo{title}{The black-box fast multipole method},
\newblock \bibinfo{journal}{Journal of Computational Physics}
  \bibinfo{volume}{228} (\bibinfo{year}{2009}) \bibinfo{pages}{8712--8725}.
\bibitem[{Greengard and Rokhlin(1987)}]{GREENGARD1987325}
\bibinfo{author}{L.~Greengard}, \bibinfo{author}{V.~Rokhlin},
\newblock \bibinfo{title}{A fast algorithm for particle simulations},
\newblock \bibinfo{journal}{Journal of Computational Physics}
  \bibinfo{volume}{73} (\bibinfo{year}{1987}) \bibinfo{pages}{325--348}.
\bibitem[{Garniron et~al.(2019)Garniron, Applencourt, Gasperich, Benali,
  Ferté, Paquier, Pradines, Assaraf, Reinhardt, Toulouse, Barbaresco, Renon,
  David, Malrieu, Véril, Caffarel, Loos, Giner, and Scemama}]{quan_pack}
\bibinfo{author}{Y.~Garniron}, \bibinfo{author}{T.~Applencourt},
  \bibinfo{author}{K.~Gasperich}, \bibinfo{author}{A.~Benali},
  \bibinfo{author}{A.~Ferté}, \bibinfo{author}{J.~Paquier},
  \bibinfo{author}{B.~Pradines}, \bibinfo{author}{R.~Assaraf},
  \bibinfo{author}{P.~Reinhardt}, \bibinfo{author}{J.~Toulouse},
  \bibinfo{author}{P.~Barbaresco}, \bibinfo{author}{N.~Renon},
  \bibinfo{author}{G.~David}, \bibinfo{author}{J.-P. Malrieu},
  \bibinfo{author}{M.~Véril}, \bibinfo{author}{M.~Caffarel},
  \bibinfo{author}{P.-F. Loos}, \bibinfo{author}{E.~Giner},
  \bibinfo{author}{A.~Scemama},
\newblock \bibinfo{title}{Quantum package 2.0: An open-source
  determinant-driven suite of programs},
\newblock \bibinfo{journal}{Journal of Chemical Theory and Computation}
  \bibinfo{volume}{15} (\bibinfo{year}{2019}) \bibinfo{pages}{3591--3609}.
  \bibinfo{note}{PMID: 31082265}.
\bibitem[{Scheiber(2015)}]{scheiber2015chebyshev}
\bibinfo{author}{E.~Scheiber}, \bibinfo{title}{On the chebyshev approximation
  of a function with two variables}, \bibinfo{year}{2015}.
\bibitem[{Townsend and Trefethen(2013)}]{chebfun2}
\bibinfo{author}{A.~Townsend}, \bibinfo{author}{L.~N. Trefethen},
\newblock \bibinfo{title}{An extension of chebfun to two dimensions},
\newblock \bibinfo{journal}{SIAM Journal on Scientific Computing}
  \bibinfo{volume}{35} (\bibinfo{year}{2013}) \bibinfo{pages}{C495--C518}.
\bibitem[{JafariBehbahani and Roodaki(2015)}]{err2}
\bibinfo{author}{Z.~JafariBehbahani}, \bibinfo{author}{M.~Roodaki},
\newblock \bibinfo{title}{Two-dimensional chebyshev hybrid functions and their
  applications to integral equations},
\newblock \bibinfo{journal}{Beni-Suef University Journal of Basic and Applied
  Sciences} \bibinfo{volume}{4} (\bibinfo{year}{2015})
  \bibinfo{pages}{134--141}.
\bibitem[{Gupta(1991)}]{err1}
\bibinfo{author}{M.~Gupta},
\newblock \bibinfo{title}{Numerical methods and software (david kahaner, cleve
  moler, and stephen nash)},
\newblock \bibinfo{journal}{Siam Review - SIAM REV} \bibinfo{volume}{33}
  (\bibinfo{year}{1991}).
\bibitem[{Liu and TRENKLER(2008)}]{kroni}
\bibinfo{author}{S.~Liu}, \bibinfo{author}{O.~TRENKLER},
\newblock \bibinfo{title}{Hadamard, khatri-rao, kronecker and other matrix
  products},
\newblock \bibinfo{journal}{International Journal of Information , Systems
  Sciences} \bibinfo{volume}{4} (\bibinfo{year}{2008}).
\bibitem[{Kolda and Bader(2009)}]{doi:10.1137/07070111X}
\bibinfo{author}{T.~G. Kolda}, \bibinfo{author}{B.~W. Bader},
\newblock \bibinfo{title}{Tensor decompositions and applications},
\newblock \bibinfo{journal}{SIAM Review} \bibinfo{volume}{51}
  (\bibinfo{year}{2009}) \bibinfo{pages}{455--500}.
\bibitem[{Platte and Trefethen(2010)}]{Platte2010}
\bibinfo{author}{R.~B. Platte}, \bibinfo{author}{L.~N. Trefethen},
  \bibinfo{title}{Chebfun: A New Kind of Numerical Computing},
  \bibinfo{publisher}{Springer Berlin Heidelberg}, \bibinfo{address}{Berlin,
  Heidelberg}, \bibinfo{year}{2010}, pp. \bibinfo{pages}{69--87}. \URLprefix
  \url{https://doi.org/10.1007/978-3-642-12110-4_5}.
  \DOIprefix\doi{10.1007/978-3-642-12110-4_5}.
\bibitem[{Frigo and Johnson(2005)}]{FFTW}
\bibinfo{author}{M.~Frigo}, \bibinfo{author}{S.~G. Johnson},
\newblock \bibinfo{title}{The design and implementation of {FFTW3}},
\newblock \bibinfo{journal}{Proceedings of the IEEE} \bibinfo{volume}{93}
  (\bibinfo{year}{2005}) \bibinfo{pages}{216--231}. \bibinfo{note}{Special
  issue on ``Program Generation, Optimization, and Platform Adaptation''}.
\bibitem[{Dongarra et~al.(1990)Dongarra, Croz, Hammarling, and Duff}]{blas}
\bibinfo{author}{J.~Dongarra}, \bibinfo{author}{J.~Croz},
  \bibinfo{author}{S.~Hammarling}, \bibinfo{author}{I.~Duff},
\newblock \bibinfo{title}{A set of level 3 basic linear algebra subprograms},
\newblock \bibinfo{journal}{ACM Transactions on Mathematical Software}
  \bibinfo{volume}{16} (\bibinfo{year}{1990}) \bibinfo{pages}{1--17}.
\bibitem[{Chollet(2021)}]{igor}
\bibinfo{author}{I.~Chollet}, \bibinfo{title}{{Symmetries and Fast Multipole
  Methods for Oscillatory Kernels}}, \bibinfo{type}{Theses}, {Sorbonne
  Universit{\'e}}, \bibinfo{year}{2021}. \URLprefix
  \url{https://tel.archives-ouvertes.fr/tel-03203231}.
\bibitem[{{Barnes} and {Hut}(1986)}]{Barnes}
\bibinfo{author}{J.~{Barnes}}, \bibinfo{author}{P.~{Hut}},
\newblock \bibinfo{title}{{A hierarchical O(N log N) force-calculation
  algorithm}},
\newblock \bibinfo{journal}{Nature} \bibinfo{volume}{324}
  (\bibinfo{year}{1986}) \bibinfo{pages}{446--449}.
\bibitem[{Chaillat et~al.(2017)Chaillat, Desiderio, and
  Ciarlet}]{chaillat:hal-01543919}
\bibinfo{author}{S.~Chaillat}, \bibinfo{author}{L.~Desiderio},
  \bibinfo{author}{P.~Ciarlet},
\newblock \bibinfo{title}{{Theory and implementation of $\mathcal{H}$-matrix
  based iterative and direct solvers for Helmholtz and elastodynamic
  oscillatory kernels}},
\newblock \bibinfo{journal}{{Journal of Computational Physics}}
  (\bibinfo{year}{2017}).
\bibitem[{Bebendorf(2008)}]{Hierarchical}
\bibinfo{author}{M.~Bebendorf},
\newblock \bibinfo{title}{Hierarchical matrices},
\newblock \bibinfo{journal}{Lecture notes in computational science and
  engineering, v.63 (2008)} \bibinfo{volume}{63} (\bibinfo{year}{2008}).
\bibitem[{Hackbusch(2015)}]{Wolfgang}
\bibinfo{author}{W.~Hackbusch}, \bibinfo{title}{Hierarchical Matrices:
  Algorithms and Analysis}, volume~\bibinfo{volume}{49}, \bibinfo{year}{2015}.
  \DOIprefix\doi{10.1007/978-3-662-47324-5}.
\bibitem[{Losilla et~al.(2015)Losilla, Watson, Aspuru-Guzik, and
  Sundholm}]{doi:10.1021/ct501128u}
\bibinfo{author}{S.~A. Losilla}, \bibinfo{author}{M.~A. Watson},
  \bibinfo{author}{A.~Aspuru-Guzik}, \bibinfo{author}{D.~Sundholm},
\newblock \bibinfo{title}{Construction of the fock matrix on a grid-based
  molecular orbital basis using gpgpus},
\newblock \bibinfo{journal}{Journal of Chemical Theory and Computation}
  \bibinfo{volume}{11} (\bibinfo{year}{2015}) \bibinfo{pages}{2053--2062}.
  \bibinfo{note}{PMID: 26574409}.
\bibitem[{Xing and Chow(2020)}]{doi:10.1137/19M1252855}
\bibinfo{author}{X.~Xing}, \bibinfo{author}{E.~Chow},
\newblock \bibinfo{title}{Fast coulomb matrix construction via compressing the
  interactions between continuous charge distributions},
\newblock \bibinfo{journal}{SIAM Journal on Scientific Computing}
  \bibinfo{volume}{42} (\bibinfo{year}{2020}) \bibinfo{pages}{A162--A186}.
\bibitem[{Rosal~Sandberg(2014)}]{meth}
\bibinfo{author}{J.~A. Rosal~Sandberg}, \bibinfo{title}{New efficient integral
  algorithms for quantum chemistry}, Ph.D. thesis, KTH, Theoretical Chemistry
  and Biology, \bibinfo{year}{2014}. \bibinfo{note}{QC 20140826}.
\bibitem[{Hansen(1987)}]{chaillat}
\bibinfo{author}{P.~C. Hansen},
\newblock \bibinfo{title}{The truncatedsvd as a method for regularization},
\newblock \bibinfo{journal}{BIT Numerical Mathematics} \bibinfo{volume}{27}
  (\bibinfo{year}{1987}) \bibinfo{pages}{534--553}.
\bibitem[{Chollet et~al.(2022)Chollet, Claeys, Fortin, and
  Grigori}]{chollet:hal-03563005}
\bibinfo{author}{I.~Chollet}, \bibinfo{author}{X.~Claeys},
  \bibinfo{author}{P.~Fortin}, \bibinfo{author}{L.~Grigori}, \bibinfo{title}{{A
  Directional Equispaced interpolation-based Fast Multipole Method for
  oscillatory kernels}}, \bibinfo{year}{2022}. \URLprefix
  \url{https://hal.archives-ouvertes.fr/hal-03563005}, \bibinfo{note}{working
  paper or preprint}.
\bibitem[{Lu~B(2007)}]{LU20071348}
\bibinfo{author}{M.~J. Lu~B, Cheng~X},
\newblock \bibinfo{title}{New-version-fast-multipole-method" accelerated
  electrostatic interactions in biomolecular systems},
\newblock \bibinfo{journal}{J Comput Phys.}  (\bibinfo{year}{2007}).
\bibitem[{Hedeg{\aa}rd et~al.(2015)Hedeg{\aa}rd, Knecht, Kielberg, Jensen, and
  Reiher}]{Hedeg_rd_2015}
\bibinfo{author}{E.~D. Hedeg{\aa}rd}, \bibinfo{author}{S.~Knecht},
  \bibinfo{author}{J.~S. Kielberg}, \bibinfo{author}{H.~J.~A. Jensen},
  \bibinfo{author}{M.~Reiher},
\newblock \bibinfo{title}{Density matrix renormalization group with efficient
  dynamical electron correlation through range separation},
\newblock \bibinfo{journal}{The Journal of Chemical Physics}
  \bibinfo{volume}{142} (\bibinfo{year}{2015}) \bibinfo{pages}{224108}.

\end{thebibliography}






\end{document}